\documentclass[10pt]{amsart}

\usepackage[lite, initials]{amsrefs}
\usepackage{amssymb}
\usepackage{amsmath, amsfonts}
\usepackage{amsthm, array}

\usepackage{tikz}
\usepackage{pgfplots}
\usetikzlibrary{calc, arrows, through, positioning, decorations.markings, shapes, fit, decorations.shapes, matrix, plotmarks}
\RequirePackage{latexsym, verbatim, xspace, setspace}
\pgfplotsset{compat=newest}

\usepackage{graphics}
\usepackage{etex}
\usepackage{hyperref}
\usepackage{yhmath}
\usepackage{moresize}
\usepackage{float}
\usepackage{enumitem}
\usepackage[margin=1.25in]{geometry}
\usepackage{adjustbox}
\usepackage{booktabs, tabularx, multirow}
\usepackage[flushleft]{threeparttable}
\usepackage[fixamsmath]{mathtools}
\mathtoolsset{mathic}
\usepackage{mathdots}
\usepackage[justification=centering, labelfont=bf, width=\textwidth]{caption}
\usepackage[justification=centering, labelfont=bf]{subcaption}
\usepackage{multicol}
\usepackage{color}
\usepackage{xparse}
\usepackage{wasysym}
\usepackage{url}
\usepackage{mathrsfs}
\usepackage{bigstrut}
\usepackage{cite}

\allowdisplaybreaks[1]

\theoremstyle{plain}
\newtheorem{theorem}{Theorem}
\setenumerate{label={
	\normalfont \bfseries (\alph*)}}

\setenumerate{label={
	\normalfont \bfseries (\alph*)}}
\newtheorem{lemma}{Lemma}
\setenumerate{label={
	\normalfont \bfseries (\alph*)}}

\setenumerate{label={
	\normalfont \bfseries (\alph*)}}
\newtheorem{conjecture}{Conjecture}
\setenumerate{label={
	\normalfont \bfseries (\alph*)}}

\newtheorem*{assumption}{Assumption}

\theoremstyle{definition}

\newtheorem*{remark}{Remark}

\theoremstyle{remark}

\renewcommand{\pod}[1]{\allowbreak\mathchoice
  {\if@display \mkern 18mu\else \mkern 8mu\fi (#1)}
  {\if@display \mkern 18mu\else \mkern 8mu\fi (#1)}
  {\mkern4mu(#1)}
  {\mkern4mu(#1)}
}
\makeatletter
\newcases{lrdcases}
  {\quad}
  {$\m@th\displaystyle{##}$\hfil}
  {$\m@th\displaystyle{##}$\hfil}
  {\lbrace}
  {\rbrace}
\makeatother
\newcommand*{\textstack}[1]{%
  \text{%
    \renewcommand*{\arraystretch}{.9}%
    \begin{tabular}{@{}c@{}}%
      #1%
    \end{tabular}%
  }%
}

\DeclareMathOperator{\card}{card}
\DeclareMathOperator{\ch}{ch}
\DeclareMathOperator{\BZL}{BZL}
\DeclareMathOperator{\diag}{diag}
\DeclareMathOperator{\End}{End}
\DeclareMathOperator{\gl}{\mathfrak{gl}}
\DeclareMathOperator{\GL}{GL}
\DeclareMathOperator{\GT}{GT}
\DeclareMathOperator{\HH}{H}
\DeclareMathOperator{\ord}{ord}
\DeclareMathOperator{\SL}{SL}
\DeclareMathOperator{\SO}{SO}
\DeclareMathOperator{\Sp}{Sp}
\DeclareMathOperator{\ssp}{sp}
\DeclareMathOperator{\fraksp}{\mathfrak{sp}}
\DeclareMathOperator{\Span}{span}
\DeclareMathOperator{\supp}{supp}
\DeclareMathOperator{\wt}{wt}
\DeclareMathOperator{\indeg}{indeg}
\DeclareMathOperator{\outdeg}{outdeg}
\DeclareMathOperator{\Ind}{Ind}

\newcolumntype{Y}{>{\centering\arraybackslash}X}

\newcommand{\dspin}[2]{{ #1 \atop #2}}

\makeatletter
\def\th@plain{%
  \thm@notefont{}
  \itshape 
}
\def\th@definition{%
  \thm@notefont{}
  \normalfont 
}
\makeatother

\begin{document}
\title[Metaplectic Ice for Cartan Type C]{Metaplectic Ice for Cartan Type C}
\author[Nathan Gray]{Nathan Gray}
\address{Department of Mathematics, Mount Holyoke College, South Hadley, MA 01075}
\email{ngray@mtholyoke.edu}

\maketitle

%
%
\begin{abstract}
We use techniques from statistical mechanics to provide new formulas for Whittaker coefficients of metaplectic Eisenstein series on odd orthogonal groups, matching~\cite{FZ}. We study a particular variation/generalization of the six-vertex model of Cartan type C having ``domain-wall boundary conditions'' dependent on a given integer partition $\lambda$ of length at most $r$, where $r$ is a fixed positive integer. More precisely, we examine a planar, non-nested, U-turn model whose partition functions $Z_{\lambda}$ are a generalization of a deformation of characters of the symplectic group $\Sp(2r, \mathbb{C})$. Special cases appeared in Kuperberg~\cite{Kuperberg2} and then in Hamel and King~\cite{HK3}, Brubaker, Bump, Chinta, and Gunnells~\cite{BBCG}, and Ivanov~\cite{Ivanov}.

Our main result is that these new families of ``metaplectic'' models are solvable---i.e., they possess Yang--Baxter equations. We use this to derive two types of functional equations involving $Z_{\lambda}$ corresponding to the two root lengths for simple reflections of the symplectic Weyl group. It is widely believed that the local component of metaplectic Eisenstein series is a metaplectic Whittaker function, though this is subtle owing to the lack of uniqueness of Whittaker models and only verified in type A~\cite{McNamara}. Thus, we also give evidence for the conjecture that $Z_{\lambda}$ is a spherical Whittaker function by showing that $Z_{\lambda}$ satisfies the same identities under our solution to the Yang--Baxter equation as the metaplectic Whittaker function under intertwining operators on the unramified principal series of an $n$-fold metaplectic cover of $\SO(2r + 1)$, for $n$ odd.
\end{abstract}

%
%
 \tableofcontents

\section{Introduction}

Exact solutions of statistical-mechanical models on planar lattices were explored by Baxter~\cite{Baxter}. As we will describe later in this paper, an ``exactly solvable'' model is one where an explicit generating function on states of the model---referred to as a ``partition function'' of the model---may be computed in closed form. Baxter's techniques (particularly the use of the so-called Yang--Baxter equation) later made their way to combinatorics in Kuperberg's proof of the alternating-sign matrix conjecture~\cite{Kuperberg1}. His proof made use of one such model: the six-vertex model (or ``square ice''). Later, Kuperberg~\cite{Kuperberg2} extended these techniques to symmetry classes of alternating-sign matrices using planar lattices with a variety of interesting configurations. The subject of this paper is an exactly solvable model on a generalization of one such lattice in~\cite{Kuperberg2}.

To provide context for our results, we begin with some history. Tokuyama~\cite{Tokuyama} found a generating-function identity that simultaneously deformed the Weyl character formula and the combinatorial generating function for highest-weight characters of $\GL(r,\mathbb{C})$. His generating function was initially expressed as a sum over (shifted) strict Gelfand--Tsetlin patterns, but was later given by Hamel and King~\cite{HK1, HK2, HK3} as a ``partition function'' of a six-vertex model on a rectangular lattice. Tokuyama's deformation matches precisely the output of Shintani~\cite{Shintani} and Casselman--Shalika~\cite{CS} for the spherical Whittaker function on $\GL(r,F)$ ($F$ a nonarchimedean local field) evaluated at a dominant integral element. The spherical Whittaker function is a complex-valued function defined on $\GL(r,F)$, and its evaluation at elements of a maximal torus of $\GL(r,F)$ is a critical ingredient in many aspects of automorphic forms and representation theory; it is described in greater detail in Section~\ref{sec:Algebraic_Preliminaries}. In summary, combining this string of equalities, one obtains a description of the spherical Whittaker function as a partition function of a lattice model. This string of equalities is rather ad-hoc, and several immediate questions arise:
\begin{enumerate}[label=\textbf{\arabic*.}]
\item Does there exist a more direct proof of such an identity?
\item To what extent do these formulas generalize to other reductive groups (and their arithmetic covers)?
\item What do we learn about the Whittaker function as a result of these connections?
\end{enumerate}

Satisfactory answers to these questions have recently been given in Cartan type A, and this paper answers some of them in Cartan type C. Let us begin with the known answers to these questions in type A.

Brubaker, Bump, and Friedberg~\cite{BBF2} gave a partial answer to Question 1 by providing a statistical-mechanical proof of the results above of Tokuyama and Hamel--King. In particular, they demonstrated a family of Yang--Baxter equations for the underlying models. Very recently, Brubaker, Buciumas, and Bump~\cite{BBB} found a Yang--Baxter equation for a statistical-mechanical model for metaplectic Whittaker functions. A result of this paper is a type C analogue of the results of~\cite{BBB}.

In type C, a deformed Weyl character formula was proved by Hamel and King~\cite{HK1}, and a statistical-mechanical proof in the spirit of~\cite{BBF2} was given by Ivanov~\cite{Ivanov}.

In this paper, we use similar techniques to present formulas for nonarchimedean metaplectic Whittaker functions, arising in the local theory of automorphic forms. We study particular variations/generalizations of the six-vertex model of type C having ``domain-wall boundary conditions'' dependent on a given integer partition $\lambda$ of length at most $r$; these will be described in detail in Section~\ref{sec:Partition_Boltzmann_Charge}. More precisely, we examine a planar, non-nested, U-turn model whose partition function is related to characters of the symplectic group $\Sp (2r, \mathbb{C})$. We refer to this model as symplectic ice so that the underlying group is clear. (Technically, we should refer to it as \textit{metaplectic} ice.) Special cases appeared in Kuperberg~\cite{Kuperberg2} and then in Hamel and King~\cite{HK3}, Brubaker, Bump, Chinta, and Gunnells~\cite{BBCG}, and Ivanov~\cite{Ivanov}.

A recent application was shown in Brubaker, Bump, Buciumas, and Gray~\cite{BBBGray}, in which values of spherical Whittaker functions on an $n$-fold metaplectic cover of $\GL(r,F)$ ($F$ a nonarchimedean local field) were interpreted as partition functions of two statistical-mechanical models that differed in their Boltzmann weights. Equality of the partition functions was shown by using the commutativity of transfer matrices associated to the two models by the Yang--Baxter equation proved in Section~\ref{sec:Yang_Baxter_equation} of the current paper.

The solution to the Yang--Baxter equation was unknown when the proof of Statement B in~\cite{BBF_annals} was written, a lengthy proof written in terms of crystal bases. The existence of a solution to the Yang--Baxter equation is an important result in the current paper, reducing much of the work done in~\cite{BBF_annals} to about a dozen pages. We also use the Yang--Baxter equation to prove local functional equations for the partition function of our statistical-mechanical model, much like the functional equations appearing in Kuperberg~\cite{Kuperberg2}.

We now describe the contents of this paper. In Section~\ref{sec:Partition_Boltzmann_Charge}, we will describe our symplectic-ice model and define some key terms. We define the partition function $Z$ as a certain weighted sum over all the allowable ``admissible states'' of our model. We will also mention and describe an integer-valued global statistic, called charge, that distinguishes our model from the usual six-vertex model of type C. Finally, we list the Boltzmann weights of the vertices in our model, weights that depend on charge.

In Section~\ref{sec:Yang_Baxter_equation}, we state and prove our main tool: a Yang--Baxter equation for metaplectic Boltzmann weights. We introduce an alternative viewpoint on charge that we call a decoration, which is needed in order to make our Boltzmann weights local (i.e., depending only on nearest-neighbor interactions). We then prove Theorem~\ref{thm:metaplecticYBE}, showing that a solution to the Yang--Baxter equation exists.

In Section~\ref{sec:Connections_Metaplectic}, we mention the fact that admissible states of our model are in bijection with symplectic patterns (Gelfand--Tsetlin patterns of type C). The remainder of the section goes through the proof of Theorem~\ref{T:sym_pattern_bij}, which relates certain admissible states of our model to the $p$-parts of a multiple Dirichlet series.

In Section~\ref{sec:Function_Eqns_Partition}, we prove Theorems~\ref{T:transposition} and~\ref{thm:inverse}, giving two functional equations that involve the partition function $Z$. (These functional equations are used in Section~\ref{sec:Connections_intertwining_Whittaker}.) One functional equation describes the result of interchanging any two adjacent rows $i$ and $i+1$ in the model; this amounts to interchanging the spectral parameters $z_{i}$ and $z_{i+1}$ in $Z$ (here $z_i$ and $z_{i+1}$ are factors of our Boltzmann weights), so represents a functional equation for the partition function under the action by a short simple root. The other equation describes the result of interchanging $z_r$ and $z_{r}^{-1}$ and reflects an action by the long simple root. A peculiar feature of the proofs for Theorems~\ref{T:transposition} and~\ref{thm:inverse} is their dependence on certain lemmas that we have called the caduceus and fish relations, after~\cite{BBCFG}.

In Section~\ref{sec:Algebraic_Preliminaries}, we give a brief account of metaplectic groups and their Whittaker functions, with pointers to the relevant literature for each needed result.

In Section~\ref{sec:Connections_intertwining_Whittaker}, we demonstrate that the partition functions of symplectic ice could satisfy the same identities under our solution to the Yang--Baxter equation as the metaplectic Whittaker function under intertwining operators on unramified principal series. Being able to show these identities by using an algorithm given in~\cite{McNamara} is a topic for future exploration.

In Section~\ref{sec:further_questions}, we give a brief discussion on a few questions related to some of the topics in this paper, questions that can support future work.

Thanks to Benjamin Brubaker for his patience, encouragement, and guidance through this paper and other work. 

\section{The Partition Function $Z$}
\label{sec:Partition_Boltzmann_Charge}

Symplectic ice is a collection of digraphs, each arranged on a rectangular lattice and having exterior edges and interior edges, with ``bends'' connecting adjacent rows; see Figure~\ref{fig:Example_state}. Each exterior edge is incident to one vertex if the edge is not part of a bend. We assign each edge a sign of either $+$ or $-$, called the \textbf{spin} of the edge. The spins of the exterior edges along the top, left, and bottom boundaries are referred to as the model's boundary conditions. The boundary conditions are fixed as part of the model's data. Each interior edge is incident to two vertices. Assigning spins to all interior edges of the model yields a digraph called a \textbf{state} of the model.

\begin{figure}[H]
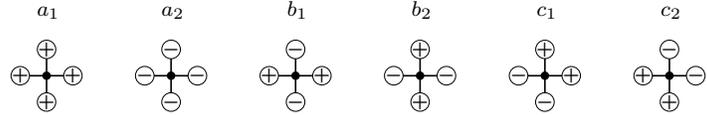

\makebox[\textwidth][c]{%
%
}
\caption{The six admissible configurations.}
\label{fig:configs}
\end{figure}

A state is \textbf{admissible} if for each vertex $v$ in the state, one of the following is true: if $v$ is in a bend, the two adjacent spins differ; else, $v$ and its four adjacent spins match one of the configurations in Figure~\ref{fig:configs}, called the \textbf{admissible configurations} and denoted by $a_1$, $a_2$, $b_1$, $b_2$, $c_1$, $c_2$. For example, the state in Figure~\ref{fig:Example_state} is admissible.

\begin{figure}[h]
\makebox[\textwidth][c]{%
%
%
}
\caption{}
\label{fig:Example_state}
\end{figure}

Fix a positive integer $r$, and let $\rho = (r, r-1, \ldots, 1)$. Let $\lambda = (\lambda_1, \lambda_2, \ldots, \lambda_r)$ be an integer partition. Each digraph of symplectic ice is arranged as follows (see Figure~\ref{fig:Symplectic}):
\begin{itemize}
\item \textsc{Rows/Columns:} There are $2r$ rows and $\lambda_1 + r$ columns arranged in a rectangular lattice. The columns are numbered $1$,~$2$,~\ldots,~$\lambda_1 + r$ from right to left. The rows are numbered $1$, $\overline{1}$, $2$, $\overline{2}$,~\ldots,~$r$,~$\overline{r}$ from top to bottom. We call $1$,~$2$,~\ldots,~$r$ the spectral indices. A vertex is at the intersection of each column and each row. 
\item \textsc{Bends:} For every $i \in \{ 1, \ldots, r \}$, there is a ``bend'' at the right boundary that consists of one vertex and two adjacent edges that connect rows $(i, \overline{i} \,)$. 
\item \textsc{Boundary Conditions:} For every state, all spins along the left/bottom boundaries are $+$. All spins along the top boundary are $-$ if they occur in the columns numbered by the parts of $\lambda+\rho$; else, they are $+$.
\end{itemize}
Thus, $\lambda = (2,1,1)$ corresponds to the top boundary conditions in Figure~\ref{fig:Example_state}. We will denote by $\mathfrak{S}_{\lambda}$ the set of all admissible states of our model, whose top boundary condition is determined by $\lambda$.

\begin{figure}[h]
\makebox[\textwidth][c]{%
%
}
\caption{}
\label{fig:Symplectic}
\end{figure}

Our main concern will be the study of the partition function of our model---a generating function equal to a weighted sum over the set of all admissible states. For every state $\mathfrak{s}$, we assign to each vertex $v$ in $\mathfrak{s}$ a weight, called the \textbf{Boltzmann weight} of $v$ and denoted by $\wt(v)$, where the weight depends on the spins of the edges adjacent to $v$. The \textbf{Boltzmann weight} of the state $\mathfrak{s}$, denoted by $\wt(\mathfrak{s})$, is the product of the Boltzmann weights of all vertices in $\mathfrak{s}$.

The \textbf{partition function} $Z(\mathfrak{S}_{\lambda})$, or simply $Z$, is the sum of the Boltzmann weights of all states. For every non-bend vertex, its weight is zero if its adjacent edges are not one of the six admissible configurations; so $Z$ can be defined as the sum of the Boltzmann weights of all admissible states. Otherwise, the weight of every non-bend vertex is taken from two distinct sets: vertices in row $i$ are assigned weights from a set $\Delta$, and vertices in row $\overline{i}$ from a set $\Gamma$. (See Figure~\ref{fig:Symplectic}.) The weights depend on an integer-valued global statistic called \textbf{charge}, described next. 
\begin{itemize}
\item Consider a row of vertices having weights taken from $\Delta$ ice. The charge at any horizontal edge is the number of spins of $-$ at and to the left of the edge. The leftmost edge has charge $0$. The charge at any vertex is the charge at the left edge incident to the vertex. The charge at the edge of the bend connected to the row is defined in the same manner, and the charge at the bend's vertex is the charge at this edge. 
\item Consider a row of vertices having weights taken from $\Gamma$ ice. The charge at any horizontal edge is the sum of the number of spins of $+$ at and to the right of the edge and the charge at the vertex of the bend. The charge at any vertex is the charge at the right edge incident to the vertex. 
\end{itemize}
Thus, charge in a row of $\Delta$ ice begins at the leftmost edge and increments from left to right. Charge in a row of $\Gamma$ ice begins at the vertex of the bend and increments from right to left. For example, Figure~\ref{fig:Charge_example_state} shows the charge at each edge along rows $(1, \overline{1}\,)$ of the state given in Figure~\ref{fig:Example_state}.

\begin{figure}[h]
\makebox[\textwidth][c]{%
%
}
\caption{}
\label{fig:Charge_example_state}
\end{figure}

Fix a parameter $v$ and positive integers $n$ and $r$, with $n$ odd. Let $z_1$,~\ldots,~$z_r \in \mathbb{C}^{\times}$. Define the functions $g$, $\delta$, and $h$ on $\mathbb{Z}$ as follows: $g$ is periodic modulo $n$ and satisfies $g(0) = -v$ and $g(a) g(n-a) = v$ for all $a \in \mathbb{Z}$ with $n \nmid a$. Function $\delta$ is defined for every $a \in \mathbb{Z}$ by $\delta(a) = 1$ if $n \mid a$, and $\delta(a) = 0$ with $n \nmid a$. Function $h$ is defined for every $a \in \mathbb{Z}$ by $h(a) = (1 - v) \delta(a)$. The Boltzmann weights of non-bend vertices are listed in Table~\ref{tab:Delta_Gamma}, where the subscript $i$ means the illustrated vertices belong to row $i$, and where $a$ and $a+1$ stand for charge. The $z_i$ appearing in these weights is called a \textbf{spectral parameter}. If the vertex is in row $\overline{i}$, the same weights are used but with spectral parameter $z_{i}^{-1}$ rather than $z_i$.

\begin{table}[t]
\caption{Boltzmann weights of $\Delta$ ice and $\Gamma$ ice.}
\noindent
\begin{tabularx}{0.86\textwidth}{c*{6}{Y}}
\toprule
	{}
	& \multicolumn{6}{l}{\textbf{Boltzmann}} 
	\\
	{}
	& \multicolumn{6}{l}{\textbf{Weights}} 
	\\
\end{tabularx}
\label{tab:Delta_Gamma}
\end{table}

We will often refer to a bend connecting rows $(i, \overline{i}\,)$ as a $\Delta \Gamma$-bend to emphasize that row $i$ (resp., row $\overline{i}\,$) consists of admissible configurations having Boltzmann weights from $\Delta$ ice (resp., $\Gamma$ ice). Other types of bends will appear later; for now, we give the weights of the usual $\Delta \Gamma$-bends. For every $i$, the weights of the $\Delta \Gamma$-bends connecting rows $(i, \overline{i}\,)$ are
\begin{equation*}
	\wt_{\Delta \Gamma} 
	\biggl(\!
	%
	\biggr)
	= 
	z_{i}^{-1},
\end{equation*}
where $a$ and $a+1$ are the charges at the edges of the bend. These weights are \textbf{spectrally dependent:} if $i$ and $\overline{i}$ are interchanged, then $z_i$ and $z_{i}^{-1}$ are interchanged. We will call the bends after such an interchange ``flipped'' $\Delta \Gamma$-bends. 

\begin{theorem}
\label{P:n_admissible}
Let $\mathfrak{s}$ be an admissible state. If\/ $\wt(\mathfrak{s}) \neq 0$, then for every row of $\Delta$ ice (resp., $\Gamma$ ice) in $\mathfrak{s}$, each horizontal edge with spin $+$ (resp., $-$) in the row has a charge divisible by $n$.
\end{theorem}
\begin{proof}
Suppose $\wt(\mathfrak{s}) \neq 0$. Assume, for contradiction, that row $i$ of $\Delta$ ice has an edge with a spin of $+$ and a charge not divisible by $n$. Let $v$ be the leftmost vertex in the row such that the edge to the right of $v$ satisfies those conditions, say the edge has spin $+$ and charge $a$ with $n \nmid a$. The charge at $v$ is $a$ (nonzero), so $v$ is not the leftmost vertex in row $i$. Let $v'$ be the vertex to the left of $v$. 
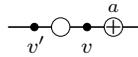
\[
	\mathord{%
\begin{tikzpicture}[>=stealth', scale=0.7, every node/.style={scale=1.0}]
\node[above] at (2,0.10) {\scriptsize$a$};
\foreach \x in {0.5,1.5}
{%
	\filldraw[black] (\x,0) circle(2pt);
}%
\draw[black, semithick] (0,0) -- (2.5,0);
\node[below, black] at (0.5,0) {\footnotesize$\,v'$};
\node[below, black] at (1.5,0) {\footnotesize$v \vphantom{v'}$};
\begin{scope}
\tikzstyle{every node}=[draw] 
\path[black, shape=circle, very thin] 
	(1,0) 
	node[fill=white, inner sep=-0.5pt] 
	{\tiny$\boldsymbol{\phantom{+}}$};
	
\path[black, shape=circle, very thin] 
	(2,0) 
	node[fill=white, inner sep=-0.5pt] 
	{\tiny$\boldsymbol{+}$};
\end{scope}
\end{tikzpicture}%
	}
\]
If the spin of the edge to the left of $v$ is $+$, then the charge at $v'$ equals the charge at $v$, contradicting our choice of $v$ as being the leftmost counterexample in row $i$. Thus, the spin to the left of $v$ is $-$, so $v$ is a $c_1$-vertex with $\wt(v) = h(a) z_i = 0$, since $n \nmid a$. But then $\wt(\mathfrak{s}) = 0$, a contradiction.

The case involving $\Gamma$ ice is handled similarly and is left to the reader. 
\end{proof}


Following~\cite{BBB}, we can make the Boltzmann weights of vertices in our model ``local'' by introducing data called decorated spins. By local, we mean weights depending on nearest-neighbor interactions rather than weights depending on the ``global'' statistic of charge. A \textbf{decorated spin} for a horizontal edge is an ordered pair $(\varepsilon, a)$, where $\varepsilon \in \{ +, - \}$ is a spin and $a \in \{ 0, 1, \ldots, n-1 \}$; we call $a$ the \textbf{decoration} of the decorated spin. We will usually denote $(\varepsilon, a)$ by either $\varepsilon a$ or $\begin{smallmatrix} a \\ \varepsilon \end{smallmatrix}$. For our configurations, we will denote $(\varepsilon, a)$ by drawing a circle with $\varepsilon$ in it and $a$ next to it, as shown below.
\[
	\mathord{%
	\begin{tikzpicture}[scale=0.95, every node/.style={scale=1.0}]
	\begin{scope}
	\tikzstyle{spin}=[draw, circle] 
	\path[black, very thin] (0,0) 
	node[inner sep=-0.5pt, fill=white, spin]
	{\tiny$\phantom{\boldsymbol{+}}$};
	\node[above] at (0,0.10) {\scriptsize$a$};
	\end{scope}
	\node at (0,0) {\scriptsize$\varepsilon$};
	\end{tikzpicture}%
	}
\]

Given an admissible state, to each horizontal edge in the state we assign a decoration $a$ that depends on both the spin $\varepsilon$ of the edge and the ice type of the edge. If the edge is in a row of $\Delta$ ice (resp., $\Gamma$ ice), the spin is $\varepsilon = -$ (resp., $\varepsilon = +$), and if the edge to the left (resp., right) has charge $c$, then $a \equiv c + 1 \pod{n}$; otherwise, $a \equiv c \pod{n}$. For each row of $\Delta$ ice, the leftmost edge in the row has decoration $0$. So the decoration at any horizontal edge of the model is precisely the charge at the edge modulo $n$. Since the Boltzmann weights in Table~\ref{tab:Delta_Gamma} depend only on the charge modulo $n$, the weights have a local interpretation. But it can be useful to view the charges shown in Table~\ref{tab:Delta_Gamma} as having values from $\mathbb{Z}$ rather than from $\{ 0, 1, \ldots, n-1 \}$. With this viewpoint, replacing the charge $a$ found in any configuration from Table~\ref{tab:Delta_Gamma} by an integer $a' \equiv a \pod{n}$ reduces the Boltzmann weight of the configuration to zero unless $a' = a$.

The use of decorated spins makes the partition functions of our model differ from the partition functions of the usual symplectic-ice model in that, rather than every horizontal edge being assigned only a spin, the edge is assigned a spin and an integer modulo $n$.

\section{The Yang--Baxter Equation}
\label{sec:Yang_Baxter_equation}

The Yang--Baxter equation (YBE) involves three vertices, which we will label as $z_1$, $z_2$, and $R_{z_1, z_2}$, where each of $z_1$ and $z_2$ also serves to indicate the spectral parameter used in the Boltzmann weight of the associated vertex. The weights of $z_1$ and $z_2$ are taken from $\Delta$ or $\Gamma$ ice. Concerning the possible Boltzmann weights of $R_{z_1, z_2}$, let $X$,~$Y \in \{ \Delta, \Gamma \}$, and suppose $z_1$ and $z_2$ have Boltzmann weights of ice types $X$ and $Y$, respectively. Then $R_{z_1, z_2}$ has a Boltzmann weight of ice type $XY$. The four possible ice types for $R_{z_1, z_2}$, along with their Boltzmann weights, are given in Table~\ref{tab:braid_weights}. The fact that the fixed integer $n$ is odd plays an important role in these weights and therefore in all the computations that follow.

\begin{theorem}[Yang--Baxter equation]
\label{thm:metaplecticYBE}
Let $X$,~$Y \in \{ \Delta, \Gamma \}$. Let $z_1$, $z_2$, and $R_{z_1, z_2}$ be the vertices in Figure~\ref{fig:YBE}, and suppose their Boltzmann weights are of ice types $X$, $Y$, and $XY$, respectively. Suppose the exterior spins $\smash{ (\dspin{c_1}{\varepsilon_1}, \dspin{c_2}{\varepsilon_2}, \varepsilon_3, \dspin{c_4}{\varepsilon_4}, \dspin{c_5}{\varepsilon_5}, \varepsilon_6) }$ are fixed. Then the partition functions of the configurations in Figure~\ref{fig:YBE} are equal, where the partition functions are computed by summing over all possible values of the interior spins $\smash{ (\dspin{e_1}{\alpha_1}, \dspin{e_2}{\alpha_2}, \alpha_3) }$ and $\smash{ (\dspin{f_1}{\omega_1}, \dspin{f_2}{\omega_2}, \omega_3) }$. 

\begin{figure}[h]
\centering
\makebox[\textwidth][c]{%
%
}
\caption{}
\label{fig:YBE}
\end{figure}
\end{theorem}

\begin{table}[t]
\begin{threeparttable}
\caption{Boltzmann weights of $\Gamma \Delta$, $\Delta \Delta$, $\Delta \Gamma$, and $\Gamma \Gamma$ ice. For both $\Delta \Delta$ and $\Gamma \Gamma$ ice, $a \neq b$ for every configuration whose decorations involve only $a$ and $b$. Also, $a \neq b$ for the last two configurations in the top row of $\Gamma \Delta$ ice.}
\noindent\begin{tabularx}{1.0\textwidth}{@{}l*{5}{Y}@{}}
\toprule
	{}
	& \multicolumn{5}{l}{\textbf{Boltzmann}} 
	\\
	{}
	& \multicolumn{5}{l}{\textbf{Weights}} 
	\\
\end{tabularx}
\begin{tablenotes}
{\footnotesize
\item[*] The Boltzmann weight is $v^2 z_{2}^n {-}  z_{1}^n$ if $2a \equiv 1 \pod{n}$. Else the Boltzmann weight is $g(2a {-} 1) (z_{1}^n {-} v z_{2}^n)$.
\item[\textdagger] Here $a {+} b \equiv 1 \pod{n}$.
\item[\textdaggerdbl] Here $a {+} b \not\equiv 1 \pod{n}$. The Boltzmann weight is $g(a {+} b {-} 1) (z_{1}^n {-} v z_{2}^n)$.
\item[\S] Here $a {+} b \equiv c {+} d \equiv 1 \pod{n}$, $a \not \equiv c \pod{n}$. Let $e \equiv a {-} c \pod{n}$ with $e \in [0, n {-} 1]$. The Boltzmann weight is $(v {-} 1) z_{1}^{n-e} z_2^{e}$ if $ad=0$ or if both $abcd \neq 0$ and $a > c$. The Boltzmann weight is $v(v {-} 1) z_{1}^{n-e} z_2^{e}$ if $bc=0$ or if both $abcd \neq 0$ and $a < c$. 
\item[\parbox{\widthof{$\!$\#}}{$\|$\hfil}\!] Here $a {+} b \equiv 1 \pod{n}$. Choose $a$ and $b$ in $[1,n]$.
\item[$\!$\#\!] Here $c \equiv a {-} b \pod{n}$ with $c \in [1,n {-} 1]$.  
\item[**] Choose $a$ in $[1, n]$. 
\item[\textdagger\textdagger] Here $a {+} b \not \equiv 1 \pod{n}$. The Boltzmann weight is $(z_{2}^n {-} v^n z_{1}^n)/g(a {+} b {-} 1)$.
\item[\textdaggerdbl\textdaggerdbl] Here $a {+} b \equiv c {+} d \equiv 1 \pod{n}$, $a \not \equiv c \pod{n}$. Let $e \equiv c {-} a \pod{n}$ with $e \in [1, n {-} 1]$. The Boltzmann weight is $(1 {-} v) v^{e-1} z_{1}^{e} z_{2}^{n-e}$.
}
\end{tablenotes}
\label{tab:braid_weights}
\end{threeparttable}
\end{table}

\begin{proof}
For every configuration in Figure~\ref{fig:YBE}, the exterior spins $\varepsilon_1$,~\ldots,~$\varepsilon_6$ have an even number of $+$ spins, so there are 32 choices for $(\varepsilon_1, \ldots, \varepsilon_6)$. The 32 cases when $R_{z_1, z_2}$ is of ice type $\Gamma \Gamma$ are given in~\cite{BBB}. We list all cases for the remaining three ice types in Appendices~\ref{app:appendixDD}--\ref{app:appendixGD}.

As an example, we give the case where $R_{z_1, z_2}$ is of ice type $\Delta \Delta$ and where $(\varepsilon_1, \ldots, \varepsilon_6) = (-, +, +, +, -, +)$. This is listed as Case 6 in Appendix~\ref{app:appendixDD}. The left side of the YBE has exactly two admissible states, given in Figure~\ref{fig:YBE_DDcase6left}. The right side of the YBE has exactly one admissible state, given in Figure~\ref{fig:YBE_DDcase6right}.

\begin{figure}[htb]
\makebox[\textwidth][c]{%
\hfill\begin{minipage}[b]{0.45\textwidth}
\centering
%
\subcaption{}
\label{fig:YBE_DDcase6left}
\end{minipage}\hspace{-1.15in}%
\begin{minipage}[b]{0.45\textwidth}
\centering
%
%
\subcaption{}
\label{fig:YBE_DDcase6right}
\end{minipage}\hfill%
}
\caption{}
\label{fig:YBE_DDcase6}
\end{figure}

\textbf{Case 6a:} $a \neq 0$. The state at the top of Figure~\ref{fig:YBE_DDcase6left} is excluded. The Boltzmann weight of the state at the bottom of Figure~\ref{fig:YBE_DDcase6left} is $(1-v) z_{1}^{n-a+1} z_{2}^a$, and the Boltzmann weight of the state in Figure~\ref{fig:YBE_DDcase6right} is $(1-v) z_{1}^{n-a+1} z_{2}^a$.

\textbf{Case 6b:} $a = 0$. The Boltzmann weights of the states in Figure~\ref{fig:YBE_DDcase6left} are $(1-v) (z_{1}^n - z_{2}^n) z_1$ and $(1-v) z_1 z_{2}^n$, respectively, the sum of which is $(1-v) z_{1}^{n+1}$, the Boltzmann weight of the state in Figure~\ref{fig:YBE_DDcase6right}.
\end{proof}

The Yang--Baxter equation given above was unknown when the proof of Statement B in~\cite{BBF_annals} was written, a lengthy proof written in terms of crystal bases. The existence of a solution to the Yang--Baxter equation in the current paper reduces much of the work done in~\cite{BBF_annals} to about a dozen pages. In Section~\ref{sec:Function_Eqns_Partition}, we will use the Yang--Baxter equation to prove local functional equations for the partition function, much like the functional equations appearing in Kuperberg~\cite{Kuperberg2}.

A recent application was shown in Brubaker, Bump, Buciumas, and Gray~\cite{BBBGray}, in which values of spherical Whittaker functions on an $n$-fold metaplectic cover of $\GL(r,F)$ ($F$ a nonarchimedean local field) were interpreted as partition functions of two statistical-mechanical models differing in their Boltzmann weights, which are the Boltzmann weights listed in Tables~\ref{tab:Delta_Gamma} and~\ref{tab:braid_weights}. Equality of the partition functions was shown by using the commutativity of transfer matrices associated to the two models by the Yang--Baxter equation in Theorem~\ref{thm:metaplecticYBE}.

\section{Connections to Metaplectic Eisenstein Series}
\label{sec:Connections_Metaplectic}

Friedberg and Zhang~\cite{FZ} showed that the generating function on strict symplectic patterns given in Beineke, Brubaker, and Frechette~\cite{BeBrFr2} is the prime-power supported coefficients of a metaplectic Eisenstein series on an odd-degree cover of $\SO (2r + 1)$. The connection in~\cite{FZ} is made through an intermediate bijection in~\cite{BeBrFr1}. We show that this generating function---a multiple Dirichlet series---is related to the Boltzmann weights of certain admissible states for our symplectic-ice model. We first introduce some notation and terminology from~\cite{BeBrFr2}.

By a \textbf{symplectic pattern}, or a \textbf{Gelfand--Tsetlin pattern of type C}, we mean a triangular arrangement $P$ of nonnegative integers of the form
\[
	\begin{matrix}
	a_{0,1} & {} & a_{0,2} & \ldots & a_{0,r} & {} 
	\\
	{} & b_{1,1} & {} & b_{1,2} & \ldots & b_{1,r} 
	\\
	{} & {} & a_{1,2} & \ldots & a_{1,r} & {} 
	\\
	{} & {} & {} &\hdotsfor{3} 
	\\
	{} & {} & {} & {} & a_{r-1,r} & {}
	\\
	{} & {} & {} & {} & {} & b_{r,r}
	\end{matrix}
\]
where the rows interleave: for all $i$ and $j$,
\[
	\min\{ a_{i-1, j} , a_{i, j} \} \geq b_{i,j} \geq \max \{ a_{i-1, j+1} , a_{i, j+1} \}
\]
and 
\[
	\min\{ b_{i+1, j-1} , b_{i, j-1} \} \geq a_{i,j} \geq \max \{ b_{i+1, j} , b_{i, j} \}.
\]
We say $P$ is \textbf{strict} if the entries in any row are strictly decreasing and if $a_{i,r} \neq 0$ for all $i$. (See Proctor~\cite{Proctor}.) Denote by $\GT_{\mathrm{str}} (\lambda + \rho)$ the set of all strict symplectic patterns with a fixed top row of $\lambda + \rho$. This set parametrizes a basis for the highest-weight representation of $\Sp (2r, \mathbb{C})$ with highest weight $\lambda + \rho$.

\begin{theorem}
\label{P:states_bijection_patterns}
The sets\/ $\mathfrak{S}_{\lambda}$ and\/ $\GT_{\mathrm{str}} (\lambda + \rho)$ are in bijective correspondence.
\end{theorem}

\noindent The proof is omitted.

For example, Figure~\ref{fig:state_pattern} shows an element of $\mathfrak{S}_{(2,1,1)}$ and the corresponding element of $\GT_{\mathrm{str}} (5,3,2)$. The entries in the pattern are the column numbers of those vertices in the admissible state having vertical spins of $-$. (Each entry $\ast$, which records the vertex in a bend having a ``vertical'' spin of $-$, is set equal to $0$.) In general, row $i$ (resp., $\overline{i}\,$) in an admissible state of ice gives rise to the row of entries $a_{i-1,i}$, $a_{i-1,i+1}$,~\ldots, $a_{i-1,r}$ (resp., $b_{i,i}$, $b_{i,i+1}$,~\ldots, $b_{i,r}$) of the corresponding strict pattern, and vice versa.

\begin{figure}[ht]
\centering
\makebox[\textwidth]{%
%
}
\caption{An admissible state (left) and the corresponding pattern (right).}
\label{fig:state_pattern}
\end{figure}

Let $P \in \GT_{\mathrm{str}}(\lambda + \rho)$, and let $\mathfrak{s} \in \mathfrak{S}_{\lambda}$ be the state of symplectic ice corresponding to $P$. Shown below are typical rows in $P$. 
%
}
\]
For all $i$,~$j \in \{ 1, 2, \ldots, r \}$ with $i \leq j$, let
\begin{equation}
\label{E:uvw}
	v_{i,j} = \sum_{k=i}^{j} (a_{i-1,k} - b_{i,k}), 
	\qquad
	w_{i,j} = \sum_{k=j}^{r} (a_{i,k} - b_{i,k}),
	\qquad
	u_{i,j} = v_{i,r} + w_{i,j}.
\end{equation}
(The entries $a_{i,j}$ and $b_{i,j}$ are set equal to $0$ if they do not appear in $P$.) The integer $v_{i,j}$ is the sum of the differences $a_{i-1,k} - b_{i,k}$ with $k$ ranging from $i$ to $j$. These differences are indicated by the arrows shown below. 
\[
	\mathord{%
%
%
}
\]
One can verify that $v_{i,j}$ is the charge at the vertex in row $i$, column $b_{i,j}$ of $\mathfrak{s}$. Thus, the charge at the vertex in the bend connecting rows $(i, \overline{i} \,)$ is $v_{i,r}$. For example, if $\mathfrak{s}$ is the admissible state shown in Figure~\ref{fig:state_pattern}, then $v_{1,3} = (5-4) + (3-2) + (2-0) = 4$, which equals the charge at the vertex in the bend connecting rows $(1, \overline{1}\,)$, as is shown in Figure~\ref{fig:Charge_example_state}. The integer $u_{i,j}$ is the sum $v_{i,r} + w_{i,j}$, and the differences involved in computing $u_{i,j}$ are indicated by the arrows shown below. One can verify that if $i < r$, then $u_{i,j}$ is the charge at the vertex in row $\overline{i}$, column $a_{i,j}$ of $\mathfrak{s}$. 
\[
	\mathord{%
\begin{tikzpicture}[>=stealth', scale=0.7, every node/.style={scale=1.0}]
\matrix (m) [matrix of math nodes]
{
	\text{\parbox[t]{\widthof{$a_{i-1,j+1}$}}{\centering$a_{i-1, i}$}}
	& {} 
	& \text{\parbox[t]{\widthof{$a_{i-1,j+1}$}}{\centering$a_{i-1,i+1}$}}
	& \ldots 
	& \text{\parbox[t]{\widthof{$a_{i-1,j+1}$}}{\centering$a_{i-1,j-1}$}}
	& {} 
	& \text{\parbox[t]{\widthof{$a_{i-1,j+1}$}}{\centering$a_{i-1,j}$}}
	& \ldots & {} 
	& a_{i-1, r} 
	& {}
	\\
	{} & b_{i,i} & {} & b_{i,i+1} & \ldots 
	& b_{i,j-1} & {} & b_{i,j} & \ldots 
	& {} & b_{i,r}
	\\
	{} & {} & a_{i,i+1} & \ldots 
	& a_{i,j-1} & {} & a_{i,j} & \ldots 
	& {} & a_{i,r} & {}
	\\
};
\draw[semithick, ->, black] 
	(m-1-1.south) -- (m-2-2.175);
\draw[semithick, ->, black] 
	(m-1-3.south) -- (m-2-4.175);
\draw[semithick, ->, black] 
	(m-1-5.south) -- (m-2-6.175);
\draw[semithick, ->, black] 
	(m-1-7.south) -- (m-2-8.175);
\draw[semithick, ->, black] 
	(m-1-10.south) -- (m-2-11.175);
\draw[semithick, ->, black] 
	(m-3-7.north) -- (m-2-8.185);
\draw[semithick, ->, black] 
	(m-3-10.north) -- (m-2-11.185);
\end{tikzpicture}%
}
\]

\begin{remark}
If $U_q ( \fraksp (2r))$ denotes the quantized universal enveloping algebra of the Lie algebra $\fraksp (2r)$, then the integers $v_{i,j}$, $w_{i,j}$, and $u_{i,j}$ are related to Kashiwara lowering and raising operators in the crystal graph associated to the highest-weight representation for $U_q (\fraksp (2r))$ of highest weight $\lambda + \rho$. (See Littelmann~\cite{Littelmann}.)
\end{remark}

We now introduce some algebraic preliminaries. Let $F$ be a number field containing the $2n$th roots of unity. Let $S$ be a finite set of places of $F$ containing all archimedean places and all places ramified over $\mathbb{Q}$, and suppose $S$ is large enough so that $\mathcal{O}_S  = \{ a \in F \mid \text{$a \in \mathcal{O}_v$ for all $v \notin S$} \}$, the ring of $S$-integers in $F$, is a PID. Let $\mathcal{O}_{S}^{\times}$ be the units in $\mathcal{O}_S$. For every $r$-tuple $\mathbf{m} = (m_1, \ldots, m_r)$ of nonzero integers in $\mathcal{O}_S$, the associated multiple Dirichlet series in the $r$ complex variables $s_1$,~\ldots,~$s_r$ is the sum
\begin{equation}
\label{E:gen_fun}
	\mathcal{Z}_{\Psi} (\mathbf{s}; \mathbf{m}) = \mathcal{Z}_{\Psi} (s_1,\ldots,s_r; \mathbf{m}) =
	\sum_{
	\substack{\mathbf{c} \in (\mathcal{O}_S / \mathcal{O}_{S}^{\times})^r \\ \mathbf{c} = (c_1, \ldots, c_r)}
	}
	\dfrac{ H^{(n)} (\mathbf{c}; \mathbf{m}) \Psi (\mathbf{c}) }{|c_{1}|^{2 s_1} \ldots |c_{r}|^{2 s_r}} \, ,
\end{equation}
where the sum ranges over all nonzero ideals $c_i$ of $\mathcal{O}_S$, the coefficients $H^{(n)} (\mathbf{c}; \mathbf{m})$ are related to the $n$th-power reciprocity law in $F$, and $\Psi$ is some $\mathbb{C}$-valued function defined on $(F_{S}^{\times})^r$. Also $|c_{i}| = |c_{i}|_S$ is the norm of $c_{i} \in \mathcal{O}_S / \mathcal{O}_{S}^{\times}$ as a product of local norms in $F_S = \prod_{v \in S} F_v$.

Given $\mathbf{c} \in (\mathcal{O}_S / \mathcal{O}_S^{\times})^r$ and $\mathbf{m} \in \mathcal{O}_{S}^{r}$, say $\mathbf{c} = (c_1, \ldots, c_r)$ and $\mathbf{m} = (m_1, \ldots, m_r)$, one can determine $H^{(n)}(\mathbf{c}; \mathbf{m})$ by specifying the prime-power coefficients $H^{(n)} (p^{\mathbf{k}}; p^{\boldsymbol{\ell}})$ for a generator $p$ of  some prime ideal in $\mathcal{O}_S$, where $\mathbf{k} = (k_1, \ldots, k_r)$, $\boldsymbol{\ell} = (\ell_1, \ldots, \ell_r)$, and $k_i = \ord_p (c_i)$ and $\ell_i = \ord_p (m_i)$ for all $i$. We therefore focus our attention on the prime-power coefficients.

Throughout this section, fix the $r$ complex variables $s_1$,~\ldots,~$s_r$, and fix a prime $p$ in $\mathcal{O}_S$, i.e., $p$ generates some prime ideal of $\mathcal{O}_S$. Let $q = |\mathcal{O}_S / p \mathcal{O}_S|$, the cardinality of the residue class $\mathcal{O}_S / p \mathcal{O}_S$. Let $\varphi (p^a)$ be the Euler $\varphi$-function for $\mathcal{O}_S / p^a \mathcal{O}_S$; then $\varphi(p^a) = q^{a}(1-q^{-1})$. Let $g_t (p^{\alpha}, p^{\beta})$ be an $n$th-power Gauss sum. In the prime-power coefficient $H^{(n)} (p^{\mathbf{k}}; p^{\boldsymbol{\ell}})$, which we will describe shortly, the components of $\boldsymbol{\ell} = (\ell_1,\ldots,\ell_r)$ satisfy $\lambda = (\ell_1 + \cdots + \ell_r, \ldots, \ell_1 + \ell_2, \ell_1)$, where our fixed $\lambda$ is a dominant integral element for $\Sp (2r, \mathbb{C})$. Write $\lambda + \rho$ as $(L_r,\ldots,L_1)$.

Let $s_{a}(i)$ and $s_b (i)$ be the sums of the $i$th row of, respectively, the $a$- and $b$-entries in $P$: $s_{a}(i) = \sum_{k=i+1}^{r} a_{i,k}$ and $s_{b}(i) = \sum_{k=i}^{r} b_{i,k}$. Let $\wt (P) \vcentcolon= (\wt_1(P), \ldots, \wt_r(P))$, where
\begin{equation}
\label{E:weight_vector}
	\wt_{i}(P) = 
	s_a (r-i) - 2 s_b (r-i+1) + s_a (r-i+1) 
\end{equation}
for all $i$. One can verify that $\wt (P) = (u_{r,r}, \ldots, u_{1,1})$. Let $\mathbf{k}(P) \vcentcolon= (k_1(P),\ldots,k_r(P))$, where
\begin{equation}
\label{E:k_vector}
	k_1(P) = \tfrac{1}{2} \sum_{j=1}^{r} \wt_j(P) + L_j,
	\qquad
	k_i(P) = \sum_{j=i}^{r} \wt_j(P) + L_j ,
\end{equation}
for all $i \neq 1$. Each coefficient $H^{(n)}(p^{\mathbf{k}}; p^{\boldsymbol{\ell}})$ in the generating function $\mathcal{Z}_{\Psi} (\mathbf{s}; p^{\boldsymbol{\ell}})$ is defined as
\begin{equation}
\label{E:prime_coef}
	H^{(n)}(p^{\mathbf{k}}; p^{\boldsymbol{\ell}}) = \sum_{P \colon \mathbf{k}(P) = \mathbf{k}} G (P),
\end{equation}
where the sum ranges over all $P \in \GT_{\mathrm{str}} (\lambda + \rho)$ such that the row sums of $P$ are fixed according to~\eqref{E:weight_vector} and~\eqref{E:k_vector}, and where $G(P)$, which we now define, is a weighting function dependent on $P$. To each entry $a_{i,j}$ in $P$ with $i \geq 1$, let
\begin{align*}
	\gamma( a_{i,j} ) &= 
	\begin{cases}
	\text{\parbox[t]{\widthof{$q^{v_{i,j} - 1} g_{(1 + \delta_{jr})v_{i,j}} (1, p)$}}{$q^{u_{i,j} - 1} g_{u_{i,j}} (1, p)$}}
	&\text{if $a_{i,j} = b_{i,j-1}$,}\\
	q^{u_{i,j}} 
	&\text{if \parbox[t]{\widthof{$a_{i,j} = b_{i,j-1}$}}{$a_{i,j} = b_{i,j}$,}}\\
	\varphi(p^{u_{i,j}})
	&\text{if $b_{i,j} < a_{i,j} < b_{i,j-1}$ and $n \mid u_{i,j}$,}\\
	0	
	&\text{if $b_{i,j} < a_{i,j} < b_{i,j-1}$ and $n \nmid u_{i,j}$,}
	\end{cases}
\intertext{and to each entry $b_{i,j}$ in $P$, let}
	\gamma( b_{i,j} ) &=
	\begin{cases}
	q^{v_{i,j}}	&\text{if \parbox[t]{\widthof{$b_{i,j} = a_{i-1,j+1}$}}{$b_{i,j} = a_{i-1,j}$,}}\\  
	q^{v_{i,j} - 1} g_{(1 + \delta_{jr})v_{i,j}} (1, p)
	&\text{if $b_{i,j} = a_{i-1,j+1}$,}\\ 
	\varphi(p^{v_{i,j}})	
	&\text{if $a_{i-1,j+1} < b_{i,j} < a_{i-1,j}$ and $n \mid (1+\delta_{jr})v_{i,j}$,}\\
	0	
	&\text{if $a_{i-1,j+1} < b_{i,j} < a_{i-1,j}$ and $n \nmid (1+\delta_{jr})v_{i,j}$,}
	\end{cases}
\end{align*}
where $\delta$ is the Kronecker delta function. Let
\[
	G(P) = 
	\prod_{1 \leq i \leq j \leq r} 
	\gamma(a_{i,j}) \gamma(b_{i,j}),
\]
where we set each $\gamma(a_{i,i})$ equal to $1$, since $a_{i,i}$ is not in $P$.

Rather than working with $H^{(n)}$, we will work with $\smash{ \widetilde{H}^{(n)} }$, a ``normalization'' of $H^{(n)}$, which we now define. To each of $a_{i,j}$ and $b_{i,j}$ in $P$, set $\widetilde{\gamma}(a_{i,j}) \vcentcolon = q^{-u_{i,j}} \gamma(a_{i,j})$ and $\widetilde{\gamma}(b_{i,j}) \vcentcolon = q^{-v_{i,j}} \gamma(b_{i,j})$. Let 
\[
	\widetilde{H}^{(n)}(p^{\mathbf{k}}; p^{\boldsymbol{\ell}}) = \sum_{P \colon \mathbf{k}(P) = \mathbf{k}} \widetilde{G} (P),
\]
where $\smash{\widetilde{G} (P) = \prod \widetilde{\gamma} (a_{i,j}) \widetilde{\gamma} (b_{i,j}) }$. Then $\smash{ H^{(n)} = \widetilde{H}^{(n)}q^{k_1+\cdots+k_r}}$ by the following lemma.

\begin{lemma}
For every $P \in \GT_{\mathrm{str}} (\lambda + \rho)$,
\[
	\sum_{i=1}^{r} k_i(P) = \sum_{i=1}^{r} \left[ \, \sum_{j=i+1}^{r} u_{i,j} + \sum_{j=i}^{r} v_{i,j} \right],
\]
where $k_1(P)$,~\ldots,~$k_r(P)$ are defined according to~\eqref{E:k_vector}. 
\end{lemma}

\noindent The proof is given in~\cite{BeBrFr2}.

\begin{theorem}
\label{T:sym_pattern_bij}
Continue using the notation above. Set the parameter $v$ equal to $q^{-1}$, and write $\mathbf{z} = (z_1,\ldots,z_r)$. Then
\[
	\mathcal{Z}_{\Psi}(\mathbf{s}; p^{\boldsymbol{\ell}}) = 
	\mathbf{z}^{\lambda + \rho} \sum_{\mathbf{k} = (k_1,\ldots,k_r)} \Biggl[  \Psi (p^{\mathbf{k}}) \sideset{}{'}\sum_{\mathfrak{s} \in \mathfrak{S}_{\lambda}} \wt (\mathfrak{s}) \Biggr],
\]
where the primed sum ranges over the $\mathfrak{s}$ in $\mathfrak{S}_{\lambda}$ that correspond to those $P$ in $\GT_{\mathrm{str}} (\lambda + \rho)$ satisfying $\mathbf{k}(P) = \mathbf{k}$. The $\mathbf{s}=(s_1,\ldots,s_r)$ and $\mathbf{m}=(m_1,\ldots,m_r)$ are related by $q^{1-2s_1} = z_{r}^2$ and $q^{1-2s_i} = z_{r-i+1}^{\vphantom{-1}} / z_{r-i+2}$ for every $i > 1$.
\end{theorem}

\begin{proof}
Let $\mathfrak{s} \in \mathfrak{S}_{\lambda}$ correspond to $P \in \GT_{\mathrm{str}}(\lambda + \rho)$. For each entry $x_{i,j}$ in $P$, we will associate to $\widetilde{\gamma}(x_{i,j})$ the Boltzmann weight of a vertex in either rows $i$ or $\overline{i}$ of $\mathfrak{s}$, but we will exclude any factor of $z_{i}^{\pm 1}$ in this weight, compensating for all the $z_{i}^{\pm 1}$ later. We will keep track of all these $\smash{ z_{i}^{\pm 1} }$ as follows:
\begin{itemize}
\item Given $b_{i,j}$, we count factors of $z_i$ that are part of the weights of those vertices in row~$i$ between columns $a_{i-1,j}$ and $a_{i-1,j+1}$, including column $a_{i-1,j+1}$. If $b_{i,j} = a_{i-1,j}$, we include this column. (See the left side of Figure~\ref{fig:excluded}.)
\item Given $a_{i,j}$, we count factors of $z_i^{-1}$ that are part of the weights of those vertices in row~$\overline{i}$ between columns $b_{i,j-1}$ and $b_{i,j}$, including column $b_{i,j}$. If $a_{i,j} = b_{i-1,j}$, we include this column. (See the right side of Figure~\ref{fig:excluded}.)
\end{itemize}

\begin{figure}[htb]
\centering
\makebox[\textwidth][c]{%
%
}
\caption{Given $b_{i,j}$ (left) or $a_{i,j}$ (right), the ellipses show those vertices that are considered when counting the number of excluded factors of $\smash{z_{i}^{\pm 1}}$.}
\label{fig:excluded}
\end{figure}

Note that $a_{i,i+1}$ and $b_{i,i}$ in $P$ correspond to vertices whose weights do not include $\smash{z_{i}^{\pm 1}}$, so we need not worry about missing a factor of $z_{i}^{\pm 1}$ for these leftmost entries. Also each vertex in row $i$ (resp., $\overline{i}\,$) to the left of column $a_{i-1,i}$ (resp., $b_{i,i}$) has a weight of $1$, so we need not worry about this vertex. In the figures below, we will indicate  by {\small$\square$} the vertex in $\mathfrak{s}$ that gives rise to the entry in $P$ under consideration, and we will denote by $v$ the vertex directly above that vertex.

Let $b_{i,j}$ be an entry in $P$. There are four cases.

\textbf{Case 1:} $b_{i,j} = a_{i-1,j}$. Then $v$ is a $b_1$-vertex. If $j \neq r$, rows $(i, \overline{i}\,)$ are shown on the left side of Figure~\ref{fig:b_left}, where $a = v_{i,j}$. Since $\wt(v) = 1$ and $\widetilde{\gamma} (b_{i,j}) = 1$, we let $\widetilde{\gamma} (b_{i,j})$ correspond to $\wt(v)$. Each vertex in row $i$ strictly between columns $a_{i-1,j-1}$ and $a_{i-1,j+1}$ has a weight of $1$.

Note that the vertex in row $i$, column $a_{i-1,j+1}$ could be a $c_2$-vertex with charge $a$, so its weight would be $\delta (a) = 0$ if $n \nmid a$, which it seems would kill off the bijection we are trying to establish. However, row $i$ would contain a vertex with charge $a$ and weight $h(a) z_i$, so whether $n \mid a$ or $n \nmid a$ is handled at this vertex; such a vertex is dealt with in Case 3.

If $j = r$, the rows are shown on the right side of Figure~\ref{fig:b_left}, where $a = v_{i,r}$. The same correspondence is taken. Each vertex in row $i$ to the right of column $b_{i,r}$ has a weight of $1$. (Later we will worry about the weight of the vertex in this type of bend.)

In general, $\widetilde{\gamma} (b_{i,j}) \longleftrightarrow \wt(v)$. The number of factors of $z_i$ excluded is $\smash{ a_{i-1,j} - b_{i,j} }$.

\begin{figure}[H]
\centering
\makebox[\textwidth][c]{%
%
}
\caption{}
\label{fig:b_left}
\end{figure}

\textbf{Case 2:} $b_{i,j} = a_{i-1,j+1}$. If $j \neq r$, rows $(i, \overline{i}\,)$ are shown on the left side of Figure~\ref{fig:b_right}, where $a = v_{i,j}$. Then $v$ is an $a_2$-vertex. Since $\wt (v) = g(v_{i,j}) z_i$ and $\widetilde{\gamma}(b_{i,j}) = q^{- 1} g_{v_{i,j}} (1, p)$, we let $\widetilde{\gamma}(b_{i,j})$ correspond to $\wt(v)/z_i$. Each vertex in row $i$ strictly between columns $a_{i-1,j}$ and $b_{i,j}$ has a weight of $z_i$. The number of factors of $z_i$ excluded is $a_{i-1,j} - b_{i,j} $.

If $j = r$, then $b_{i,j} = 0$. The rows are shown on the right side of Figure~\ref{fig:b_right}, where $a = v_{i,r}$. Then $v$ is the vertex in the bend. Since $\wt(v) = g(2v_{i,r}) z_i$ and $\widetilde{\gamma}(b_{i,r}) = q^{- 1} g_{2v_{i,r}} (1, p)$, we let $\widetilde{\gamma}(b_{i,r})$ correspond to $\wt(v)/z_i$. Each vertex in row $i$ to the right of column $a_{i-1,r}$ has a weight of $z_i$. The number of factors of $z_i$ excluded is $a_{i-1, r}$.

In general, $\widetilde{\gamma}(b_{i,j}) \longleftrightarrow \wt(v)/z_i$, and the number of factors of $z_i$ excluded is $\smash{ a_{i-1, j} - b_{i, j} }$.

\begin{figure}[H]
\centering
\makebox[\textwidth][c]{%
%
}
\caption{}
\label{fig:b_right}
\end{figure}

\textbf{Case 3:} $a_{i-1,j+1} < b_{i,j} < a_{i-1,j}$ and $n \mid (1+\delta_{jr})v_{i,j}$. Since $n$ is odd, $n \mid v_{i,j}$. Then $v$ is a $c_1$-vertex. If $j \neq r$, rows $(i, \overline{i} \,)$ are shown on the left side of Figure~\ref{fig:b_neither1_divisible}, where $a = v_{i,j}$. Since $\wt(v) = (1-v)z_i$ (the $v$ on the right side of the previous equation is the fixed parameter) and $\widetilde{\gamma} (b_{i,j}) = 1-q^{-1}$, we let $\widetilde{\gamma} (b_{i,j})$ correspond to $\wt(v)/z_i$ after setting $q^{-1}$ equal to the parameter $v$. Each vertex in row $i$ strictly between columns $a_{i-1,j}$ and $b_{i,j}$ (resp., $b_{i,j}$ and $a_{i-1,j+1}$) has a weight of $z_i$ (resp., $1$). The number of factors of $z_i$ excluded is $a_{i-1,j} - b_{i,j}$.

If $j = r$, then $b_{i,j} \neq 0$. The rows are shown on the right side of Figure~\ref{fig:b_neither1_divisible}, where $a = v_{i,r}$. The same correspondence is taken. Each vertex in row $i$ strictly between columns $a_{i-1,r}$ and $b_{i,r}$ has a weight of $z_i$, and each vertex to the right of column $b_{i,r}$ has a weight of $1$. The number of factors of $z_i$ excluded is $a_{i-1, r} - b_{i, r}$.

In general, $\widetilde{\gamma} (b_{i,j}) \longleftrightarrow \wt(v)/z_i$, and the number of factors of $z_i$ excluded is $\smash{ a_{i-1, j} - b_{i, j} }$.

\begin{figure}[H]
\makebox[\textwidth][c]{%
%
}
\caption{}
\label{fig:b_neither1_divisible}
\end{figure}

\textbf{Case 4:} $a_{i-1,j+1} < b_{i,j} < a_{i-1,j}$ and $n \nmid (1+\delta_{jr})v_{i,j}$. Then $\wt(v) = 0$ and $\widetilde{\gamma}(b_{i,j}) = 0$. The number of factors of $z_i$ excluded is the same as before.

We have taken into account the weights of all vertices in row $i$. In summary, we have the following correspondence
\begin{equation}
\label{E:association_b}
	\widetilde{\gamma}( b_{i,j} )
	\quad \longleftrightarrow \quad
	\begin{cases}
	1	
	&\text{if $b_{i,j} = a_{i-1,j}$,}\\
	q^{-1}g((1 + \delta_{jr}) v_{i,j})
	&\text{if $b_{i,j} = a_{i-1,j+1}$,}\\
	1-v
	&\text{if $a_{i-1,j+1} < b_{i,j} < a_{i-1,j}$ and $n \mid (1+\delta_{jr})v_{i,j}$,}\\
	0	
	&\text{if $a_{i-1,j+1} < b_{i,j} < a_{i-1,j}$ and $n \nmid (1+\delta_{jr})v_{i,j}$.}
	\end{cases}
\end{equation}

Let $a_{i,j}$ be an entry in $P$ with $i \geq 1$. There are four cases. Note that if $b_{i,r} \neq 0$, there are vertices in row~$\overline{i}$ to the right of column $b_{i,r}$. So when analyzing the entry $a_{i,r}$, we will need to make sure we include in our calculations the Boltzmann weights of these vertices. Then the weights of all vertices in row $\overline{i}$, including the other type of bend, will be included.

\textbf{Case 1:} $a_{i,j} = b_{i,j-1}$. Then $v$ is a $b_1$-vertex. If $j \neq r$, rows $(i, \overline{i} \,)$ are shown in Figure~\ref{fig:a_left1}, where $a = u_{i,j}$. Since $\wt(v) = g(u_{i,j})$ and $\widetilde{\gamma} (a_{i,j}) = q^{- 1} g_{u_{i,j}} (1, p)$, we let $\widetilde{\gamma} (a_{i,j})$ correspond to $\wt(v)$. Each vertex in row $\overline{i}$ strictly between columns $a_{i,j}$ and $b_{i,j}$ has a weight of $1$.

The vertex in row $\overline{i}$, column $b_{i,j}$ is either a $b_1$-vertex or a $c_2$-vertex. If it is a $b_1$-vertex, then it can be omitted in the current discussion, since it would be included when dealing with the entry $a_{i,j+1}$. Suppose it is a $c_2$-vertex with some charge of $c$. Since $c$ might satisfy $n \nmid c$, in which case its weight would be $\delta (c) = 0$, it seems this might kill off the bijection we are trying to establish. However, row~$i$ or row $\overline{i}$ will then contain a vertex with charge $c$ and weight $h(c) z_{i}^{\pm 1}$, so whether or not $n \nmid c$ is handled at this vertex.

\begin{figure}[H]
\centering
\makebox[\textwidth][c]{%
%
}
\caption{}
\label{fig:a_left1}
\end{figure}

If $j=r$, the rows are shown on either side of Figure~\ref{fig:a_left2}, where $a = u_{i,r}$. The same correspondence is taken. The left (resp., right) side of the figure is for the case when $b_{i,r} \neq 0$ (resp., $b_{i,r} = 0$). Each vertex in row $\overline{i}$ strictly between columns $a_{i,r}$ and $b_{i,r}$ has a weight of $1$.

We need to take into account the weights of those vertices to the right of column $b_{i,r}$. If $b_{i,r} \neq 0$, the weight of each of these vertices and the vertex in the bend is $z_{i}^{-1}$ (the weight of the vertex in the other type of bend, which appears on the right side of Figure~\ref{fig:a_left2}, was dealt with previously). The number of these factors of $z_{i}^{-1}$ is $b_{i,r}$, and this number will be included in our total after Case 4.

In general, $\widetilde{\gamma}(a_{i,j}) \longleftrightarrow \wt(v)$, and the number of factors of $\smash{ z_{i}^{-1} }$ excluded is $b_{i,j-1} - a_{i,j}$. In addition, there are $b_{i,r}$ factors of $z_{i}^{-1}$ excluded.

\begin{figure}[H]
\centering
\makebox[\textwidth][c]{%
%
}
\caption{}
\label{fig:a_left2}
\end{figure}

\textbf{Case 2:} $a_{i,j} = b_{i,j}$. Then $v$ is an $a_2$-vertex. If $j \neq r$, rows $(i, \overline{i} \,)$ are shown on the left side of Figure~\ref{fig:a_right}, where $a = u_{i,j}$. Since $\wt(v) = z_{i}^{-1}$ and $\widetilde{\gamma}(a_{i,j}) = 1$, we let $\widetilde{\gamma}(a_{i,j})$ correspond to $\wt(v)/z_{i}^{-1}$. Each vertex in row $\overline{i}$ strictly between columns $b_{i,j-1}$ and $a_{i,j}$ has a weight of $\smash{ z_{i}^{-1} }$. The number of factors of $\smash{ z_{i}^{-1}}$ excluded is $b_{i,j-1} - a_{i,j} $.

If $j = r$, then $b_{i,r} \neq 0$. The rows are shown on the right side of Figure~\ref{fig:a_right}, where $a = u_{i,r}$. The same correspondence is taken. Each vertex in row $\overline{i}$ strictly between columns $b_{i, r-1}$ and $a_{i,r}$ has a weight of $\smash{ z_{i}^{-1} }$. The number of factors of $\smash{ z_{i}^{-1} }$ excluded is $b_{i, r-1} - a_{i,r}$. In addition, each vertex to the right of column $b_{i,r}$ (including the vertex in the bend) has a weight of $z_{i}^{-1}$. As in Case 1, the number of these factors of $z_{i}^{-1}$ is $b_{i,r}$, and this number will be included in our total after Case 4.

In general, $\widetilde{\gamma}(a_{i,j}) \longleftrightarrow \wt(v)/z_{i}^{-1}$. The number of factors of $\smash{ z_{i}^{-1} }$ excluded is $b_{i, j-1} - a_{i, j}$.  In addition, there are $b_{i,r}$ factors of $\smash{ z_{i}^{-1} }$ excluded.

\begin{figure}[H]
\centering
\makebox[\textwidth][c]{%
%
}
\caption{}
\label{fig:a_right}
\end{figure}

\textbf{Case 3:} $b_{i,j} < a_{i,j} < b_{i,j-1}$ and $n \mid u_{i,j}$. Then $v$ is a $c_1$-vertex. If $j \neq r$, rows $(i, \overline{i} \,)$ are shown in Figure~\ref{fig:a_neither1_divisible}, where $a = u_{i,j}$. Since $\wt(v) = (1-v) z_{i}^{-1}$ and $\widetilde{\gamma} (a_{i,j}) = 1-q^{-1}$, we let $\widetilde{\gamma} (a_{i,j})$ correspond to $\wt(v)/z_{i}^{-1}$ after setting $q^{-1}$ equal to the parameter $v$. Each vertex in row $\overline{i}$ strictly between columns $b_{i, j-1}$ and $a_{i,j}$ (resp., $a_{i,j}$ and $b_{i,j}$) has a weight of $\smash{ z_{i}^{-1} }$ (resp., $1$). The number of factors of $z_{i}^{-1}$ excluded is $\smash{ b_{i, j-1} - a_{i,j} }$.

If $j=r$, the rows are shown on either side of Figure~\ref{fig:a_neither1_divisible2}, where $a=u_{i,r}$. The same correspondence is taken. The left (resp., right) side of the figure is for the case when $b_{i,r} \neq 0$ (resp., $b_{i,r} = 0$). Each vertex in row $\overline{i}$ strictly between columns $b_{i,r-1}$ and $a_{i,r}$ (resp., $a_{i,r}$ and $b_{i,r}$) has a weight of $z_{i}^{-1}$ (resp., $1$). So the number of factors of $\smash{ z_{i}^{-1} }$ excluded is $\smash{ b_{i, r-1} - a_{i,r} }$.

\begin{figure}[H]
\makebox[\textwidth][c]{%
%
}
\caption{}
\label{fig:a_neither1_divisible}
\end{figure}

If $b_{i,r} \neq 0$, each vertex in row $\overline{i}$ to the right of column $b_{i,r}$, including the vertex in the bend, has a weight of $z_{i}^{-1}$. (The vertex in the other type of bend, which appears on the right side of Figure~\ref{fig:a_neither1_divisible2}, was dealt with previously.) As in Case 1, the number of these factors of $z_{i}^{-1}$ is $b_{i,r}$, and this number will be included in our total after Case 4.

In general, $\widetilde{\gamma}(a_{i,j}) \longleftrightarrow \wt(v)/z_{i}^{-1}$. The number of factors of $\smash{ z_{i}^{-1} }$ excluded is $b_{i, j-1} - a_{i, j}$.  In addition, there are $b_{i,r}$ factors of $z_{i}^{-1}$ excluded.

\begin{figure}[H]
\makebox[\textwidth][c]{%
%
}
\caption{}
\label{fig:a_neither1_divisible2}
\end{figure}

\textbf{Case 4:} $b_{i,j} < a_{i,j} < b_{i,j-1}$ and $n \nmid u_{i,j}$. Then $\wt(v) = 0$ and $\widetilde{\gamma} (a_{i,j}) = 0$. The number of factors of $\smash{ z_{i}^{-1}}$ excluded is the same as before.

We have taken into account the weights of all vertices in row $\overline{i}$. In summary, we have the following correspondence:
\begin{equation}
\label{E:association_a}
	\widetilde{\gamma}( a_{i,j} )
	\quad \longleftrightarrow \quad
	\begin{cases}
	q^{-1}g(u_{i,j})
	&\text{if $a_{i,j} =b_{i,j-1}$,}\\
	1
	&\text{if $a_{i,j} = b_{i,j}$,}\\
	1-v
	&\text{if $b_{i,j} < a_{i,j} < b_{i,j-1}$ and $n \mid u_{i,j}$,}\\
	0	
	&\text{if $b_{i,j} < a_{i,j} < b_{i,j-1}$ and $n \nmid u_{i,j}$.}
	\end{cases}
\end{equation}

Up to now, we have obtained from $P$ the Boltzmann weights of all vertices in $\mathfrak{s}$ except for those in row $\overline{r}$ (and for the vertex in one type of bend connected to this row). We now consider this row. Only $a_1$-, $b_2$-, and $c_2$-vertices can appear in row $\overline{r}$. Since the weight of any such $a_1$- or $b_2$-vertex is $1$ after excluding factors of $z_{r}^{-1}$, we need only look at the weight of any $c_2$-vertex in the row. Let $v$ be such a vertex, which is in column $b_{r,r}$, and let $c$ be the charge at $v$. If $b_{r,r} = a_{r-1,r}$, then $\wt(v) = 1$. If $b_{r,r} \neq a_{r-1,r}$, then there is a $c_1$-vertex $v'$ in row $r$ with charge $c$. In either case, omitting $\wt(v)$ does not change $\wt (\mathfrak{s})$: if $\wt(v) \neq 1$, then $\wt(v) = 0$, but then $\wt(v') = 0$ also. So the only contributions to $\wt (\mathfrak{s})$ from row $\overline{r}$ are factors of $z_{r}^{-1}$. One can verify that there are $b_{r,r}$ such factors (the argument is similar to those made in Cases 1--4 for the entry $a_{i,j}$).

We now count all the factors of both $z_{i}$ and $z_{i}^{-1}$ in $\wt(\mathfrak{s})$ that have been excluded.

\begin{lemma}
\label{thm:count}
Let $v_{i,j}$, $w_{i,j}$, and $u_{i,j}$ be defined as in~\eqref{E:uvw}. Let $i \in \{ 1, \ldots, r \}$.
\begin{enumerate}
\item The product of all the $z_i$ that are factors of\/ $\wt(\mathfrak{s})$ is $z_{i}^{ v_{i,r} }$. 
\item The product of all the $z_{i}^{-1}$ that are factors of\/ $\wt(\mathfrak{s})$ is $z_{i}^{ w_{i,i} }$.
\item The product of all the $z_i$ and $z_{i}^{-1}$ that are factors of\/ $\wt(\mathfrak{s})$ is $z_{i}^{u_{i,i}}$.
\end{enumerate} 
\end{lemma}
\begin{proof}[Proof of Lemma~\ref{thm:count}]
Based on Cases 1--4 above for an entry $b_{i,j}$, the number of factors of $z_i$ is $\sum_{k=i}^{r} (a_{i-1,k} - b_{i,k})$, which equals $v_{i,r}$. This proves \textbf{(a)}. Based on Cases 1--4 above for an entry $a_{i,j}$ and on the discussion that precedes this lemma, the number of factors of $z_{i}^{-1}$ is $\sum_{k=i}^{r} (b_{i,k} - a_{i,k+1})$, which equals $- w_{i,i}$ (here $a_{i,r+1} = 0$). This proves \textbf{(b)}. Then \textbf{(c)} follows: the product of all such $z_i$ and $\smash{ z_{i}^{-1 \vphantom{-v_{i,r}}} }$ is $\smash{ z_{i}^{ v_{i,r} \vphantom{-}}  (z_{i}^{-1 \vphantom{-v_{i,r}}}) { \vphantom{z_{i}^{v_{i,r}}} }^{-w_{i,i}} }$, which equals $\smash{ z_{i}^{u_{i,i}} }$. 
\end{proof}

Set $y_i \vcentcolon = |p|^{-2s_i}$ for every $i$. Then
\begin{align*}
	\mathcal{Z}_{\Psi}(\mathbf{s}; p^{\boldsymbol{\ell}})
	&=
	\sum_{\mathbf{k} = (k_1,\ldots,k_r)}
	\dfrac{H^{(n)} (p^{\mathbf{k}}; p^{\boldsymbol{\ell}}) \Psi (p^{\mathbf{k}})}{|p|^{2 k_1 s_1} \ldots |p|^{2 k_r s_r}} \\
	&=
	\sum_{\mathbf{k} = (k_1,\ldots,k_r)}
	\widetilde{H}^{(n)} (p^{\mathbf{k}}; p^{\boldsymbol{\ell}}) \Psi (p^{\mathbf{k}}) (qy_{1})^{k_1} \ldots (qy_{r})^{k_r} ,
\end{align*}
where each sum ranges over the finitely many $\mathbf{k}$ such that $\smash{ H^{(n)} (p^{\mathbf{k}}; p^{\boldsymbol{\ell}}) }$ has nonzero support for fixed $\boldsymbol{\ell}$ by~\eqref{E:prime_coef}. Make the change of variables $q y_1 \mapsto x_{1}^2$, $q y_2 \mapsto x_{1}^{-1}x_2$, \ldots, $q y_r \mapsto x_{r-1}^{-1}x_r$. Then
\begin{align*}
	\mathcal{Z}_{\Psi}(\mathbf{s}; p^{\boldsymbol{\ell}}) 
	&=
	\sum_{\mathbf{k} = (k_1,\ldots,k_r)} \widetilde{H}^{(n)} (p^{\mathbf{k}}; p^{\boldsymbol{\ell}}) \Psi (p^{\mathbf{k}}) x_{1}^{\wt_1 + L_1} \ldots x_{r}^{\wt_r + L_r} \\
	&=
	x_{1}^{L_1} \ldots x_{r}^{L_r} \sum_{\mathbf{k} = (k_1,\ldots,k_r)} \Biggl[  \Psi (p^{\mathbf{k}}) \sum_{P \colon \mathbf{k}(P)=\mathbf{k}} \widetilde{G}(P) x_{1}^{\wt_1} \ldots x_{r}^{\wt_r} \Biggr] \\
	&=
	x_{1}^{L_1} \ldots x_{r}^{L_r} \sum_{\mathbf{k} = (k_1,\ldots,k_r)} \Biggl[  \Psi (p^{\mathbf{k}}) \sum_{P \colon \mathbf{k}(P)=\mathbf{k}} \widetilde{G}(P) x_{1}^{u_{r,r}} \ldots x_{r}^{u_{1,1}} \Biggr].
\end{align*}
If we write as an equality the correspondence $\longleftrightarrow$ given in the proof of Theorem~\ref{T:sym_pattern_bij}, it then follows from~\eqref{E:association_b},~\eqref{E:association_a}, and Lemma~\ref{thm:count} that
\[
	\widetilde{G}(P) = \dfrac{\wt(\mathfrak{s})}{z_{1}^{u_{1,1}} \ldots z_{r}^{u_{r,r}}} \, .
\]
Make the change of variables $x_i \mapsto z_{r-i+1}$ for every $i$. Then
\[
	\mathcal{Z}_{\Psi}(\mathbf{s}; p^{\boldsymbol{\ell}}) = 
	z_{1}^{L_r} \ldots z_{r}^{L_1} \sum_{\mathbf{k} = (k_1,\ldots,k_r)} \Biggl[  \Psi (p^{\mathbf{k}}) \sideset{}{'}\sum_{\mathfrak{s} \in \mathfrak{S}_{\lambda}} \wt (\mathfrak{s}) \Biggr],
\]
where the primed sum ranges over the $\mathfrak{s}$ in $\mathfrak{S}_{\lambda}$ that correspond to those $P$ in $\GT_{\mathrm{str}} (\lambda + \rho)$ satisfying $\mathbf{k}(P) = \mathbf{k}$. This proves the theorem.
\end{proof}


\section{Functional Equations for the Partition Function}
\label{sec:Function_Eqns_Partition}

In this section, we prove two functional equations involving $Z$ that demonstrate the action of the Weyl group of type B/C acting on $z_1$,~\ldots,~$z_r$ by the transposition $z_i \leftrightarrow z_{i+1}$ and by the transformation $z_{r} \leftrightarrow z_{r}^{-1}$.

Let $\mathbf{z} = (z_1, \ldots, z_r) \in \mathbb{C}^r$ and $\mathbf{c} = (c_1, \ldots, c_r) \in \mathbb{Z}^r$, where $c_i \in [0,n-1]$ for each~$i$. Denote by $Z_{\lambda} (\mathbf{z}; \mathbf{c})$ the partition function of the model with boundary conditions determined by $\lambda$, such that for each admissible state, the leftmost edges in rows $\overline{1}$,~\ldots,~$\overline{r}$ have charges congruent modulo $n$ to $c_1$,~\ldots,~$c_r$, respectively.

\begin{theorem}
\label{T:transposition}
Let $i$ be a row spectral index with $i < r$, and put $j = i+1$. Let $s_i$ be the simple reflection $i \leftrightarrow j$. Let $e \equiv c_i - c_j \pod{n}$ with $e \in [0, n-1]$. If $c_i \neq c_{j}$, then 
\begin{equation}
\label{E:identity1}
\begin{split}
	(1-v) z_{i}^{n-e} z_{j}^{e}
	\, Z 
	\left(\!
	\mathord{%
%
	}
	\,\right) .
\]
\end{theorem}

\noindent The row labels indicate the spectral parameters for the Boltzmann weights of the vertices in the rows.

\begin{proof}
For brevity, we will write $Z_{\lambda}$ simply as $Z$. Let $B$, $I_1$, $I_2$ be the configurations in Figure~\ref{fig:Alpha}, where $I_2$ is obtained by attaching $B$ to $I_1$, and where the other rows of $I_1$ and $I_2$ match, alternate in ice type, and satisfy the following: for each $k$ besides $i$ and $j$, the leftmost edges in rows $\smash{(k,\overline{k} \, )}$ have fixed decorated spins $(+0, + c_k)$.

\begin{figure}[htb]
\centering
\makebox[\textwidth][c]{%
%
}
\caption{}
\label{fig:Alpha}
\end{figure}

If $c_i \neq c_j$, the only admissible states of $B$ are the first two in Figure~\ref{fig:states_caduceusB}. If $c_i = c_j$, the only admissible state of $B$ is the last one in Figure~\ref{fig:states_caduceusB}. Each admissible state of $I_2$ yields unique admissible states of $B$ and $I_1$, where the rightmost decorations on $B$ match the leftmost decorations on $I_1$. Similarly, each admissible state of $I_1$, with leftmost decorations in rows $(i, \overline{i} \,)$ and $(j, \overline{j} \,)$ matching the rightmost decorations in an admissible state of $B$, yields a unique admissible state of $I_2$. It follows that
\begin{align*}
	Z(I_2) &= 
	(z_{j}^{-n} - v^n z_{i}^n) 
	(z_{i}^{-n} - v z_{j}^n) 
	(z_{i}^n - v z_{j}^n)
	\\
	&\qquad {}\times
	\begin{cases}
	\wt_{\Gamma \Gamma} 
	\biggl(\!\!
	\mathord{%

	}
	\!\!\biggr) \, Z(\mathbf{z}; \mathbf{c})
	&\text{if $c_i = c_j$,}
	\end{cases}
\end{align*}
where $\wt_{\Gamma \Gamma}$ denotes the Boltzmann weights of the $\Gamma \Gamma$ $R$-vertices in Table~\ref{tab:braid_weights} (based on Figure~\ref{fig:states_caduceusB}, the spectral parameters are $z_{i}^{-1}$ and $z_{j}^{-1}$ rather than $z_i$ and $z_j$).

\begin{figure}[htb]
\centering
\makebox[\textwidth][c]{%
%
%
}
\caption{}
\label{fig:states_caduceusB}
\end{figure}

We apply the YBE repeatedly to $I_2$ to push the four vertices in $B$ to the right. Doing so interchanges rows $(i, \overline{i}\,)$ and $(j, \overline{j}\,)$ of $I_2$, does not affect the partition function, and yields the configuration $I_3$ in Figure~\ref{fig:pushed}; so $Z(I_2) = Z(I_3)$.

\begin{figure}[htb]
\makebox[\textwidth][c]{%
%
%
}
\caption{$I_3$}
\label{fig:pushed}
\end{figure}

To relate $Z(I_3)$ to a symplectic-ice state, we use the configurations $I_4$, $I_5$, $I_6$ in Figure~\ref{fig:sandman} (the decorated spins $\varepsilon_1$,~\ldots,~$\varepsilon_4$ are arbitrary), where $I_4$ and $I_5$ together are $I_3$ and where $I_6$ consists of two $\Delta \Gamma$-bends. 

\begin{figure}[H]
\centering
\makebox[\textwidth][c]{%
%
%
}
\caption{}
\label{fig:sandman}
\end{figure}

\begin{lemma}[Caduceus relation]
\label{thm:caduceus}
Suppose $\varepsilon_1$,~\ldots,~$\varepsilon_4$ in Figure~\ref{fig:sandman} are fixed. The ratio $Z(I_5)/\wt(I_6)$ is independent of these decorated spins and equals
\begin{equation}
\label{E:cad_result}
	(z_{j}^{-n} - v_{\vphantom{i}}^n z_{i}^n)
	(z_{i}^{-n} - v_{\vphantom{j}} z_{j}^n) 
	(z_{i}^n - v_{\vphantom{j}} z_{j}^n)
	(z_{i}^{-n} - v_{\vphantom{j}} z_{j}^{-n}).
\end{equation}
\end{lemma}

\noindent See Appendix~\ref{app:appendixCF} for the proof.

Let $P$ be the expression in~\eqref{E:cad_result}. Then
\[
	Z(I_3) 
	= \sum_{\varepsilon_k} Z(I_4) \, Z(I_5)
	= P \sum_{\varepsilon_k} Z(I_4) \, \wt (I_6)
	= P \cdot Z (s_i (\mathbf{z}); \mathbf{c}),
\]
where each sum ranges over all $\varepsilon_1$,~\ldots,~$\varepsilon_4$. Thus, $Z (I_2) = P \cdot Z (s_i (\mathbf{z}); \mathbf{c})$. Then~\eqref{E:identity1} follows after canceling common factors of $Z(I_2)$ and $P$ and multiplying by $z_{i}^n z_{j}^n$.
\end{proof}

\begin{theorem}
\label{thm:inverse}
Let $\mathbf{z} = (z_1, \ldots, z_r) \in \mathbb{C}^r$ and $\mathbf{c} = (c_1, \ldots, c_r) \in \mathbb{Z}^r$, where $c_i \in [0,n-1]$ for every $i$. Put $N = \lambda_1 + r + 1$. Then
\begin{equation}
\label{E:result_inverse}
\begin{split}
	(1-v) z_{r}^{n-2e}
	\, Z
	\left(\!
	\mathord{%
%
	}
	\,\right),
\end{split}
\end{equation}
where $e \equiv c_r - N \pod{n}$, $e \in [0, n-1]$, $C = c_r - 2e$, and $C \in [0, n-1]$. 
\end{theorem}

\noindent Since the proof is lengthy and tedious, we give a brief outline:
\begin{enumerate}[label=\textbf{\arabic*.}]
\item Attach a $\Delta \Gamma$ $R$-vertex $B_1$ to the bottom two rows. Push $B_1$ through via the YBE. Use a ``fish relation'' (Lemma~\ref{thm:Fish1}).
\item Change the ice type of the bottom row from $\Delta$ to $\Gamma$. Show the resulting partition function is invariant under the change.
\item Attach a $\Gamma \Gamma$ $R$-vertex $B_2$ to the bottom two rows. Push $B_2$ through via the YBE. Use a second fish relation (Lemma~\ref{thm:Fish2}).
\item Change the ice type of the bottom row from $\Gamma$ to $\Delta$. Show the resulting partition function is invariant under the change.
\item Attach a $\Gamma \Delta$ $R$-vertex $B_3$ to the bottom two rows. Push $B_3$ through via the YBE. Use a third fish relation (Lemma~\ref{thm:Fish3}).
%
\end{enumerate}

Throughout the proof, the bottom two rows of configurations will often interchange positions, which amounts to the interchange $z_{r} \leftrightarrow z_{r}^{-1}$ in the appropriate Boltzmann weights. To keep track of such interchanges, at times we will use the following notation. For every configuration $I$, its partition function will often be denoted by $Z(I; a,b)$ if for each admissible state of $I$, the leftmost edges in rows $\overline{1}$,~\ldots,~$\overline{r-1}$ have charges congruent modulo $n$ to $c_1$,~\ldots,~$c_{r-1}$, respectively, while the leftmost edges in the bottom two rows have charges congruent to $a$ and $b$, with $b$ for the bottom row. To emphasize when $r$ and $\overline{r}$ have been interchanged in a symplectic-ice model, we will often write $\overline{Z} (I; a,b)$; if $a$ and $b$ are clear from context, we will write $Z(I)$ or $\overline{Z}(I)$.

\begin{proof}
Let $B_1$, $I_1$, $I_2$ be the configurations in Figure~\ref{fig:First_Attach}, where $I_2$ is obtained by attaching $B_1$ to $I_1$, and where the other rows of $I_1$ and $I_2$ match, alternate in ice type, and satisfy the following: for each $k$ with $k < r$, the leftmost edges in rows $\smash{(k,\overline{k} \, )}$ have fixed decorated spins $(+0, + c_k)$.

\begin{figure}[htb]
\centering
\makebox[\textwidth][c]{%
%
}
\caption{}
\label{fig:First_Attach}
\end{figure}

Configuration $B_1$ has only one admissible state, namely the one with the decorations on the right matching those on the left but in opposite order. Clearly, each admissible state of $I_2$ yields a unique state of $I_1$ for which the rightmost decorations on $B_1$ match the leftmost decorations on $I_1$. Similarly, each admissible state of $I_1$, for which the leftmost decorations in rows $(r, \overline{r})$ are $(c_r, 0)$, yields a unique state of $I_2$. It follows that
\begin{equation}
\label{E:glasses}
	Z(I_2) 
	= (z_{r}^{-n} - v^n z_{r}^n) 
	\, Z(I_1; 0, c_r).
\end{equation}

We apply the YBE repeatedly to $I_2$ to push the vertex in $B_1$ to the right. Doing so interchanges rows $r$ and $\overline{r}$ of $I_2$, does not affect the partition function, and yields the configuration $I_3$ in Figure~\ref{fig:First_YBE}; so $Z(I_2) = Z(I_3)$.

\begin{figure}[htb]
\makebox[\textwidth][c]{%
%
}
\caption{$I_3$}
\label{fig:First_YBE}
\end{figure}

To relate $Z(I_3)$ to a symplectic-ice state, we use the configurations $I_4$, $I_5$, $I_6$ in Figure~\ref{fig:Canada}, where $I_4$ and $I_5$ together are $I_3$. 

\begin{figure}[htb]
\centering
\makebox[\textwidth][c]{%
%
%
}
\caption{}
\label{fig:Canada}
\end{figure}

\noindent We give the flipped $\Gamma \Delta$-bends connecting rows $(\overline{r}, r)$ the following weights.

\begin{assumption}
The Boltzmann weights of the flipped\/ $\Gamma \Delta$-bends connecting rows $(\overline{r}, r)$ are 
\begin{equation}
\label{E:nebraska}
	\wt_{\Gamma \Delta}
	\biggl(\!
	%
	\biggr)
	=
	z_{r},
\end{equation}
for every nonnegative integer $c$. 
\end{assumption}

\begin{lemma}[Fish relation, type $\boldsymbol{\Delta \Gamma}$]
\label{thm:Fish1}
Suppose $\varepsilon_1$ and $\varepsilon_2$ in Figure~\ref{fig:Canada} are fixed. The ratio $Z (I_5) / \wt(I_6)$ is independent of these decorated spins and equals $z_{r}^{-n} - v_{\vphantom{1}}^{n} z_{r}^{n}$.
\end{lemma}

\noindent See Appendix~\ref{app:appendixCF} for the proof.

Let $I_7$ be the configuration obtained by attaching $I_6$ to $I_4$. Then  
\[
	Z(I_3)
	= \sum_{\varepsilon_k} Z (I_4) 
	\, Z (I_5)
	= (z_{r}^{-n} - v^n z_{r}^n)
	\sum_{\varepsilon_k}^{\phantom{x}} 
	Z (I_4) \, \wt (I_6) 
	= (z_{r}^{-n} - v^n z_{r}^n) \, 
	Z (I_7; c_r, 0),
\]
where each sum ranges over all $\varepsilon_1$ and $\varepsilon_2$. It follows from this and~\eqref{E:glasses} that 
\begin{equation}
\label{E:Lagrange}
	Z(I_1; 0, c_r) = 
	Z (I_7; c_r, 0).
\end{equation}

Since all spins along the bottom boundary are $+$, the only admissible configurations in row $r$ are of types $a_1$, $b_2$, and $c_2$; moreover, the kinds of functions appearing as weights in $a_1$, $b_2$, and $c_2$ of $\Delta$ and $\Gamma$ ice are the same. Thus, it seems reasonable that we can change row $r$ from $\Delta$ ice to $\Gamma$ ice without affecting $Z(I_7)$.


Consider an admissible state $\mathfrak{s}$ of $I_7$. Row $r$ (resp., $\overline{r}$) has at most one (resp., exactly one) vertex with a vertical spin of $-$. There are three possibilities:
\begin{itemize}
\item Case 1: row $r$ has no vertex with a vertical spin of $-$, but row $\overline{r}$ has a unique vertex $\overline{v}$ with a vertical spin of $-$.
\item Case 2: row $r$ has a unique vertex $v$ with a vertical spin of $-$, row $\overline{r}$ has a unique vertex $\overline{v}$ with a vertical spin of $-$, and $v$ and $\overline{v}$ are in different columns. 
\item Case 3: row $r$ has a unique vertex $v$ with a vertical spin of $-$, row $\overline{r}$ has a unique vertex $\overline{v}$ with a vertical spin of $-$, and $v$ and $\overline{v}$ are in the same column. 
\end{itemize}
We analyze each of these cases below to see the effects of changing row $r$ from $\Delta$ ice to $\Gamma$ ice. After the change in ice type, we assign the vertex in the bend a charge of $0$ and let charge propagate along rows $(\overline{r}, r)$ from the bend. We denote by $c$ and $\overline{c}$ the column numbers of vertices $v$ and $\overline{v}$, respectively. In Figures~\ref{fig:Case1_change}--\ref{fig:Case3_change} below, we will indicate by $\square$ the vertices $v$ and $\overline{v}$. We denote by $N$ the number
\[
	N = \lambda_1 + r + 1,
\]
which is one more than the number of columns.

\textbf{Case 1:} The spins in rows $(\overline{r}, r)$ of $\mathfrak{s}$ are shown on the left side of Figure~\ref{fig:Case1_change}. Here $c_r$ satisfies $c_r \equiv N - \overline{c} \pod{n}$. Before changing row $r$, $\wt(\overline{v}) = 1$.

\begin{figure}[htb]
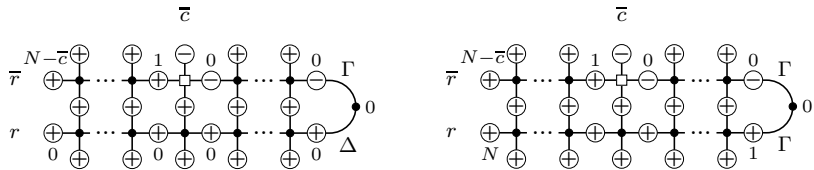

\makebox[\textwidth][c]{%
%
}
\caption{Case 1 for changing row $r$ from $\Delta$ ice (left) to $\Gamma$ ice (right).}
\label{fig:Case1_change}
\end{figure}

After changing row $r$, we assign the vertex in the bend a charge of $0$ (done already) and let charge propagate along rows $(\overline{r}, r)$ from the bend, as shown on the right side of Figure~\ref{fig:Case1_change}. The leftmost charges in rows $\overline{r}$ and $r$ are $N-\overline{c}$ and $N$, respectively. Assuming the weight of the bend has not changed, $\wt (\mathfrak{s})$ remains the same since only $a_1$-vertices appear in row $r$, each of which has a weight independent of charge.

\textbf{Case 2:} The spins in rows $(\overline{r}, r)$ of $\mathfrak{s}$ are shown on the left side of Figure~\ref{fig:Case2_change}. Here $c_r$ satisfies $c_r \equiv N - \overline{c} + 2c \pod{n}$. Both $v$ and $\overline{v}$ are $c_2$-vertices. Before changing row $r$, $\wt(v) = 1$ and $\wt (\overline{v}) = \delta (2c)$. For $\delta(2c)$ to be nonzero, we must have $2c \equiv 0$, hence $c \equiv 0$, as $n$ is odd.

\begin{figure}[htb]
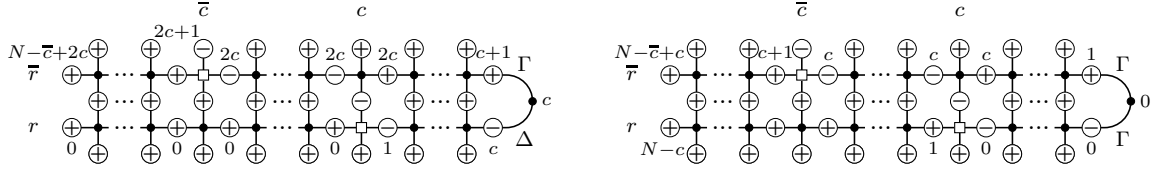

\makebox[\textwidth][c]{%
%
}
\caption{Case 2 for changing row $r$ from $\Delta$ ice (left) to $\Gamma$ ice (right).}
\label{fig:Case2_change}
\end{figure}

After changing row $r$, we assign the vertex in the bend a charge of $0$ and let charge propagate along rows $(\overline{r}, r)$ from the bend, as shown on the right side of Figure~\ref{fig:Case2_change}. The leftmost charges in rows $\overline{r}$ and $r$ are $N - \overline{c} + c$ and $N-c$, respectively. Assuming the weight of the bend has not changed, $\wt(\mathfrak{s})$ remains the same.

\textbf{Case 3:} The spins in rows $(\overline{r}, r)$ of $\mathfrak{s}$ are shown on the left side of Figure~\ref{fig:Case3_change}. Here $c_r$ satisfies $c_r \equiv N + c \pod{n}$. Before changing row $r$, $\wt(v) = 1$ and $\wt(\overline{v}) = g(2 c)$.

\begin{figure}[htb]
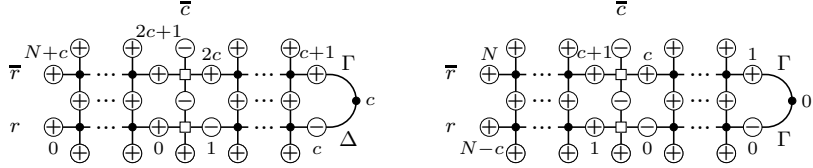

\makebox[\textwidth][c]{%
%
}
\caption{Case 3 for changing row $r$ from $\Delta$ ice (left) to $\Gamma$ ice (right).}
\label{fig:Case3_change}
\end{figure}

After changing row $r$, we assign the vertex in the bend a charge of $0$ and let charge propagate along rows $(\overline{r}, r)$ from the bend, as shown on the right side of Figure~\ref{fig:Case3_change}. The leftmost charges in rows $\overline{r}$ and $r$ are $N$ and $N-c$, respectively. However, $\wt(\overline{v})$ is now $g(c)$ rather than $g(2 c)$ as before. Assume the weight of the bend has not changed. Since $\wt(\overline{v})$ has changed, we need to make one alteration: \textit{we replace every occurrence of $g(a)$ in rows $r$ and\/ $\overline{r}$ by $g(2a)$, for any $a \in \mathbb{Z}$.} Doing so affects neither the results in Cases 1 and 2 nor the YBE, since the identities defining $g$ will remain true. Then $\wt(\overline{v})$ before the change matches $\wt (\overline{v})$ after. One can verify that $\wt (\mathfrak{s})$ remains the same.

Thus, we can change the ice type of row $r$ from $\Delta$ to $\Gamma$ without changing $Z (I_7)$ if we do the alteration to the function $g$ in rows $r$ and $\overline{r}$ as described in Case 3. To emphasize this alteration, we will write $\wt^{\ast}$ and $Z^{\ast}$.

To summarize, we make the following assumption and state a lemma.

\begin{assumption}
The Boltzmann weights of the flipped\/ $\Gamma \Gamma$-bends connecting rows $(\overline{r}, r)$ are the same as those of the flipped\/ $\Gamma \Delta$-bends in~\eqref{E:nebraska}:
\[
	\wt_{\Gamma \Gamma}
	\biggl(\!

	\biggr)
	=
	z_r.
\]
Furthermore, the weights are spectrally dependent:
\[
	\wt_{\Gamma \Gamma}
	\biggl(\!
	%
	\biggr)
	=
	z_{r}^{-1}.
\]
\end{assumption}


\begin{lemma}
\label{thm:sets_ice}
Continue using the notation above. There exists a bijection between the set of admissible states of $I_7$ and the set of admissible states of the second configuration shown in~\eqref{E:bijection_summary} below, where: the configurations differ only in rows $\overline{r}$ and $r$; the Boltzmann weights in rows $(\overline{r}, r)$ of the second configuration have been altered as described above; for each state of the second configuration, if $a$ and $b$ are the leftmost charges in rows $\overline{r}$ and $r$, respectively, then $a - b \equiv c_r - N \pod{n}$. All other rows in the configurations alternate in ice type as usual.
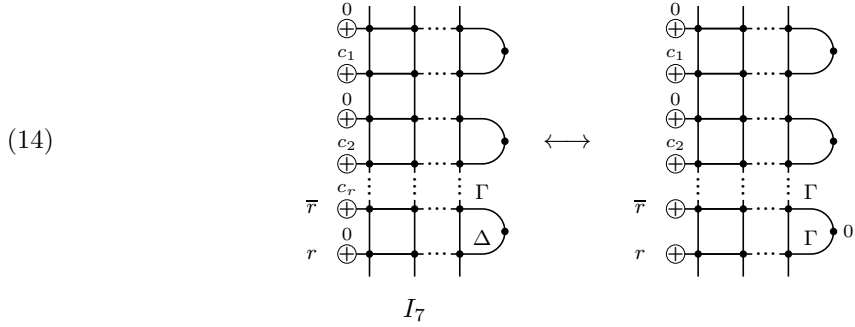
\begin{equation}
\label{E:bijection_summary}
\mathord{%
%
}
\end{equation}
\end{lemma}

\begin{proof}[Proof of Lemma~\ref{thm:sets_ice}]
Take a state of $I_7$. Change its bottom row from $\Delta$ ice to $\Gamma$ ice, assign the vertex in the bend connecting rows $(\overline{r}, r)$ a charge of $0$, and let charge propagate in these rows from that vertex. Let $a$ and $b$ be the leftmost charges in rows $\overline{r}$ and $r$, respectively. Cases 1--3 above describe these charges.
\begin{itemize}
\item In Case 1, we have $a = N -\overline{c}$ and $b = N$. So $a - b = - \overline{c} \equiv c_r - N \pod{n}$. 
\item In Case 2, we have $a = N -\overline{c} + c$ and $b = N - c$. So $a - b = 2c - \overline{c} \equiv c_r - N \pod{n}$. 
\item In Case 3, we have $a = N$ and $b = N - c$. So $a - b = c \equiv c_r - N \pod{n}$. 
\end{itemize}
The result is a state of the second configuration.

Take a state of the second configuration. Let $a$ and $b$ be the leftmost charges in rows $\overline{r}$ and $r$, respectively. Then $a$ and $b$ must match those leftmost charges shown on the right side of either Figure~\ref{fig:Case1_change},~\ref{fig:Case2_change}, or~\ref{fig:Case3_change}. One can verify that $a - b \equiv c_r - N \pod{n}$. We obtain a state of $I_7$ by reversing the steps above.
\end{proof}

Let $B_2$, $I_8$, $I_9$ be the configurations in Figure~\ref{fig:Second_Attach}, where: $I_9$ is obtained by attaching $B_2$ to $I_8$, $a$ and $b$ are some fixed decorations in $[0, n-1]$ satisfying $a - b \equiv c_r - N \pod{n}$, and the other rows of $I_8$ and $I_9$ match those of $I_1$.

\begin{figure}[htb]
\centering
\makebox[\textwidth][c]{%
%
}
\caption{}
\label{fig:Second_Attach}
\end{figure}

If $a \neq b$, then the only admissible states of $B_2$ are the first two in Figure~\ref{fig:states_fishB2}. If $a = b$, then the only admissible state of $B_2$ is the last one in Figure~\ref{fig:states_fishB2}. Each admissible state of $I_9$ yields unique states of $B_2$ and $I_8$ for which the rightmost decorations on $B_2$ match the leftmost decorations on $I_8$. Similarly, each admissible state of $I_8$, with leftmost decorations in rows $(r, \overline{r})$ matching the rightmost decorations in an admissible state of $B_2$, yields a unique state of $I_9$. It follows from this, Lemma~\ref{thm:sets_ice}, and~\eqref{E:Lagrange} that
\begin{align}
	Z^{\ast} (I_9) &= 
	\begin{cases}
	\wt_{\Gamma \Gamma}^{\ast}
	\biggl(\!\!
	\mathord{%
%
}
\caption{}
\label{fig:states_fishB2}
\end{figure}

We apply the YBE repeatedly to $I_9$ to push the vertex in $B_2$ to the right. Doing so interchanges rows $r$ and $\overline{r}$ of $I_9$, does not affect the partition function, and yields a configuration $I_{10}$; so $Z^{\ast} (I_9) = Z^{\ast} (I_{10})$. To relate the partition function of twisted ice $Z^{\ast} (I_{10})$ to a symplectic-ice state, we use the configurations $I_{11}$, $I_{12}$, $I_{13}$ in Figure~\ref{fig:Lewisburg}, where $I_{11}$ and $I_{12}$ together are $I_{10}$. 

\begin{figure}[htb]
\centering
\makebox[\textwidth][c]{%
%
}
\caption{}
\label{fig:Lewisburg}
\end{figure}

\begin{lemma}[Fish relation, type $\boldsymbol{\Gamma \Gamma}$]
\label{thm:Fish2}
Suppose $\varepsilon_1$ and $\varepsilon_2$ in Figure~\ref{fig:Lewisburg} are fixed. The ratio $Z^{\ast} (I_{12}) / \wt^{\ast} (I_{13})$ is independent of these decorated spins and equals $z_{r}^{-n} - v z_{r}^{n}$.
\end{lemma}

\noindent See Appendix~\ref{app:appendixCF} for the proof.

Let $I_{14}$ be the configuration obtained by attaching $I_{13}$ to $I_{11}$. Then 
\[
	Z^{\ast} (I_{10})
	= \sum_{\varepsilon_k} Z^{\ast} (I_{11}) 
	\, Z^{\ast} (I_{12})
	= (z_{r}^{-n} - v z_{r}^{n})
	\sum_{\varepsilon_k}^{\phantom{x}} 
	Z^{\ast} (I_{11}) \, \wt^{\ast} (I_{13}) 
	= (z_{r}^{-n} - v z_{r}^{n}) \, 
	Z^{\ast} (I_{14}; a, b),
\]
so that
\begin{equation}
\label{E:Newton}
	Z^{\ast} (I_{9}) 
	= (z_{r}^{-n} - v z_{r}^{n}) \, 
	Z^{\ast} (I_{14}; a, b).
\end{equation}

We now change the bottom row in $I_{14}$ from $\Gamma$ ice to $\Delta$ ice. We also revert to our original Boltzmann weights. This change in ice type, which changes the leftmost (fixed) decorations in the bottom row back to $0$ for each admissible state of $I_{14}$, yields the configuration $I_{15}$ given in Figure~\ref{fig:Third_Attach}, all of whose admissible states have leftmost charges in rows $\overline{r}$ and $r$ congruent modulo $n$ to $a$ and $b$, respectively; so $Z^{\ast}(I_{14}; a, b) = Z(I_{15}; c_r, 0)$. This change will not affect $Z^{\ast} (I_{14})$ if we assume the weights of the bend connecting rows $(r, \overline{r})$ remain the same.

\begin{assumption}
The Boltzmann weights of the $\Gamma \Delta$-bends connecting rows $(\overline{r}, r)$ are
\begin{equation}
\label{E:Bundaberg}
	\wt_{\Gamma \Delta}
	\biggl(\!
	\begin{tikzpicture}[scale=0.5, every node/.style={scale=1.0}, baseline=2.25]
	\draw[black, semithick] 
	(-0.25,0) -- (0.25,0) arc (-90:90:{3*sqrt(2)/10}) 
	-- (-0.25,{3*sqrt(2)/5});
	\tikzstyle{spin}=[draw, circle] 
	\path[black, very thin] 
	(0,{0*3*sqrt(2)/5}) node[inner sep=-0.5pt, 
		fill=white, spin] 
	{\tiny$\boldsymbol{+}$}
	node[left] {\scriptsize$c \, \;$};
	\node[left] at (-1,{1*3*sqrt(2)/5}) {\footnotesize$r$};
	\node[left] at (-1,{0*3*sqrt(2)/5 - 0.12}) {\footnotesize$\smash{\overline{r}}$};
	\tikzstyle{spin}=[draw, circle] 
	\path[black, very thin] 
	(0,{1*3*sqrt(2)/5}) node[inner sep=-0.5pt, 
		fill=white, spin] 
		{\tiny$\boldsymbol{- \vphantom{+}}$}
	node[left] {\scriptsize$c \, \;$};
	\filldraw[black] ({0.25+3*sqrt(2)/10},{3*sqrt(2)/10}) circle(2.25pt);
	\end{tikzpicture}
	\biggr)
	=
	z_{r},
	\qquad
	\wt_{\Gamma \Delta} 
	\biggl(\!
	\begin{tikzpicture}[scale=0.5, every node/.style={scale=1.0}, baseline=2.65]
	\draw[black, semithick] 
	(-0.25,0) -- (0.25,0) arc (-90:90:{3*sqrt(2)/10}) 
	-- (-0.25,{3*sqrt(2)/5});
	\tikzstyle{spin}=[draw, circle] 
	\path[black, very thin] 
	(0,{0*3*sqrt(2)/5}) node[inner sep=-0.5pt, 
		fill=white, spin] 
	{\tiny$\boldsymbol{- \vphantom{+}}$}
	node[left] {\scriptsize$c \, \;$};
	\tikzstyle{spin}=[draw, circle] 
	\path[black, very thin] 
	(0,{1*3*sqrt(2)/5}) node[inner sep=-0.5pt, 
		fill=white, spin] 
		{\tiny$\boldsymbol{+}$}
	node[left] {\scriptsize$c{+}1 \, \;$};
	\node[left] at (-1-0.7,{1*3*sqrt(2)/5}) 
	{\footnotesize$r$};
	\node[left] at (-1-0.7,{0*3*sqrt(2)/5 - 0.12}) 
	{\footnotesize$\smash{\overline{r}}$};
	\filldraw[black] ({0.25+3*sqrt(2)/10},{3*sqrt(2)/10}) circle(2.25pt);
	\end{tikzpicture}
	\biggr)
	=
	z_{r}^{-1},
\end{equation}
for every nonnegative integer $c$.
\end{assumption}

We can change the bottom row in $I_8$ from $\Gamma$ ice to $\Delta$ ice also, where each admissible state of $I_8$ before the change has leftmost charges congruent modulo $n$ to $b$ and $a$, with $a$ (not $b$) for the bottom row. This allows us to get rid of the factors $\overline{Z}^{\ast} (I_8; b, a)$ and $\overline{Z}^{\ast} (I_8; a, a)$ in~\eqref{E:messy}. Then $\overline{Z}^{\ast} (I_8; b, a) = Z(I_1; 0, c_r - 2e)$, whether $a = b$ or $a \neq b$.

Let $B_3$ and $I_{16}$ be the other configurations in Figure~\ref{fig:Third_Attach}, where $I_{16}$ is obtained by attaching $B_3$ to $I_{15}$, and where the other rows of $I_{15}$ and $I_{16}$ match those of $I_1$.

\begin{figure}[htb]
\centering
\makebox[\textwidth][c]{%
%
}
\caption{}
\label{fig:Third_Attach}
\end{figure}

Configuration $B_3$ has only one admissible state. Clearly, each admissible state of $I_{16}$ yields a unique admissible state of $I_{15}$ for which the rightmost decorations on $B_3$ match the leftmost decorations on $I_{15}$. Similarly, each admissible state of $I_{15}$, for which the leftmost decorations in rows $(\overline{r}, r)$ are $(c_r, 0)$, yields a unique admissible state of $I_{16}$. It follows that 
\begin{equation}
\label{E:Evans}
	Z(I_{16}) 
	= (z_{r}^{n} - v z_{r}^{-n}) 
	\, Z(I_{15}; c_r, 0).
\end{equation}

We apply the YBE repeatedly to $I_{16}$ to push the vertex in $B_3$ to the right. Doing so interchanges rows $r$ and $\overline{r}$ of $I_{16}$, does not affect the partition function, and yields a configuration $I_{17}$; so $Z(I_{16}) = Z (I_{17})$. To find $Z (I_{17})$, we use the configurations $I_{18}$, $I_{19}$, $I_{20}$ in Figure~\ref{fig:Erebor}, where $I_{18}$ and $I_{19}$ together are $I_{17}$.

\begin{figure}[htb]
\centering
\makebox[\textwidth][c]{%
%
}
\caption{}
\label{fig:Erebor}
\end{figure}

\begin{lemma}[Fish relation, type $\boldsymbol{\Gamma \Delta}$]
\label{thm:Fish3}
Suppose $\varepsilon_1$ and $\varepsilon_2$ in Figure~\ref{fig:Erebor} are fixed. The ratio $Z(I_{19}) / \wt(I_{20})$ is independent of these decorated spins and equals $z_{r}^{n} - v z_{r}^{-n}$.
\end{lemma}

\noindent See Appendix~\ref{app:appendixCF} for the proof.

We then have 
\[
	Z (I_{17})
	= \sum_{\varepsilon_k} 
	Z (I_{18}) \, Z(I_{19})
	= (z_{r}^{n} - v z_{r}^{-n})
	\sum_{\varepsilon_k}^{\phantom{x}} 
	Z (I_{18}) \, \wt (I_{20}) 
	= (z_{r}^{n} - v z_{r}^{-n}) \, \overline{Z} (I_{1}; 0, c_r).
\]
It follows from this and~\eqref{E:Evans} that $Z (I_{15}; c_r, 0) = \overline{Z} (I_{1}; 0, c_r)$. Then~\eqref{E:Newton} yields
\[
	Z^{\ast} (I_9) 
	= (z_{r}^{-n} - v z_{r}^{n})
	\, \overline{Z} (I_{1}; 0, c_r).
\]
Substituting $(z_{r}^{-n} - v z_{r}^{n}) \, \overline{Z} (I_{1}; 0, c_r)$ for $Z^{\ast} (I_9)$ in~\eqref{E:messy} proves the theorem.
\end{proof}

\section{Algebraic Preliminaries}
\label{sec:Algebraic_Preliminaries}

In Section~\ref{sec:Connections_intertwining_Whittaker}, we will relate the partition function $Z$ of our model to Whittaker functions on metaplectic covers of $\SO (2r+1)$. We will make some conjectures that relate certain structure constants arising from intertwining operators for metaplectic principal series representations to modified Boltzmann weights of the $R$-vertices shown in Table~\ref{tab:Delta_Gamma}. In this section, we introduce some terminology of the algebraic preliminaries that are needed to state these conjectures. All of the results in this section summarize some of the work found in, e.g., McNamara~\cite{McNamara, McNamara2, McNamara3}.

As mentioned in Section~\ref{sec:Partition_Boltzmann_Charge}, $n$ and $r$ are fixed positive integers, with $n$ odd.

Let $F$ be a nonarchimedean local field with valuation ring $\mathfrak{o}_{F}$, uniformizer $\varpi$ (i.e., $\varpi$ is a generator of the maximal ideal of $\mathfrak{o}_F$), and residue field $\frak{o}_{F} / \varpi \mathfrak{o}_{F}$ of cardinality $q$. Let $\mu_n$ be the cyclic group of all $n$th roots of unity in $F$. Let $(\cdot \, , \cdot) \colon F^{\times} \times F^{\times} \to \mu_n$ be the $n$th-power Hilbert symbol, and let $\epsilon \colon \mu_n \to \mathbb{C}^{\times}$ be an embedding; we will omit $\epsilon$ from all of our notation.

Let $G \vcentcolon = G(F)$ be a split reductive algebraic group over $F$ with maximal split torus $T \vcentcolon = T(F)$. Let $B = TU$ be the Borel subgroup of $G$, where $U$ is the unipotent radical of $B$. Let $K \vcentcolon = G(\mathfrak{o}_F)$ be the maximal compact subgroup of $G$.

Let $\Lambda$ be the coweight lattice of $G$. Then $\Lambda$ can be identified with the cocharacter group $X_{\ast}(T)$, which is isomorphic to $\mathbb{Z}^r$. Let $\{e_1,\ldots,e_r\}$ be the standard basis for $\mathbb{Z}^r$, where $e_i$ is the vector $(0,\ldots,0,1,0,\ldots,0)$ with $1$ in the $i$th coordinate. Let $\Phi$ be a reduced root system of type $C_r$, with subsets $\Phi^{+}$ and $\Phi^{-}$ of positive roots and negative roots. Let $W$ be the Weyl group of $\Phi$, where $W$ is generated by simple reflections $s_{\alpha}$ and has long element $w_0$.

Let $\smash{\widetilde{G}} \vcentcolon = \widetilde{G}^{(n)}$ be the metaplectic $n$-fold cover of $G$. In order to have simpler formulas in this section, we make the assumption that $q \equiv 1 \pod{2n}$, so that $F$ contains the cyclic group $\mu_{2n}$ of all $2n$th roots of unity. (For more on these simplifications, see Weissman~\cite{We}.) The group $\smash{\widetilde{G}}$ can be constructed as a central extension of $G$ by $\mu_n$; thus, there is a short exact sequence of topological groups:
\begin{equation}
\label{E:short}
	1
	\longrightarrow
	\mu_n
	\longrightarrow
	\widetilde{G}
	\overset{p}{\longrightarrow}
	G
	\longrightarrow
	1,
\end{equation}
where $\mu_n$ lies in the center of $\smash{ \widetilde{G} }$. As a set, $\smash{\widetilde{G}}$ equals $G \times \mu_n$, and $p$ in~\eqref{E:short} is the natural projection defined by $p(g, \zeta) = g$ for all $g \in G$ and all $\zeta \in \mu_n$. Multiplication in $\smash{ \widetilde{G} }$ depends on our choice of a cocycle $\sigma$ in $\HH^2 (G, \mu_n)$: for all $(g,\zeta)$,~$\smash{ (g',\zeta') \in \widetilde{G} }$,
\[
	(g,\zeta)(g',\zeta') = (gg', \sigma(g,g') \zeta \zeta'). 
\]
Thus, in order to describe $\smash{\widetilde{G}}$, we need to be able to describe $\sigma$. Matsumoto~\cite{Matsumoto} gave formulas that describe such a cocycle, but it remains difficult to evaluate explicitly on arbitrary elements. Fortunately, the central extension $\widetilde{G}$ in~\eqref{E:short} can be constructed by using a $W$-invariant symmetric bilinear form $B \colon \Lambda \times \Lambda \to \mathbb{Z}$ that we may choose to be given by the usual dot product. Let $Q \colon \Lambda \to \mathbb{Z}$ be the quadratic form given by $Q(\mu) \vcentcolon= B(\mu,\mu)/2$ for all $\mu \in \Lambda$, and let
\[
	\Lambda^{(n)} = \{ \mu \in \Lambda \mid \text{$B(\mu',\mu) \in n \mathbb{Z}$ for all $\mu' \in \Lambda$} \}.
\]
The quotient $\Lambda/\Lambda^{(n)}$ is isomorphic to $(\mathbb{Z}/n \mathbb{Z})^r$. For every $\alpha^{\vee} \in \Phi^{\vee}$, let 
\begin{equation}
\label{E:number_na}
	n_{\alpha} = \dfrac{n}{\gcd(n, Q(\alpha^{\vee}) )} \, . 
\end{equation}
Our choice of $B$ is so that $Q(\mu) = 1$ if $\mu$ is a short simple root, and $Q(\mu) = 2$ if $\mu$ is a long simple root. It then follows that $n_{\alpha} = n$, as $n$ is odd.

We review the construction of unramified principal series of $\smash{ \widetilde{G} }$. Let $\smash{ \widetilde{B} = p^{-1} (B)}$ and $\smash{ \widetilde{T} = p^{-1} (T) }$. Let $H$ be the centralizer of $\smash{ \widetilde{T} \cap K }$ in $\smash{ \widetilde{T} }$. Let $\chi$ be a \textbf{genuine} character on $H$, i.e., $\chi(\zeta h) = \zeta \chi(h)$ for all $\zeta \in \mu_n$ and all $h \in H$ (this definition applies to characters on all subgroups of $\widetilde{G}$, not just to those on $H$). Suppose $\chi$ is \textbf{unramified}, i.e., $\chi$ is trivial on $\smash{ \widetilde{T} \cap K}$.

Induce $\chi$ from $H$ up to $\smash{\widetilde{T}}$, obtaining $i(\chi) \vcentcolon = \Ind_{H}^{\widetilde{T}} (\chi)$, a vector space with $\dim i(\chi) = |\widetilde{T}/H|$, the cardinality of $\widetilde{T}/H$. Inflate $i(\chi)$ from $\smash{ \widetilde{T}}$ to $\smash{ \widetilde{B}}$, then induce from $\smash{ \widetilde{B}}$ up to $\smash{\widetilde{G}}$, obtaining $I(\chi) \vcentcolon = \Ind_{\widetilde{B}}^{\widetilde{G}} (i(\chi))$. We call $I(\chi)$ the \textbf{unramified principal series representation} of $\widetilde{G}$ induced by $\chi$ (when $\chi$ is unramified, which we are assuming it is, we use the adjective unramified for $I(\chi)$ also). Then $I(\chi)$, which has a $\widetilde{G}$-action given by right translation, is the vector space consisting of all the locally constant functions $f \colon \widetilde{G} \to i(\chi)$ that satisfy 
\begin{equation}
\label{E:loc_const_fun}
	f(bg) = (\delta^{1/2} \chi) (b) f(g)
\end{equation}
for all $b \in \widetilde{B}$ and all $g \in \widetilde{G}$, where $\delta$ is the modular quasicharacter of $\widetilde{B}$. (In~\eqref{E:loc_const_fun}, we are considering $\chi$ as a character on $\widetilde{B}$.) The subset $I(\chi)^K$ of $I(\chi)$ consisting of all $K$-fixed elements in $I(\chi)$ is a vector space with $\dim I(\chi)^K = 1$; the elements of $I(\chi)^K$ are called \textbf{spherical} vectors. We choose a nonzero spherical vector $\phi_K^{\chi}$ in $I(\chi)^K$.

For every $\alpha \in \Phi^{+}$, let $U_{\alpha}$ be the one-parameter unipotent subgroup of $G$ that corresponds to the embedding $\iota_{\alpha} \colon \SL (2) \to G$. For every $w \in W$, define $\prescript{w}{}{\chi}$ by $\prescript{w}{}{\chi}(t) = \chi(w^{-1} t w)$ for all $t \in \widetilde{T}$. Then for every $w \in W$, let
\[
	U_w = \prod_{\substack{
	\alpha \in \Phi^{+} \\ 
	w(\alpha) \in \Phi^{-}
	}}
	U_{\alpha},
\]
a unipotent subgroup, and define $\mathcal{A}_{w} \colon I(\chi) \to I(\prescript{w}{}{\chi})$, the \textbf{unnormalized intertwining operator}, by
\[
	\mathcal{A}_{w} (f)(g) = \int_{U_w} 
	f(w^{-1} u g) \, du,
\]
assuming the integral is absolutely convergent. We may extend this definition by meromorphic continuation. By this we mean in the $\mathbf{z} \vcentcolon = (z_1,\ldots,z_r) \in \mathbb{C}^r$ that parametrizes $\chi$, $\mathbf{z}$ varies as $\chi$ varies, and so $\mathcal{A}_w (f)(g)$ actually depends on $\mathbf{z}$.

Let $w \in W$. Since $\dim I(\chi)^K = 1$ and $\mathcal{A}_w \phi_{K}^{\chi}$ is $K$-invariant,
\begin{equation}
\label{E:scalar_inter}
	\mathcal{A}_w \phi_{K}^{\chi} = c_w (\chi) \phi_{K}^{\prescript{w}{}{\chi}}
\end{equation}
for some nonzero element $c_w (\chi)$ in the fraction field of the coordinate ring of a certain algebraic variety. We do not need to say anything on this, but we do mention the following: if $s_{\alpha}$ is a simple reflection, and if $w \in W$ satisfies $\ell(s_{\alpha} w) = \ell(w) + 1$, where $\ell$ is the length function on the Weyl group $W$, then  
\[
	c_{s_{\alpha}} (\chi) = \dfrac{1 - q^{-1} \mathbf{z}^{n_{\alpha}\alpha^{\vee}}}{1 - \mathbf{z}^{n_{\alpha}\alpha^{\vee}}} \, ,
	\qquad
	c_{s_{\alpha} w} (\chi) = c_{s_{\alpha}} (\prescript{w}{}{\chi}) c_{w} (\chi).
\]
Then the \textbf{normalized intertwining operator} $\overline{\mathcal{A}}_w$ is defined to be
\[
	\overline{\mathcal{A}}_w = 
	(c_{w} (\chi))^{-1} \mathcal{A}_w.
\]
It follows from~\eqref{E:scalar_inter} that $\overline{\mathcal{A}}_w \phi_{K}^{\chi} = \phi_{K}^{\prescript{w}{}{\chi}}$.

An advantage of working with $\overline{\mathcal{A}}_{w}$ rather than with $\mathcal{A}_w$ is that for all $w$,~$w' \in W$, we have $\overline{\mathcal{A}}_{ww'} = \overline{\mathcal{A}}_w \overline{\mathcal{A}}_{w'}$, while we must have $\ell(ww') = \ell(w) + \ell(w')$ in order for $\mathcal{A}_{ww'} = \mathcal{A}_{w} \mathcal{A}_{w'}$. We will therefore work with the $\overline{\mathcal{A}}_w$. In addition, it will suffice to work with $\overline{\mathcal{A}}_{s_{\alpha}}$ on simple reflections $s_{\alpha}$.

Let $(\pi, V)$ be a representation of $\widetilde{G}$. Let $\psi$ be an unramified character on $U$. A \textbf{Whittaker functional} on $(\pi,V)$ is a linear functional $W$ on $V$ satisfying
\[
	W(\pi(u)v) = \psi(u) W(v)
\]
for all $u \in U$ and all $v \in V$. In particular, let us take $(\pi,V)$ to be $I(\chi)$, and let $\mathcal{W}^{\chi} \colon I(\chi) \to i(\chi)$ be the linear functional on $I(\chi)$ defined by
\begin{equation}
\label{E:mayday}
	\mathcal{W}^{\chi} (f) = \int_{U^{-}} f(u w_0) \psi(u) \, du
\end{equation}
for all $f \in I(\chi)$, where $U^{-}$ is the opposite group to the unipotent radical of $B$.

Denote by $i(\chi)^{\ast}$ the dual space of $i(\chi)$ and by $\mathcal{S}$ the vector space of all Whittaker functionals $I(\chi) \to \mathbb{C}$. The following is essentially Theorem 6.2 in~\cite{McNamara3}:

\begin{theorem}
\label{thm:McNamara_bij}
There exists an isomorphism $i(\chi)^{\ast} \to \mathcal{S}$, with the isomorphism given by $\mathcal{L} \mapsto W_{\mathcal{L}}$, where $W_{\mathcal{L}} (f) = \mathcal{L}(\mathcal{W}^{\chi} (f))$ for all $f \in I(\chi)$.
\end{theorem}

\noindent It follows that $\dim \mathcal{S} = |\widetilde{T}/H |$. Moreover, it follows from results in~\cite{McNamara2} that $\widetilde{T}/H \cong \Lambda/\Lambda^{(n)}$.

Let $\{ W_{b}^{\chi}\}_{b}^{\phantom{\chi}}$ be a basis for $\mathcal{S}$, where the $W_{b}^{\chi}$ are indexed by a complete set of coset representatives $b$ for $\widetilde{T}/H$. Let $\{ \mathcal{L}_b \}_{b}$ be the corresponding basis for $i(\chi)^{\ast}$ under the bijection in Theorem~\ref{thm:McNamara_bij}.

Let $w \in W$, and let $a$ be a coset representative for $\widetilde{T}/H$. Consider now the Whittaker functional $\mathcal{W}^{\prescript{w}{}{\chi}}$ on $I(\prescript{w}{}{\chi})$, where $\mathcal{W}^{\chi}$ is defined in~\eqref{E:mayday}. The composition $\mathcal{W}^{\prescript{w}{}{\chi}} \circ \overline{\mathcal{A}}_{w} \colon I(\chi) \to i(\prescript{w}{}{\chi})$ is a Whittaker functional, which becomes an $i(\chi)$-valued Whittaker functional after composing with the isomorphism $i(\prescript{w}{}{\chi}) \to i(\chi)$. Composing further with the linear functional $\mathcal{L}_{a} \in i(\chi)^{\ast}$ yields a $\mathbb{C}$-valued Whittaker functional that we denote by $W_{a}^{\prescript{w}{}{\chi}} \circ \overline{\mathcal{A}}_{w}$. Since $W_{a}^{\prescript{w}{}{\chi}} \circ \overline{\mathcal{A}}_{w} \in \mathcal{S}$, we can write it in terms of the basis vectors $W_{b}^{\chi}$:
\begin{equation}
\label{E:Whittaker_expansion}
	W_{a}^{\prescript{w}{}{\chi}} \circ \overline{\mathcal{A}}_{w} = \sum_{b \in \widetilde{T}/H} \tau_{a,b}^{(w)} (\mathbf{z}) W_{b}^{\chi}
\end{equation}
for some rational functions $\tau_{a,b}^{(w)}(\mathbf{z})$. It suffices to know all of the $\tau_{a,b}^{(w)}(\mathbf{z})$ for simple reflections. The following is essentially Lemma I.3.3 in~\cite{KP} or Theorem 13.1 in~\cite{McNamara3}:

\begin{theorem}
\label{T:Kazhdan--Patterson}
Let $s_{\alpha}$ be a simple reflection. Let $a$ and $b$ be coset representatives for $\widetilde{T}/H$, with $a = \varpi^{\nu}$ and $b = \varpi^{\mu}$ for some $\nu$,~$\mu \in \Lambda$. Write $\tau_{\nu, \mu}$ for the structure constant $\smash{ \tau_{a, b}^{(s_{\alpha})} (\mathbf{z})}$ in~\eqref{E:Whittaker_expansion}. Then $\tau_{\nu, \mu}^{\vphantom{1}} = \tau_{\nu, \mu}^1 + \tau_{\nu, \mu}^2$, where $\tau^1$ vanishes unless $\nu \sim \mu \mod \Lambda^{(n)}$, and $\tau^2$ vanishes unless $\nu \sim s_{\alpha}(\mu) + \alpha^{\vee} \mod \Lambda^{(n)}$. Moreover, if we set $C$ and $D$ to be
\[
	C = \biggl(n_{\alpha} \biggl\lceil \dfrac{B(\alpha^{\vee},\mu)}{n_{\alpha} Q(\alpha^{\vee})} \biggr\rceil - \dfrac{B(\alpha^{\vee},\mu)}{Q(\alpha^{\vee})} \biggr) \alpha^{\vee} \, ,
	\qquad
	D = g(B(\alpha^{\vee},\mu) - Q(\alpha^{\vee})),
\]
where $n_{\alpha}$ is given as in~\eqref{E:number_na} and $\lceil \, \cdot \, \rceil$ is the ceiling function, then
\[
	\tau_{\mu,\mu}^{1} = (1-q^{-1})\dfrac{\mathbf{z}^C}{1 - q^{-1} \mathbf{z}^{n_{\alpha} \alpha^{\vee}}} \, , \qquad
	\tau_{s_{\alpha}(\mu) + \alpha, \mu}^{2} = q^{-1} D \mathbf{z}^{- \alpha^{\vee}}\dfrac{1 - \mathbf{z}^{n_{\alpha} \alpha^{\vee}}}{1 - q^{-1} \mathbf{z}^{n_{\alpha}\alpha^{\vee}}} \, .
\]
\end{theorem}

In the next section, we will prove that the relationship between partition functions and $R$-vertices for our symplectic-ice model is the same as the relationship between Whittaker functionals and intertwining operators.

A \textbf{Whittaker function} is a nonzero function $W_{\chi} \colon \widetilde{G} \to \mathbb{C}$ that satisfies 
\[
	W_{\chi} (\zeta u g k) = \zeta \psi(u) W_{\chi}(g)
\]
for all $\zeta \in \mu_n$, $u \in U$, $g \in \widetilde{G}$, and $k \in K$. We explicitly mention a Whittaker function. Let $a$ be a coset representative for $\widetilde{T}/H$. Composing the function $\widetilde{G} \to i(\chi)$, defined by $g  \mapsto \mathcal{W}^{\chi} (\pi(g) \phi_K^{\chi})$, with the linear functional $\mathcal{L}_{a} \in i(\chi)^{\ast}$ yields a Whittaker function $\smash{ W_{a}^{\circ} \colon \widetilde{G} \to \mathbb{C} }$, called the \textbf{spherical Whittaker function}. 

\section{Connections to Intertwining Operators and Whittaker Functions}
\label{sec:Connections_intertwining_Whittaker}

In this section, we consider the case with $G = \SO (2r+1)$. We continue to denote by $Z_{\lambda} (\mathbf{z}; \mathbf{c})$ the partition function of our model, having top boundary conditions determined by $\lambda$, such that the leftmost edges in rows $\overline{1}$,~\ldots,~$\overline{r}$ for each admissible state have charges congruent modulo $n$ to the integers $c_1$,~\ldots,~$c_r$, respectively, where $\mathbf{c} = (c_1,\ldots,c_r)$, and where (as usual) along the left boundary all spins are $+$ and all charges in rows $1$,~\ldots,~$r$ are $0$. Here $c_i \in [0, n-1]$ for every $i$. We continue using the notation from Section~\ref{sec:Algebraic_Preliminaries}.

\begin{conjecture}
\label{C:conj_Z_spherical}
Let $\mathbf{c} = (c_1,\ldots,c_r) \in [0,n-1]^{r}$. Let $a$ be a coset representative for $\widetilde{T}/H$ with $a = \varpi^{\nu}$ for some $\nu \in \Lambda$ satisfying $\nu - \rho = c_1 e_1 + \cdots + c_r e_r$. Then $Z_{\lambda} (\mathbf{z}; \mathbf{c})$ is a spherical Whittaker function, i.e., $Z_{\lambda} (\mathbf{z}; \mathbf{c}) = \mathbf{z}^{\varepsilon(a,\mathbf{c})} W_{a}^{\circ} ( \pi(\varpi^{\lambda}) \phi_K^{\chi})$, where $\varepsilon(a,\mathbf{c}) \in \mathbb{Z}^r$ depends on both $a$ and $\mathbf{c}$.
\end{conjecture}

\begin{remark}\hfill
\begin{enumerate}
\item Friedberg and Zhang~\cite{FZ} discuss some conjectural connections between the $p$-parts of a metaplectic Eisenstein series and the local Whittaker function. Conjecture~\ref{C:conj_Z_spherical} would then follow from that result and Theorem~\ref{T:sym_pattern_bij}. This is a topic for future exploration.
\item Assume Conjecture~\ref{C:conj_Z_spherical}, and suppose $n=1$. We obtain a new proof of the Casselman--Shalika formula by combining the conjecture above with either the combinatorial result of Hamel and King~\cite{HK1} or the statistical-mechanical result of Ivanov~\cite{Ivanov}.
\end{enumerate}
\end{remark}

We will not prove Conjecture~\ref{C:conj_Z_spherical}, but we will prove that the Boltzmann weights of the $\Gamma \Gamma$ braided ice match the structure constants $\tau^1$ and $\tau^2$ in~\eqref{E:Whittaker_expansion}. We begin by modifying the Boltzmann weights in Tables~\ref{tab:Delta_Gamma} and~\ref{tab:braid_weights}, obtaining the weights in Tables~\ref{tab:Revised_Delta_Gamma} and~\ref{tab:modified_braid_weights_GG}. Refer to the configurations in Table~\ref{tab:modified_braid_weights_GG} as $\smash{ \widehat{R} }$-vertices. These modified weights are obtained by following the change-of-basis procedure outlined in~\cite{BBB}.

Recall from Theorem~\ref{P:n_admissible} that it suffices to set the decoration $a$ in any of the six admissible configurations equal to $0$ if the spin is $+$ (resp., $-$) for $\Delta$ ice (resp., $\Gamma$ ice). Define the functions $f_{\Delta}$ and $f_{\Gamma}$ as follows: for every decorated spin $\alpha$, if $\alpha$ has ice type $\Delta$, then
\[
	f_{\Delta}(\alpha, z) = 
	\begin{cases}
	z^{a}	&\text{if $\alpha = -a$ and $a \in [1, n-1]$,}\\
	1		&\text{if $\alpha = +0$,}
	\end{cases}
\]
while if $\alpha$ has ice type $\Gamma$, then
\[
	f_{\Gamma}(\alpha, z) = 
	\begin{cases}
	z^{a}	&\text{if $\alpha = +a$ and $a \in [1, n-1]$,}\\
	1		&\text{if $\alpha = -0$.}
	\end{cases}
\]
For all $X$,~$Y \in \{ \Delta, \Gamma \}$, the Boltzmann weights
\[
	\wt_{X} 
	\biggl(\!
	\mathord{%
	\begin{tikzpicture}[scale=0.5, every node/.style={scale=1.0}, baseline=-3.5]
	\draw[semithick] (1, 0.60) -- (1,-0.60);
	\draw[fill=white, white] (1,0) circle (4.5pt);	
	\draw[semithick] (0.40,0) -- (1.6,0);
	\draw[fill=black] (1,0) circle (2.25pt);	
	\begin{scope}
	\tikzstyle{spin}=[draw, circle] 
	\path[black, very thin] 
		(1,0.6) 
		node[inner sep=-0.1pt, fill=white, spin] 
		{\tiny$\phantom{\boldsymbol{+}}$}
		(1.6,0) 
		node[inner sep=-0.1pt, fill=white, spin] 
		{\tiny$\phantom{\boldsymbol{+}}$}
		(1.6,0) node {\scriptsize$\beta$}
		(1,-0.6) 
		node[inner sep=-0.1pt, fill=white, spin] 
		{\tiny$\phantom{\boldsymbol{+}}$}
		(0.4,0) 
		node[inner sep=-0.1pt, fill=white, spin] 
		{\tiny$\phantom{\boldsymbol{+}}$}
		(0.4,0) node {\scriptsize$\alpha$};
	\end{scope}
	\end{tikzpicture}
	}
	\!\biggr),
	\qquad
	\wt_{X Y} 
	\biggl(
	\mathord{%
	\begin{tikzpicture}[scale=0.5, every node/.style={scale=1.0}, baseline=-3.5]
	\begin{scope}[rotate around={45:(1,0)}]
	\draw[semithick] (1, 0.60) -- (1,-0.60);
	\draw[fill=white, white] (1,0) circle (4.5pt);	
	\draw[semithick] (0.40,0) -- (1.6,0);
	\draw[fill=black] (1,0) circle (2.25pt);	
	\begin{scope}
	\tikzstyle{spin}=[draw, circle] 
	\path[black, very thin] 
		(1,0.6) 
		node[inner sep=-0.1pt, fill=white, spin] 
		{\tiny$\phantom{\boldsymbol{+}}$}
		(1,0.6) node {\scriptsize$\beta$}
		(1.6,0) 
		node[inner sep=-0.1pt, fill=white, spin] 
		{\tiny$\phantom{\boldsymbol{+}}$}
		(1.6,0) node {\scriptsize$\gamma$}
		(1,-0.6) 
		node[inner sep=-0.1pt, fill=white, spin] 
		{\tiny$\phantom{\boldsymbol{+}}$}
		(1,-0.6) node {\scriptsize$\delta$}
		(0.4,0) 
		node[inner sep=-0.1pt, fill=white, spin] 
		{\tiny$\phantom{\boldsymbol{+}}$}
		(0.4,0) node {\scriptsize$\alpha$};
	\end{scope}
	\end{scope}
	\end{tikzpicture}
	}
	\biggr),
\]
taken from Tables~\ref{tab:Delta_Gamma} and~\ref{tab:braid_weights}, are multiplied respectively by
\[
	\dfrac{f_{X} (\alpha, z_i)}{f_{X} (\beta, z_i)} \, , \qquad
	\dfrac{f_{X} (\alpha, z_1) f_{Y} (\beta, z_2)}{f_{X} (\gamma, z_1) f_{Y} (\delta, z_2)} \, .
\] 
In addition, we divide each $\Gamma$ weight in Table~\ref{tab:Delta_Gamma} by $z_{i}$. We divide each $\Delta \Gamma$ weight in Table~\ref{tab:braid_weights} by $z_{2}^n - v^n z_{1}^n$, and we divide each $\Gamma \Delta$, $\Delta \Delta$, and $\Gamma \Gamma$ weight by $z_{1}^n - v z_{2}^n$. The reason for these divisions will become clear in a moment. In Table~\ref{tab:modified_braid_weights_GG}, we list only the modified $\Gamma \Gamma$ weights.

\begin{table}[b]
\caption{Modified Boltzmann weights of $\Delta$ ice and $\Gamma$ ice.}
\centering
\noindent\makebox[\textwidth][c]{
\footnotesize
\begin{tabularx}{0.90\textwidth}{c*{6}{Y}}
\toprule
	{}
	& \multicolumn{6}{l}{\textbf{Boltzmann}} 
	\\
	{}
	& \multicolumn{6}{l}{\textbf{Weights}} 
	\\
\end{tabularx}}
\label{tab:Revised_Delta_Gamma}
\end{table}

\begin{table}[htb]
\begin{threeparttable}
\caption{Modified Boltzmann weights of $\Gamma \Gamma$ ice. Here $a \neq b$ and $\mathbf{z}^{-\alpha} \vcentcolon = z_1/z_2$.}
\centering
\footnotesize
\begin{tabularx}{1.0\textwidth}{l*{4}{Y}}
\toprule
	{}
	& \multicolumn{4}{l}{\textbf{Boltzmann}} 
	\\
	{}
	& \multicolumn{4}{l}{\textbf{Weights}} 
	\\
\end{tabularx}
\label{tab:modified_braid_weights_GG}
\end{threeparttable}
\end{table}

\begin{theorem}
The YBE is satisfied with the Boltzmann weights given in Tables~\ref{tab:Revised_Delta_Gamma} and~\ref{tab:modified_braid_weights_GG}.
\end{theorem}
\begin{proof}
The original Boltzmann weights satisfy the YBE. Since our modifications apply to all these weights, the YBE is not affected.
\end{proof}

Brubaker, Bump, Chinta, Friedberg, and Gunnells~\cite{BBCFG} showed that for $G = \GL (r,F)$, Boltzmann weights exist for a generalization of the six-vertex model (i.e., type A metaplectic ice) for which the partition functions are values of spherical Whittaker functions on $\smash{\widetilde{G}}$. It was conjectured that properties of these Whittaker functions arose from a YBE for the model, though no YBE was found.

Later on, Brubaker, Buciumas, and Bump~\cite{BBB} proved this conjecture by finding Boltzmann weights for metaplectic ice that allowed for a YBE. This demonstrated a new connection between quantum groups and spherical Whittaker functions on $\smash{\widetilde{G}}$. (When the degree $n$ of the cover of $G$ is $1$, the metaplectic ice matches the model used by Brubaker--Bump--Friedberg~\cite{BBF2}.) In addition, they proved that the $R$-matrix, a matrix that encodes the solutions to the corresponding YBE, is a Drinfeld twist of the $R$-matrix for the quantum affine Lie superalgebra $\smash{ \widehat{\gl}(1 | n) }$, and that the scattering matrix of the intertwining operator corresponding to a simple reflection on the finite-dimensional vector space of Whittaker functionals for $\smash{\widetilde{G}}$ is the $R$-matrix of quantum affine $\gl (n)$, modified by Drinfeld twisting. (This scattering matrix was originally computed by Kazhdan--Patterson~\cite{KP}.)

It was shown in~\cite{BBB} that some of the modified weights in Table~\ref{tab:modified_braid_weights_GG} are related to the identity of Whittaker functionals given in~\eqref{E:Whittaker_expansion}. We will partially prove a similar result for symplectic ice. From now on, $G$ stands for $\SO(2r+1, F)$.

\begin{theorem}[Brubaker--Buciumas--Bump,~\cite{BBB}, Proposition 8]
\label{P:taus_R-vertices}
Let $\nu \in X_{\ast}(T) \cong \mathbb{C}[\Lambda]$, where $\nu - \rho = c_1 e_1 + \cdots + c_r e_r$ for some integers $c_i \in [0,n-1]$. Let $i \in \{ 1, 2, \ldots,r-1\}$, and set $j = i+1$. Let $s_i \vcentcolon = s_{\alpha}$ be the simple reflection $i \leftrightarrow j$. Let $a$,~$b \in \mathbb{Z}$, and suppose $a \equiv c_i \pod{n}$ and $b \equiv c_{j} \pod{n}$. Write $\tau_{\nu, \mu}$ for the structure constant $\tau_{a, b}^{(s_i)} (\mathbf{z})$, as in Theorem~\ref{T:Kazhdan--Patterson}. Let $e \equiv c_i - c_j \pod{n}$ with $e \in [0, n-1]$. Let $\wt_{\Gamma \Gamma}$ be the $\Gamma \Gamma$ Boltzmann weights for the $\smash{ \widehat{R} }$-vertices in Table~\ref{tab:modified_braid_weights_GG}, with $v = q^{-1}$. If $a \not \equiv b \pod{n}$, then
\[
	\tau_{\nu, \nu}^{1} = 
	\wt_{\Gamma \Gamma}
	\biggl(\!\!
	\mathord{%

	}
	\!\!\biggr).
\]
\end{theorem}
\begin{proof}
Suppose $a \not \equiv b \pod{n}$. Then
\[
	\tau_{\nu,\nu}^1 
	= (1- q^{-1}) \dfrac{\mathbf{z}^{-n \lceil B(\alpha,\nu)/n \rceil \alpha}}{1 - q^{-1} \mathbf{z}^{-n\alpha}}
	= \dfrac{1-v}{1-v \mathbf{z}^{-n\alpha}} 
	\begin{cases}
	\mathbf{z}^{-n\alpha} &\text{if $c_i > c_j$}\\
	1 &\text{if $c_i < c_j$,}
	\end{cases} 
\]
which equals the modified $\Gamma \Gamma$ Boltzmann weight of the first $\smash{ \widehat{R} }$-vertex mentioned in the theorem. Similarly,
\begin{align*}
	\tau_{s_i (\nu) + \alpha, \mu}^{2} 
	= g(\langle \alpha,\nu-\rho \rangle - 1) \dfrac{1-\mathbf{z}^{-n\alpha}}{1-q^{-1}\mathbf{z}^{-n\alpha}} 
	= g(c_i-c_j) \dfrac{1 - \mathbf{z}^{-n\alpha}}{1 - v \mathbf{z}^{-n\alpha}} \, ,
\end{align*}
which equals the modified $\Gamma \Gamma$ Boltzmann weight of the second $\smash{ \widehat{R} }$-vertex mentioned in the theorem.

Now suppose $a \equiv b \pod{n}$. Then the assertion follows by setting $c_i = c_j$ in the Boltzmann weights of the two $\widehat{R}$-vertices above, where we take the first case for the second $\widehat{R}$-vertex above:
\begin{align*}
	\tau_{\nu,\nu}^{1} + \tau_{s_i (\nu) + \alpha, \mu}^{2} 
	&= \dfrac{1-v}{1-v \mathbf{z}^{-n\alpha}} \left.
	\begin{cases}
	\mathbf{z}^{-n\alpha} &\text{if $c_i > c_j$}\\
	1 &\text{if $c_i < c_j$}
	\end{cases} \right\} + g(c_i-c_j) \dfrac{1 - \mathbf{z}^{-n\alpha}}{1 - v \mathbf{z}^{-n\alpha}}
	\\[1ex]
	&= \dfrac{(1-v)\mathbf{z}^{-n\alpha}}{1-v\mathbf{z}^{-n\alpha}} + \dfrac{-v(1 - \mathbf{z}^{-n\alpha})}{1-v\mathbf{z}^{-n\alpha}} \\[1ex]
	&= \dfrac{\mathbf{z}^{-n\alpha} - v}{1 - v \mathbf{z}^{-n\alpha}} \, ,
\end{align*}
which equals the modified $\Gamma \Gamma$ Boltzmann weight of the third $\widehat{R}$-vertex mentioned in the theorem.
\end{proof}

\begin{conjecture}
\label{C:modified_functional1}
Continue using the notation from Theorems~\ref{T:Kazhdan--Patterson} and~\ref{P:taus_R-vertices}. Write $w$ for the simple reflection attached to the simple short root for which $w$ interchanges $i$ and $j \vcentcolon = i+1$. Consider symplectic ice with rows $\overline{1}$,~$\overline{2}$,~\ldots,~$\overline{r}$ having fixed decorated spins of\/ $+c_1$,~$+c_2$,~\ldots,~$+c_r$, respectively, as described at the beginning of Section~\ref{sec:Function_Eqns_Partition}. Assume Conjecture~\ref{C:conj_Z_spherical} is true. It follows from Theorems~\ref{T:transposition} and~\ref{P:taus_R-vertices} that if $c_i \neq c_{j}$, then the identity
\[
	W_{a}^{\prescript{w}{}{\chi}} \circ 
	\overline{\mathcal{A}}_{w} 
	(\pi(\varpi^{\lambda}) \phi_{K}^{\chi}) 
	= 
	\tau_{\nu, \nu}^1 W_{a}^{\chi}
	(\pi(\varpi^{\lambda}) \phi_{K}^{\chi}) 
	+ 
	\tau_{w \cdot \nu, \nu}^2 W_{w \cdot a}^{\chi}
	(\pi(\varpi^{\lambda}) \phi_{K}^{\chi}) 
\]
of spherical Whittaker functions is equivalent to the identity
\begin{equation}
\label{E:identity1_modified}
	Z_{\lambda}(w (\mathbf{z}); \mathbf{c}) 
	=
	\wt_{\Gamma \Gamma} 
	\biggl(\!\!
	\mathord{%

	}
	\!\!\biggr)
	\, Z_{\lambda} (\mathbf{z}; w (\mathbf{c}))
\end{equation}
of partition functions. If $c_i = c_j$, then~\eqref{E:identity1_modified} is rewritten by adding the Boltzmann weights of the two $\widehat{R}$-vertices in~\eqref{E:identity1_modified} as was done in the proof of Theorem~\ref{P:taus_R-vertices}. (Here $w (\mathbf{z})$ means that $z_i$ and $z_j$ are interchanged.)
\end{conjecture}

The reason for the divisions in the modified weights mentioned at the beginning of the section was so that both the caduceus constant given in~\eqref{E:cad_result} and the fish constant given in Lemma~\ref{thm:Fish2} reduce to $1$ when using these weights.

\begin{conjecture}
\label{C:modified_functional2}
Continue using the notation from Theorems~\ref{T:Kazhdan--Patterson} and \ref{P:taus_R-vertices}, Conjecture~\ref{C:modified_functional1}, and Section~\ref{sec:Function_Eqns_Partition}. Write $w$ for the simple reflection attached to the simple long root for which $w$ interchanges $r$ and $-r$. Assume Conjecture~\ref{C:conj_Z_spherical} is true. It follows from Theorems~\ref{thm:inverse} and~\ref{P:taus_R-vertices} that the identity 
\[
	Z_{\lambda}(w(\mathbf{z}); \mathbf{c}) 
	=
	\overline{\wt}_{\Gamma \Gamma} 
	\biggl(\!\!
	\mathord{%
	\begin{tikzpicture}[scale=0.5, every node/.style={scale=1.0}, baseline=-3.5]
	\begin{scope}[rotate around={45:(1,0)}]
	\draw[semithick, black] (1, 0.60) -- (1,-0.60);
	\draw[fill=white, white] (1,0) circle (4.5pt);	
	\draw[semithick, black] (0.40,0) -- (1.6,0);
	\draw[fill=black, black] (1,0) circle (2.25pt);	
	\begin{scope}
	\tikzstyle{spin}=[draw, circle] 
	\path[black, very thin] 
		(1,0.6) 
		node[inner sep=-0.5pt, fill=white, spin] 
		{\tiny$\boldsymbol{+}$}
		node[left] {\scriptsize$a \;$}
		(1.6,0) 
		node[inner sep=-0.5pt, fill=white, spin] 
		{\tiny$\boldsymbol{+}$}
		node[right] {\scriptsize$\; a$}
		(1,-0.6) 
		node[inner sep=-0.5pt, fill=white, spin] 
		{\tiny$\boldsymbol{+}$}
		node[right] {\scriptsize$\; b$}
		(0.4,0) 
		node[inner sep=-0.5pt, fill=white, spin] 
		{\tiny$\boldsymbol{+}$}
		node[left] {\scriptsize$b \;$};
	\end{scope}
	\end{scope}
	\end{tikzpicture}
	}
	\!\!\biggr)
	\, Z_{\lambda} (\mathbf{z}; \mathbf{c})	
	+
	\overline{\wt}_{\Gamma \Gamma} 
	\biggl(\!\!
	\mathord{%
	\begin{tikzpicture}[scale=0.5, every node/.style={scale=1.0}, baseline=-3.5]
	\begin{scope}[rotate around={45:(1,0)}]
	\draw[semithick, black] (1, 0.60) -- (1,-0.60);
	\draw[fill=white, white] (1,0) circle (4.5pt);	
	\draw[semithick, black] (0.40,0) -- (1.6,0);
	\draw[fill=black, black] (1,0) circle (2.25pt);	
	\begin{scope}
	\tikzstyle{spin}=[draw, circle] 
	\path[black, very thin] 
		(1,0.6) 
		node[inner sep=-0.5pt, fill=white, spin] 
		{\tiny$\boldsymbol{+}$}
		node[left] {\scriptsize$a \;$}
		(1.6,0) 
		node[inner sep=-0.5pt, fill=white, spin] 
		{\tiny$\boldsymbol{+}$}
		node[right] {\scriptsize$\; b$}
		(1,-0.6) 
		node[inner sep=-0.5pt, fill=white, spin] 
		{\tiny$\boldsymbol{+}$}
		node[right] {\scriptsize$\; a$}
		(0.4,0) 
		node[inner sep=-0.5pt, fill=white, spin] 
		{\tiny$\boldsymbol{+}$}
		node[left] {\scriptsize$b \;$};
	\end{scope}
	\end{scope}
	\end{tikzpicture}
	}
	\!\!\biggr)_2
	\, Z_{\lambda} (\mathbf{z}; c_1, \ldots, c_{r-1}, c_r - 2e)
\]
of partition functions is equivalent to an identity of spherical Whittaker functions, where $a$ and $b$ satisfy $a-b \equiv c_r - N \pod{n}$. (Here $w(\mathbf{z})$ and $\overline{\wt}$ mean that $z_r$ and $z_{r}^{-1}$ are interchanged, and the subscript $2$ in $\overline{\wt}( \,\cdot\, )_2$ means that each occurrence of $g(a)$ is replaced by $g(2a)$.)
\end{conjecture}

A proof of Conjecture~\ref{C:conj_Z_spherical} would therefore establish Conjectures~\ref{C:modified_functional1} and~\ref{C:modified_functional2}, proving that the relationship between partition functions and $R$-vertices matches the relationship between spherical Whittaker functions and the structure constants $\tau^1$ and $\tau^2$. It should be mentioned, however, that this matching is not enough to show that the partition function $Z_{\lambda}$ and the Whittaker function are the same, since the functional equations do not uniquely determine the involved functions. This remains a topic for future exploration. 

\section{Further Questions}
\label{sec:further_questions}

It would be interesting to give a proof of Conjecture~\ref{C:conj_Z_spherical} from Section~\ref{sec:Connections_intertwining_Whittaker}, thereby showing that the partition functions of symplectic ice satisfy the same identities under our solution to the Yang--Baxter equation as the metaplectic Whittaker function under intertwining operators on unramified principal series.

An interesting future project would be to explore the case when $n$ is even, since the techniques that have been used in this paper have relied on $n$ being odd.

In Section~\ref{sec:Connections_Metaplectic}, we related certain admissible states of our symplectic-ice model to a multiple Dirichlet series. Some possible future work could perhaps relate all the admissible states---and therefore the partition function $Z$---to the series.

Concerning the appearance of Whittaker functions in Section~\ref{sec:Connections_intertwining_Whittaker}, this is not the first time that such functions have been connected to statistical mechanics. Connections between statistical mechanics and archimedean Whittaker functions date back to work of Kazhdan and Kostant, who recognized that quantum Toda Hamiltonians, when restricted to the space of Whittaker functions, agreed with the differential operators in the center of the universal enveloping algebra. This inspired Kostant's later proof of the total integrability of the Toda lattice.

It would be interesting to explore why statistical-mechanical models have been appropriate analogues of the Toda lattice; the models have a discrete nature, whereas the Toda lattice is continuous.

Here is another topic to explore. Consider metaplectic Whittaker functions on covering groups where the underlying field is $\mathbb{C}$. Each central extension is trivial, since everything splits. But this is different for the case when the underlying field is $\mathbb{R}$. (See~\cite{Kostant}.) It would be interesting to study the metaplectic archimedean case for the Toda lattice, i.e., a metaplectic Whittaker function for a double cover of real groups.

There is a deeper question on connections to geometry. McNamara~\cite{McNamara} showed that for $\GL(r)$, if one breaks up the unipotent radical in $\GL(r)$ into geometrically defined pieces called Mirkovi\'{c}--Vilonen cycles, then the contribution from
each piece matches the Boltzmann weight of a state of square ice and therefore is a summand in Tokuyama's generating function. Does such a connection exist for other Cartan types? We can ask a more elementary question about whether generating function identities like Tokuyama's exist for characters of other types. Only one for type C is known.

%
\appendix
\section{Proofs from \S \ref{sec:Function_Eqns_Partition}}
\label{app:appendixCF}

\begin{proof}[Proof of Lemma~\ref{thm:caduceus} (caduceus relation)]
There are four choices for $(\varepsilon_1, \varepsilon_2, \varepsilon_3, \varepsilon_4)$. If $(\varepsilon_1, \varepsilon_2, \varepsilon_3, \varepsilon_4) = (+,-,-,+)$, then $I_5$ has four admissible states, and $Z(I_5) / \wt(I_6)$ is equal to
\[
	(z_{j}^{-n} - v_{\vphantom{i}}^n z_{i}^n)
	(z_{i}^{-n} - v_{\vphantom{j}} z_{j}^n) 
	(z_{i}^n - v_{\vphantom{j}} z_{j}^n)
	(z_{i}^{-n} - v_{\vphantom{j}} z_{j}^{-n}),
\]
which is~\eqref{E:cad_result}. If $(\varepsilon_1, \varepsilon_2, \varepsilon_3, \varepsilon_4) = (-, +, +, -)$, then $I_5$ has four admissible states, in which case $Z(I_5) / \wt(I_6)$ is equal to~\eqref{E:cad_result} as well.

Suppose $(\varepsilon_1, \varepsilon_2, \varepsilon_3, \varepsilon_4) = (+,-,+,-)$; we will analyze this case in more detail. The admissible states of $I_5$ are shown in Figure~\ref{fig:cad3}, where $\alpha$ and $\beta$ range over $\{ 0,1,.\ldots,n-1\}$ and $\alpha + \beta \equiv 1 \pod{n}$.

\begin{figure}[htb]
\centering
\makebox[\textwidth][c]{%
\centering
\begin{minipage}[c]{0.87\textwidth}
\hspace{0.85in}
%
\end{minipage}
}
\caption{}
\label{fig:cad3}
\end{figure}

\noindent If $\alpha = n-1$, then $\alpha+1$ stands for $0$, and if $\beta=0$, then $\beta-1$ stands for $n-1$. Number the five configurations in Figure~\ref{fig:cad3} from $1$ to $5$, starting with the top row and going left to right in each row. The Boltzmann weights of states $1$, $2$, and $5$ are easily determined. Each choice of $(\alpha,\beta)$ yields an admissible state of each of configurations $3$ and $4$. The sum of the Boltzmann weights of all admissible states for configuration $3$ is equal to
\[
	(1-v)^3 (1 - v^n) z_{i}^{-1} z_{j}^{-1},
\]
while the sum of the Boltzmann weights of all admissible states for configuration $4$ is equal to
\[
	(1-v) 
	(
	v^{n+1}(1-v) z_{i}^{n} z_{j}^{n} + v - v^n
	) 
	(z_{i}^{n} - z_{j}^{n})
	(z_{i}^{-n} - z_{j}^{-n})
	z_{i}^{-1} z_{j}^{-1} .
\]
It is straightforward to verify that $Z(I_5) / \wt(I_6)$ is equal to~\eqref{E:cad_result}, where $Z(I_5)$ is the sum of the Boltzmann weights of all admissible states for the five configurations.

The case when $(\varepsilon_1, \varepsilon_2, \varepsilon_3, \varepsilon_4) = (-, +, -, +)$ is much more tedious and is therefore left to the reader. Again, $Z(I_5) / \wt(I_6)$ is equal to~\eqref{E:cad_result}.
\end{proof}

\begin{proof}[Proof of Lemma~\ref{thm:Fish1} (fish relation, type $\varDelta \varGamma$)]
There are two choices for the pair $(\varepsilon_1, \varepsilon_2)$ of decorated spins. If $(\varepsilon_1, \varepsilon_2) = (+, -)$, then the admissible states of $I_5$ are shown in Figure~\ref{fig:Fish1_first},

\begin{figure}[htb]
\makebox[\textwidth][c]{%
%
}
\caption{}
\label{fig:Fish1_first}
\end{figure}

\noindent where $a + b \equiv 1 \pod{n}$. There are three subcases to consider. If $a = 0$, then only the first and third configurations are considered, in which case $Z (I_5) / \wt(I_6)$ is equal to $z_{r}^{-n} - v_{\vphantom{r}}^{n} z_{r}^{n}$. If $a \neq 0$ and the leftmost charges are $(a+1,a)$, then only the second configuration is considered, in which case $Z (I_5) / \wt(I_6)$ is equal to $z_{r}^{-n} - v_{\vphantom{r}}^{n} z_{r}^{n}$. If $a \neq 0$ and the leftmost charges are $(b,a)$, then only the third and fourth configurations are considered, in which case $Z(I_5) / \wt(I_6)$ is equal to $0$. But this third subcase can be excluded from consideration if we do not allow a flipped $\Gamma \Delta$-bend of the form $\smash{( \dspin{b}{+}, \dspin{a}{-} )}$ with $a+b \equiv 1 \pod{n}$ and $a \not \equiv 0$.

If $(\varepsilon_1, \varepsilon_2) = (-, +)$, then the admissible states of $I_5$ are shown in Figure~\ref{fig:Fish1_second},

\begin{figure}[htb]
\makebox[\textwidth][c]{%
%
}
\caption{}
\label{fig:Fish1_second}
\end{figure}

\noindent in which case $Z(I_5) / \wt(I_6)$ is equal to $z_{r}^{-n} - v_{\vphantom{r}}^{n} z_{r}^{n}$ also. 
\end{proof}

\begin{proof}[Proof of Lemma~\ref{thm:Fish2} (fish relation, type $\varGamma \varGamma$)]
There are two choices for the pair $(\varepsilon_1, \varepsilon_2)$ of decorated spins, and each choice yields exactly two admissible states of $I_{12}$.

\begin{figure}[htb]
\centering
\makebox[\textwidth][c]{%
\hfill\begin{minipage}[b]{0.20\textwidth}
\centering
%
\subcaption{}
\label{fig:Fish2_first}
\end{minipage}%
\hspace{1.0cm}%
\begin{minipage}[b]{0.20\textwidth}
\centering
%
%
\subcaption{}
\label{fig:Fish2_second}
\end{minipage}\hfill
}
\caption{}
\label{fig:Fish2_cases}
\end{figure}

\noindent If $(\varepsilon_1, \varepsilon_2) = (+, -)$, then the admissible states of $I_{12}$ are shown in Figure~\ref{fig:Fish2_first}, in which case $Z^{\ast} (I_{12}) / \wt^{\ast} (I_{13})$ is equal to
\begin{equation}
\label{E:california}
	\dfrac{
	(z_{r}^{-n} - z_{r}^{n}) 
	z_{r}^{-1}
	+ (1-v) z_{r}^{n-1} 
	}{ z_{r}^{-1} } \, .
\end{equation}

\noindent If $(\varepsilon_1, \varepsilon_2) = (-, +)$, then the admissible states of $I_{12}$ are shown in Figure~\ref{fig:Fish2_second}, in which case $Z^{\ast} (I_{12}) / \wt^{\ast} (I_{13})$ is equal to
\begin{equation}
\label{E:colorado}
	\dfrac{
	v(z_{r}^{-n} - z_{r}^{n}) z_{r}
	+ (1-v) z_{r}^{-(n-1)} 
	}{ z_{r} } \, .
\end{equation}
The expressions in~\eqref{E:california} and~\eqref{E:colorado} are equal to $z_{r}^{-n} - v z_{r}^{n}$.
\end{proof}

\begin{proof}[Proof of Lemma~\ref{thm:Fish3} (fish relation, type $\varGamma \varDelta$)]
There are two choices for the pair $(\varepsilon_1, \varepsilon_2)$ of decorated spins, and each choice yields exactly two admissible states of $I_{19}$.

\begin{figure}[htb]
\centering
\makebox[\textwidth][c]{%
\hfill\begin{minipage}[b]{0.20\textwidth}
\centering
%
\subcaption{}
\label{fig:Fish3_first}
\end{minipage}%
\hspace{1.0cm}%
\begin{minipage}[b]{0.20\textwidth}
\centering
%
%
\subcaption{}
\label{fig:Fish3_second}
\end{minipage}\hfill
}
\caption{}
\label{fig:Fish3_cases}
\end{figure}

\noindent If $(\varepsilon_1, \varepsilon_2) = (+, -)$, then the admissible states of $I_{19}$ are shown in Figure~\ref{fig:Fish3_first}, in which case $Z(I_{19}) / \wt(I_{20})$ is equal to
\begin{equation}
\label{E:washington}
	\dfrac{
	(z_{r}^{n} - z_{r}^{-n}) z_{r}
	+
	(1 - v) z_{r}^{-(n-1)} 
	}{ z_{r} } \, .
\end{equation}
If $(\varepsilon_1, \varepsilon_2) = (-, +)$, then the admissible states of $I_{19}$ are shown in Figure~\ref{fig:Fish3_second}, in which case $Z(I_{19}) / \wt(I_{20})$ is equal to
\begin{equation}
\label{E:virginia}
	\dfrac{ 
	(v^2 z_{r}^{-n} - z_{r}^{n}) z_{r}^{-1}
	+ 
	(1 - v) z_{r}^{n-1} 
	}{ g(0) z_{r}^{-1} }
	\, ,
\end{equation}
where $g(0) = - v$. The expressions in~\eqref{E:washington} and~\eqref{E:virginia} are equal to $z_{r}^{n} - v z_{r}^{-n}$.
\end{proof}

\section{Cases of the Yang--Baxter Equation ($\Delta \Delta$ Ice)}
\label{app:appendixDD}

In the cases below, $a$ is always an element of $\{ 1, 2,\ldots,n-1\}$, the set of least positive residue representatives modulo $n$, unless specified otherwise. Each case below includes two tables of the following form.

\begin{table}[H]
\begin{tabular}[t]{@{}ccc@{}}
	\toprule 
	\raisebox{1.0pt}{$\dspin{e_1}{\alpha_1}$}
	& \raisebox{1.0pt}{$\dspin{e_2}{\alpha_2}$}
	& Weight$\vphantom{\raisebox{0.5pt}{$\dspin{f_1}{\omega_1}$}}$ \\
	\bottomrule 
\end{tabular}\qquad%
\begin{tabular}[t]{ccc}
	\toprule 
	\raisebox{0.5pt}{$\dspin{f_1}{\omega_1}$}
	& \raisebox{0.5pt}{$\dspin{f_2}{\omega_2}$}
	& Weight \\
	\bottomrule 
\end{tabular}
\end{table}

The table on the left is for the left-hand side of the YBE, and the table on the right is for the right-hand side. See Figure~\ref{fig:YBE}. Each table gives the two interior decorated spins for each admissible state; the third interior spin is omitted, since it can be determined from the given data. If another integer $b$ appears, both $a$ and $b$ will be distinct elements of $\{1,2,\ldots,n-1\}$ (and likewise for other integers $c$ and $d$).\\

\noindent \textbf{Case 1:} $(
	\dspin{c_1}{\varepsilon_1}, 
	\dspin{c_2}{\varepsilon_2}, 
	\varepsilon_3,
	\dspin{c_4}{\varepsilon_4}, 
	\dspin{c_5}{\varepsilon_5}, 
	\varepsilon_6) = 
	(\dspin{0}{+},
	\dspin{0}{+},
	+,
	\dspin{0}{+},
	\dspin{0}{+},
	+)$.

\begin{table}[H]
\begin{tabular}[t]{@{}ccc@{}}
	\toprule 
	$\dspin{0}{+}$
	& $\dspin{0}{+}$
	&  $z_{1}^n - v z_{2}^n$ \\
	\bottomrule
\end{tabular}\qquad%
\begin{tabular}[t]{@{}ccc@{}}
	\toprule 
	$\dspin{0}{+}$
	& $\dspin{0}{+}$
	& $z_{1}^n - v z_{2}^n$ \\
	\bottomrule
\end{tabular}
\end{table}

\noindent\textbf{Case 2:} $(\varepsilon_1, \varepsilon_2, \varepsilon_3, \varepsilon_4, \varepsilon_5, \varepsilon_6) = (+,+,+,-,-,+)$. No admissible states exist.

\noindent\textbf{Case 3:} $(
	\dspin{c_1}{\varepsilon_1}, 
	\dspin{c_2}{\varepsilon_2}, 
	\varepsilon_3,
	\dspin{c_4}{\varepsilon_4}, 
	\dspin{c_5}{\varepsilon_5}, 
	\varepsilon_6) = 
	(\dspin{0}{+},
	\dspin{0}{-},
	+,
	\dspin{1}{-},
	\dspin{0}{+},
	+)$.

\begin{table}[H]
\begin{tabular}[t]{@{}ccc@{}}
	\toprule 
	$\dspin{0}{-}$
	& $\dspin{0}{+}$
	&  $(1-v)z_{1}^n z_2$ \\
	\bottomrule
\end{tabular}\qquad%
\begin{tabular}[t]{@{}ccc@{}}
	\toprule 
	$\dspin{1}{-}$
	& $\dspin{0}{+}$
	& $(1-v) (z_{1}^n - z_{2}^n) z_2$ \\
	\midrule 
	$\dspin{0}{+}$
	& $\dspin{1}{-}$
	& $(1-v) z_{2}^{n+1}$ \\
	\bottomrule
\end{tabular}
\end{table}

\noindent\textbf{Case 4a:} $(\dspin{c_1}{\varepsilon_1}, 
	\dspin{c_2}{\varepsilon_2}, 
	\varepsilon_3,
	\dspin{c_4}{\varepsilon_4}, 
	\dspin{c_5}{\varepsilon_5}, 
	\varepsilon_6) = 
	(\dspin{0}{+},
	\dspin{a}{-},
	+,
	\dspin{0}{+},
	\dspin{a{+}1}{-},
	+)$.

\begin{table}[H]
\begin{tabular}[t]{@{}ccc@{}}
	\toprule 
	$\dspin{0}{+}$
	& $\dspin{a}{-}$
	&  $v(z_{1}^n - z_{2}^n) z_2$ \\
	\bottomrule
\end{tabular}\qquad%
\begin{tabular}[t]{@{}ccc@{}}
	\toprule 
	$\dspin{0}{+}$
	& $\dspin{a{+}1}{-}$
	& $v(z_{1}^n - z_{2}^n) z_2$ \\
	\bottomrule
\end{tabular}
\end{table}

\noindent\textbf{Case 4b:} $(\dspin{c_1}{\varepsilon_1}, 
	\dspin{c_2}{\varepsilon_2}, 
	\varepsilon_3,
	\dspin{c_4}{\varepsilon_4}, 
	\dspin{c_5}{\varepsilon_5}, 
	\varepsilon_6) = 
	(\dspin{0}{+},
	\dspin{0}{-},
	+,
	\dspin{0}{+},
	\dspin{1}{-},
	+)$.

\begin{table}[H]
\begin{tabular}[t]{@{}ccc@{}}
	\toprule  
	$\dspin{0}{+}$
	& $\dspin{0}{-}$
	&  $v(z_{1}^n - z_{2}^n) z_2$ \\
	\midrule 
	$\dspin{0}{-}$
	& $\dspin{0}{+}$
	& $(1-v)^2 z_{1}^n z_2$ \\
	\bottomrule
\end{tabular}\qquad%
\begin{tabular}[t]{@{}ccc@{}}
	\toprule 
	$\dspin{0}{+}$
	& $\dspin{1}{-}$
	& $v(z_{1}^n - z_{2}^n) z_2$ \\
	\midrule 
	$\dspin{1}{-}$
	& $\dspin{0}{+}$
	& $(1-v)^2 z_{1}^n z_2$ \\
	\bottomrule
\end{tabular}
\end{table}

\noindent\textbf{Case 5:} $(\dspin{c_1}{\varepsilon_1}, 
	\dspin{c_2}{\varepsilon_2}, 
	\varepsilon_3,
	\dspin{c_4}{\varepsilon_4}, 
	\dspin{c_5}{\varepsilon_5}, 
	\varepsilon_6) = 
	(\dspin{a}{-},
	\dspin{0}{+},
	+,
	\dspin{a{+}1}{-},
	\dspin{0}{+},
	+)$.

\begin{table}[H]
\begin{tabular}[t]{@{}ccc@{}}
	\toprule 
	$\dspin{a}{-}$
	& $\dspin{0}{+}$
	&  $(z_{1}^n - z_{2}^n) z_1$ \\
	\bottomrule
\end{tabular}\qquad%
\begin{tabular}[t]{@{}ccc@{}}
	\toprule 
	$\dspin{a{+}1}{-}$
	& $\dspin{0}{+}$
	&  $(z_{1}^n - z_{2}^n) z_1$ \\
	\bottomrule
\end{tabular}
\end{table}

\noindent This is also true when $a = 0$.

\noindent\textbf{Case 6a:} $(\dspin{c_1}{\varepsilon_1}, 
	\dspin{c_2}{\varepsilon_2}, 
	\varepsilon_3,
	\dspin{c_4}{\varepsilon_4}, 
	\dspin{c_5}{\varepsilon_5}, 
	\varepsilon_6) = 
	(\dspin{a}{-},
	\dspin{0}{+},
	+,
	\dspin{0}{+},
	\dspin{a{+}1}{-},
	+)$.

\begin{table}[H]
\begin{tabular}[t]{@{}ccc@{}}
	\toprule 
	$\dspin{0}{+}$
	& $\dspin{a}{-}$
	&  $(1-v) z_{1}^{n-a+1} z_{2}^a$ \\
	\bottomrule
\end{tabular}\qquad%
\begin{tabular}[t]{@{}ccc@{}}
	\toprule 
	$\dspin{a{+}1}{-}$
	& $\dspin{0}{+}$
	&  $(1-v) z_{1}^{n-a+1} z_{2}^a$ \\
	\bottomrule
\end{tabular}
\end{table}

\noindent\textbf{Case 6b:} $(\dspin{c_1}{\varepsilon_1}, 
	\dspin{c_2}{\varepsilon_2}, 
	\varepsilon_3,
	\dspin{c_4}{\varepsilon_4}, 
	\dspin{c_5}{\varepsilon_5}, 
	\varepsilon_6) = 
	(\dspin{0}{-},
	\dspin{0}{+},
	+,
	\dspin{0}{+},
	\dspin{1}{-},
	+)$.

\begin{table}[H]
\begin{tabular}[t]{@{}ccc@{}}
	\toprule 
	$\dspin{0}{-}$
	& $\dspin{0}{+}$
	&  $(1-v) (z_{1}^n - z_{2}^n) z_1$ \\
	\midrule 
	$\dspin{0}{+}$
	& $\dspin{0}{-}$
	& $(1-v) z_1 z_{2}^n$ \\
	\bottomrule
\end{tabular}\qquad%
\begin{tabular}[t]{@{}ccc@{}}
	\toprule 
	$\dspin{1}{-}$
	& $\dspin{0}{+}$
	& $(1-v) z_{1}^{n+1}$ \\
	\bottomrule
\end{tabular}
\end{table}

\noindent\textbf{Case 7:} $(\varepsilon_1, \varepsilon_2, \varepsilon_3, \varepsilon_4, \varepsilon_5, \varepsilon_6) = (-,-,+,+,+,+)$. No admissible states exist.

\noindent\textbf{Case 8a:} $(\dspin{c_1}{\varepsilon_1}, 
	\dspin{c_2}{\varepsilon_2}, 
	\varepsilon_3,
	\dspin{c_4}{\varepsilon_4}, 
	\dspin{c_5}{\varepsilon_5}, 
	\varepsilon_6) = 
	(\dspin{a}{-},
	\dspin{b}{-},
	+,
	\dspin{b{+}1}{-},
	\dspin{a{+}1}{-},
	+)$.

\begin{table}[H]
\begin{tabular}[t]{@{}ccc@{}}
	\toprule 
	$\dspin{b}{-}$
	& $\dspin{a}{-}$
	& $(1-v) z_{1}^{n - c + 1} z_{2}^{c+1}$ \\
	\bottomrule
\end{tabular}\qquad%
\begin{tabular}[t]{@{}ccc@{}}
	\toprule 
	$\dspin{a{+}1}{-}$
	& $\dspin{b{+}1}{-}$
	& $(1-v) z_{1}^{n - c + 1} z_{2}^{c+1}$ \\
	\bottomrule
\end{tabular}
\end{table}

\noindent Here $c \equiv a - b \pod{n}$ with $c \in [1, n-1]$.

\noindent\textbf{Case 8b:} $(\dspin{c_1}{\varepsilon_1}, 
	\dspin{c_2}{\varepsilon_2}, 
	\varepsilon_3,
	\dspin{c_4}{\varepsilon_4}, 
	\dspin{c_5}{\varepsilon_5}, 
	\varepsilon_6) = 
	(\dspin{a}{-},
	\dspin{b}{-},
	+,
	\dspin{a{+}1}{-},
	\dspin{b{+}1}{-},
	+)$.

\begin{table}[H]
\begin{tabular}[t]{@{}ccc@{}}
	\toprule 
	$\dspin{a}{-}$
	& $\dspin{b}{-}$
	& $(z_{2}^n - v z_{1}^n) z_1 z_2$ \\
	\bottomrule
\end{tabular}\qquad%
\begin{tabular}[t]{@{}ccc@{}}
	\toprule 
	$\dspin{a{+}1}{-}$
	& $\dspin{b{+}1}{-}$
	& $(z_{2}^n - v z_{1}^n) z_1 z_2$ \\
	\bottomrule
\end{tabular}
\end{table}

\noindent This is also true when $a = b$.

\noindent\textbf{Case 9:} $(\dspin{c_1}{\varepsilon_1}, 
	\dspin{c_2}{\varepsilon_2}, 
	\varepsilon_3,
	\dspin{c_4}{\varepsilon_4}, 
	\dspin{c_5}{\varepsilon_5}, 
	\varepsilon_6) = 
	(\dspin{0}{+},
	\dspin{0}{+},
	-,
	\dspin{1}{-},
	\dspin{0}{+},
	+)$.

\begin{table}[H]
\begin{tabular}[t]{@{}ccc@{}}
	\toprule 
	$\dspin{0}{+}$
	& $\dspin{0}{+}$
	&  $z_{1}^n - v z_{2}^n$ \\
	\bottomrule
\end{tabular}\qquad%
\begin{tabular}[t]{@{}ccc@{}}
	\toprule 
	$\dspin{1}{-}$
	& $\dspin{0}{+}$
	& $z_{1}^n - z_{2}^n$ \\
	\midrule 
	$\dspin{0}{+}$
	& $\dspin{1}{-}$
	& $(1-v) z_{2}^n$ \\
	\bottomrule
\end{tabular}
\end{table}

\noindent\textbf{Case 10:} $(\dspin{c_1}{\varepsilon_1}, 
	\dspin{c_2}{\varepsilon_2}, 
	\varepsilon_3,
	\dspin{c_4}{\varepsilon_4}, 
	\dspin{c_5}{\varepsilon_5}, 
	\varepsilon_6) = 
	(\dspin{0}{+},
	\dspin{0}{+},
	-,
	\dspin{0}{+},
	\dspin{1}{-},
	+)$.

\begin{table}[H]
\begin{tabular}[t]{@{}ccc@{}}
	\toprule 
	$\dspin{0}{+}$
	& $\dspin{0}{+}$
	&  $z_{1}^n - v z_{2}^n$ \\
	\bottomrule
\end{tabular}\qquad%
\begin{tabular}[t]{@{}ccc@{}}
	\toprule 
	$\dspin{0}{+}$
	& $\dspin{1}{-}$
	& $v(z_{1}^n - z_{2}^n)$ \\
	\midrule 
	$\dspin{1}{-}$
	& $\dspin{0}{+}$
	& $(1-v) z_{1}^n$ \\
	\bottomrule
\end{tabular}
\end{table}

\noindent\textbf{Case 11:} $(\varepsilon_1, \varepsilon_2, \varepsilon_3, \varepsilon_4, \varepsilon_5, \varepsilon_6) = (+, -, -, +, +, +)$. No admissible states exist.

\noindent\textbf{Case 12a:} $(\dspin{c_1}{\varepsilon_1}, 
	\dspin{c_2}{\varepsilon_2}, 
	\varepsilon_3,
	\dspin{c_4}{\varepsilon_4}, 
	\dspin{c_5}{\varepsilon_5}, 
	\varepsilon_6) =  
	(\dspin{0}{+},
	\dspin{a}{-},
	-,
	\dspin{1}{-},
	\dspin{a{+}1}{-},
	+)$.

\begin{table}[H]
\begin{tabular}[t]{@{}ccc@{}}
	\toprule 
	$\dspin{0}{+}$
	& $\dspin{a}{-}$
	&  $v(z_{1}^n - z_{2}^n) z_2$ \\
	\bottomrule
\end{tabular}\qquad%
\begin{tabular}[t]{@{}ccc@{}}
	\toprule 
	$\dspin{1}{-}$
	& $\dspin{a{+}1}{-}$
	&  $v(z_{1}^n - z_{2}^n) z_2$ \\
	\bottomrule
\end{tabular}
\end{table}

\noindent\textbf{Case 12b:} $(\dspin{c_1}{\varepsilon_1}, 
	\dspin{c_2}{\varepsilon_2}, 
	\varepsilon_3,
	\dspin{c_4}{\varepsilon_4}, 
	\dspin{c_5}{\varepsilon_5}, 
	\varepsilon_6) = 
	(\dspin{0}{+},
	\dspin{a}{-},
	-,
	\dspin{a{+}1}{-},
	\dspin{1}{-},
	+)$.

\begin{table}[H]
\begin{tabular}[t]{@{}ccc@{}}
	\toprule 
	$\dspin{a}{-}$
	& $\dspin{0}{+}$
	&  $(1-v) g(a) z_{1}^a z_{2}^{n-a+1}$ \\
	\bottomrule
\end{tabular}\qquad%
\begin{tabular}[t]{@{}ccc@{}}
	\toprule 
	$\dspin{1}{-}$
	& $\dspin{a{+}1}{-}$
	&  $(1-v) g(a) z_{1}^a z_{2}^{n-a+1}$ \\
	\bottomrule
\end{tabular}
\end{table}

\noindent\textbf{Case 12c:} $(\dspin{c_1}{\varepsilon_1}, 
	\dspin{c_2}{\varepsilon_2}, 
	\varepsilon_3,
	\dspin{c_4}{\varepsilon_4}, 
	\dspin{c_5}{\varepsilon_5}, 
	\varepsilon_6) = 
	(\dspin{0}{+},
	\dspin{0}{-},
	-,
	\dspin{1}{-},
	\dspin{1}{-},
	+)$.

\begin{table}[H]
\begin{tabular}[t]{@{}ccc@{}}
	\toprule 
	$\dspin{0}{+}$
	& $\dspin{0}{-}$
	&  $v(z_{1}^n - z_{2}^n) z_2$ \\
	\midrule 
	$\dspin{0}{-}$
	& $\dspin{0}{+}$
	& $-v(1-v)z_{1}^n z_2$ \\
	\bottomrule
\end{tabular}\qquad%
\begin{tabular}[t]{@{}ccc@{}}
	\toprule 
	$\dspin{1}{-}$
	& $\dspin{1}{-}$
	& $v (v z_{1}^n - z_{2}^n) z_2$ \\
	\bottomrule
\end{tabular}
\end{table}

\noindent\textbf{Case 13:} $(\varepsilon_1, \varepsilon_2, \varepsilon_3, \varepsilon_4, \varepsilon_5, \varepsilon_6) = (-, +, -, +, +, +)$. No admissible states exist.

\noindent\textbf{Case 14a:} $(\dspin{c_1}{\varepsilon_1}, 
	\dspin{c_2}{\varepsilon_2}, 
	\varepsilon_3,
	\dspin{c_4}{\varepsilon_4}, 
	\dspin{c_5}{\varepsilon_5}, 
	\varepsilon_6) = 
	(\dspin{a}{-},
	\dspin{0}{+},
	-,
	\dspin{a{+}1}{-},
	\dspin{1}{-},
	+)$.

\begin{table}[H]
\begin{tabular}[t]{@{}ccc@{}}
	\toprule 
	$\dspin{a}{-}$
	& $\dspin{0}{+}$
	&  $g(a) (z_{1}^n - z_{2}^n) z_1$ \\
	\bottomrule
\end{tabular}\qquad%
\begin{tabular}[t]{@{}ccc@{}}
	\toprule 
	$\dspin{a{+}1}{-}$
	& $\dspin{1}{-}$
	&  $g(a) (z_{1}^n - z_{2}^n) z_1$ \\
	\bottomrule
\end{tabular}
\end{table}

\noindent\textbf{Case 14b:} $(\dspin{c_1}{\varepsilon_1}, 
	\dspin{c_2}{\varepsilon_2}, 
	\varepsilon_3,
	\dspin{c_4}{\varepsilon_4}, 
	\dspin{c_5}{\varepsilon_5}, 
	\varepsilon_6) = 
	(\dspin{a}{-},
	\dspin{0}{+},
	-,
	\dspin{1}{-},
	\dspin{a{+}1}{-},
	+)$.

\begin{table}[H]
\begin{tabular}[t]{@{}ccc@{}}
	\toprule 
	$\dspin{0}{+}$
	& $\dspin{a}{-}$
	&  $(1-v) z_{1}^{n-a+1} z_{2}^a$ \\
	\bottomrule
\end{tabular}\qquad%
\begin{tabular}[t]{@{}ccc@{}}
	\toprule 
	$\dspin{a{+}1}{-}$
	& $\dspin{1}{-}$
	&  $(1-v) z_{1}^{n-a+1} z_{2}^a$ \\
	\bottomrule
\end{tabular}
\end{table}

\noindent\textbf{Case 14c:} $(\dspin{c_1}{\varepsilon_1}, 
	\dspin{c_2}{\varepsilon_2}, 
	\varepsilon_3,
	\dspin{c_4}{\varepsilon_4}, 
	\dspin{c_5}{\varepsilon_5}, 
	\varepsilon_6) = 
	(\dspin{0}{-},
	\dspin{0}{+},
	-,
	\dspin{1}{-},
	\dspin{1}{-},
	+)$.

\begin{table}[H]
\begin{tabular}[t]{@{}ccc@{}}
	\toprule 
	$\dspin{0}{-}$
	& $\dspin{0}{+}$
	&  $v(z_{2}^n - z_{1}^n) z_1$ \\
	\midrule 
	$\dspin{0}{+}$
	& $\dspin{0}{-}$
	& $(1-v) z_1 z_{2}^n$ \\
	\bottomrule
\end{tabular}\qquad%
\begin{tabular}[t]{@{}ccc@{}}
	\toprule 
	$\dspin{1}{-}$
	& $\dspin{1}{-}$
	& $(z_{2}^n - vz_{1}^n) z_1$ \\
	\bottomrule
\end{tabular}
\end{table}

\noindent\textbf{Case 15:} $(\varepsilon_1, \varepsilon_2, \varepsilon_3, \varepsilon_4, \varepsilon_5, \varepsilon_6) = (-, -, -, -, +, +)$. No admissible states exist.

\noindent\textbf{Case 16:} $(\varepsilon_1, \varepsilon_2, \varepsilon_3, \varepsilon_4, \varepsilon_5, \varepsilon_6) = (-, -, -, +, -, +)$. No admissible states exist.

\noindent\textbf{Case 17:} $(\varepsilon_1, \varepsilon_2, \varepsilon_3, \varepsilon_4, \varepsilon_5, \varepsilon_6) = (+, +, +, -, +, -)$. No admissible states exist.

\noindent\textbf{Case 18:} $(\varepsilon_1, \varepsilon_2, \varepsilon_3, \varepsilon_4, \varepsilon_5, \varepsilon_6) = (+, +, +, +, -, -)$. No admissible states exist.

\noindent\textbf{Case 19:} $(\dspin{c_1}{\varepsilon_1}, 
	\dspin{c_2}{\varepsilon_2}, 
	\varepsilon_3,
	\dspin{c_4}{\varepsilon_4}, 
	\dspin{c_5}{\varepsilon_5}, 
	\varepsilon_6) = 
	(\dspin{0}{+},
	\dspin{0}{-},
	+,
	\dspin{0}{+},
	\dspin{0}{+},
	-)$.

\begin{table}[H]
\begin{tabular}[t]{@{}ccc@{}}
	\toprule 
	$\dspin{0}{+}$
	& $\dspin{0}{-}$
	&  $v(1-v) (z_{1}^n - z_{2}^n) z_2$ \\
	\midrule 
	$\dspin{0}{-}$
	& $\dspin{0}{+}$
	& $(1-v)^2 z_{1}^n z_2$ \\
	\bottomrule
\end{tabular}\qquad%
\begin{tabular}[t]{@{}ccc@{}}
	\toprule 
	$\dspin{0}{+}$
	& $\dspin{0}{+}$
	& $(1-v) (z_{1}^n - v z_{2}^n) z_2$ \\
	\bottomrule
\end{tabular}
\end{table}

\noindent\textbf{Case 20:} $(\varepsilon_1, \varepsilon_2, \varepsilon_3, \varepsilon_4, \varepsilon_5, \varepsilon_6) = (+, -, +, -, -, -)$. No admissible states exist.

\noindent\textbf{Case 21:} $(\dspin{c_1}{\varepsilon_1}, 
	\dspin{c_2}{\varepsilon_2}, 
	\varepsilon_3,
	\dspin{c_4}{\varepsilon_4}, 
	\dspin{c_5}{\varepsilon_5}, 
	\varepsilon_6) = 
	(\dspin{0}{-},
	\dspin{0}{+},
	+,
	\dspin{0}{+},
	\dspin{0}{+},
	-)$.

\begin{table}[H]
\begin{tabular}[t]{@{}ccc@{}}
	\toprule 
	$\dspin{0}{-}$
	& $\dspin{0}{+}$
	&  $(1-v) (z_{1}^n - z_{2}^n) z_1$ \\
	\midrule 
	$\dspin{0}{+}$
	& $\dspin{0}{-}$
	& $(1-v)^2 z_1 z_{2}^n$ \\
	\bottomrule
\end{tabular}\qquad%
\begin{tabular}[t]{@{}ccc@{}}
	\toprule 
	$\dspin{0}{+}$
	& $\dspin{0}{+}$
	& $(1-v)(z_{1}^n - v z_{2}^n) z_1$ \\
	\bottomrule
\end{tabular}
\end{table}

\noindent\textbf{Case 22:} $(\varepsilon_1, \varepsilon_2, \varepsilon_3, \varepsilon_4, \varepsilon_5, \varepsilon_6) = (-, +, +, -, -, -)$. No admissible states exist.

\noindent\textbf{Case 23a:} $(\dspin{c_1}{\varepsilon_1}, 
	\dspin{c_2}{\varepsilon_2}, 
	\varepsilon_3,
	\dspin{c_4}{\varepsilon_4}, 
	\dspin{c_5}{\varepsilon_5}, 
	\varepsilon_6) = 
	(\dspin{n{-}1}{-},
	\dspin{0}{-},
	+,
	\dspin{0}{-},
	\dspin{0}{+},
	-)$.

\begin{table}[H]
\begin{tabular}[t]{@{}ccc@{}}
	\toprule 
	$\dspin{n{-}1}{-}$
	& $\dspin{0}{-}$
	& $(1-v) g(-1)(z_{1}^n - z_{2}^n) z_1 z_2$\\
	\bottomrule
\end{tabular}\qquad%
\begin{tabular}[t]{@{}ccc@{}}
	\toprule 
	$\dspin{0}{-}$
	& $\dspin{0}{+}$
	& $(1-v) g(-1)(z_{1}^n - z_{2}^n) z_1 z_2$\\
	\bottomrule
\end{tabular}
\end{table}

\noindent\textbf{Case 23b:} $(\dspin{c_1}{\varepsilon_1}, 
	\dspin{c_2}{\varepsilon_2}, 
	\varepsilon_3,
	\dspin{c_4}{\varepsilon_4}, 
	\dspin{c_5}{\varepsilon_5}, 
	\varepsilon_6) = 
	(\dspin{0}{-},
	\dspin{n{-}1}{-},
	+,
	\dspin{0}{-},
	\dspin{0}{+},
	-)$.

\begin{table}[H]
\begin{tabular}[t]{@{}ccc@{}}
	\toprule 
	$\dspin{n{-}1}{-}$
	& $\dspin{0}{-}$
	& $(1-v)^2 z_{1}^n z_{2}^2$ \\
	\bottomrule
\end{tabular}\qquad%
\begin{tabular}[t]{@{}ccc@{}}
	\toprule 
	$\dspin{0}{+}$
	& $\dspin{0}{-}$
	& $(1-v)^2 z_{1}^n z_{2}^2$ \\
	\bottomrule
\end{tabular}
\end{table}

\noindent\textbf{Case 23c:} $(\dspin{c_1}{\varepsilon_1}, 
	\dspin{c_2}{\varepsilon_2}, 
	\varepsilon_3,
	\dspin{c_4}{\varepsilon_4}, 
	\dspin{c_5}{\varepsilon_5}, 
	\varepsilon_6) = 
	(\dspin{0}{-},
	\dspin{0}{-},
	+,
	\dspin{1}{-},
	\dspin{0}{+},
	-)$.

\begin{table}[H]
\begin{tabular}[t]{@{}ccc@{}}
	\toprule 
	$\dspin{0}{-}$
	& $\dspin{0}{-}$
	&  $(1-v) (z_{2}^n - v z_{1}^n) z_1 z_2$ \\
	\bottomrule
\end{tabular}\qquad%
\begin{tabular}[t]{@{}ccc@{}}
	\toprule 
	$\dspin{1}{-}$
	& $\dspin{0}{+}$
	& $v (1-v) (z_{2}^n - z_{1}^n) z_1 z_2$ \\
	\midrule 
	$\dspin{0}{+}$
	& $\dspin{1}{-}$
	&  $(1-v)^2 z_1 z_{2}^{n+1}$ \\
	\bottomrule
\end{tabular}
\end{table}

\noindent\textbf{Case 23d:} $(\dspin{c_1}{\varepsilon_1}, 
	\dspin{c_2}{\varepsilon_2}, 
	\varepsilon_3,
	\dspin{c_4}{\varepsilon_4}, 
	\dspin{c_5}{\varepsilon_5}, 
	\varepsilon_6) = 
	(\dspin{a{-}1}{-},
	\dspin{0}{-},
	+,
	\dspin{a}{-},
	\dspin{0}{+},
	-)$, where $a \neq 0$ and $a \neq 1$.

\begin{table}[H]
\begin{tabular}[t]{@{}ccc@{}}
	\toprule 
	$\dspin{a{-}1}{-}$
	& $\dspin{0}{-}$
	& $(1-v) g(a-1) (z_{1}^n - z_{2}^n) z_1 z_2$ \\
	\bottomrule
\end{tabular}\qquad%
\begin{tabular}[t]{@{}ccc@{}}
	\toprule 
	$\dspin{a}{-}$
	& $\dspin{0}{+}$
	& $(1-v) g(a-1) (z_{1}^n - z_{2}^n) z_1 z_2$ \\
	\bottomrule
\end{tabular}
\end{table}

\noindent\textbf{Case 23e:} $(\dspin{c_1}{\varepsilon_1}, 
	\dspin{c_2}{\varepsilon_2}, 
	\varepsilon_3,
	\dspin{c_4}{\varepsilon_4}, 
	\dspin{c_5}{\varepsilon_5}, 
	\varepsilon_6) = 
	(\dspin{0}{-},
	\dspin{a{-}1}{-},
	+,
	\dspin{a}{-},
	\dspin{0}{+},
	-)$, where $a \neq 0$ and $a \neq 1$.

\begin{table}[H]
\begin{tabular}[t]{@{}ccc@{}}
	\toprule 
	$\dspin{a{-}1}{-}$
	& $\dspin{0}{-}$
	& $(1-v)^2 z_{1}^a z_{2}^{n-a+2}$ \\
	\bottomrule
\end{tabular}\qquad%
\begin{tabular}[t]{@{}ccc@{}}
	\toprule 
	$\dspin{0}{+}$
	& $\dspin{a}{-}$
	& $(1-v)^2 z_{1}^a z_{2}^{n-a+2}$ \\
	\bottomrule
\end{tabular}
\end{table}

\noindent\textbf{Case 24a:} $(\dspin{c_1}{\varepsilon_1}, 
	\dspin{c_2}{\varepsilon_2}, 
	\varepsilon_3,
	\dspin{c_4}{\varepsilon_4}, 
	\dspin{c_5}{\varepsilon_5}, 
	\varepsilon_6) = 
	(\dspin{a}{-},
	\dspin{0}{-},
	+,
	\dspin{0}{+},
	\dspin{a{+}1}{-},
	-)$.

\begin{table}[H]
\begin{tabular}[t]{@{}ccc@{}}
	\toprule 
	$\dspin{0}{-}$
	& $\dspin{a}{-}$
	& $(1-v)^2 g(a) z_{1}^{n-a+1} z_{2}^{a+1}$ \\
	\bottomrule
\end{tabular}\qquad%
\begin{tabular}[t]{@{}ccc@{}}
	\toprule  
	$\dspin{a{+}1}{-}$
	& $\dspin{0}{+}$
	& $(1-v)^2 g(a) z_{1}^{n-a+1} z_{2}^{a+1}$ \\
	\bottomrule
\end{tabular}
\end{table}

\noindent\textbf{Case 24b:} $(\dspin{c_1}{\varepsilon_1}, 
	\dspin{c_2}{\varepsilon_2}, 
	\varepsilon_3,
	\dspin{c_4}{\varepsilon_4}, 
	\dspin{c_5}{\varepsilon_5}, 
	\varepsilon_6) = 
	(\dspin{0}{-},
	\dspin{a}{-},
	+,
	\dspin{0}{+},
	\dspin{a{+}1}{-},
	-)$.

\begin{table}[H]
\begin{tabular}[t]{@{}ccc@{}}
	\toprule 
	$\dspin{0}{-}$
	& $\dspin{a}{-}$
	&  $v(1-v)(z_{1}^n - z_{2}^n) z_1 z_2$ \\
	\bottomrule
\end{tabular}\qquad%
\begin{tabular}[t]{@{}ccc@{}}
	\toprule 
	$\dspin{0}{+}$
	& $\dspin{a{+}1}{-}$
	&  $v(1-v)(z_{1}^n - z_{2}^n) z_1 z_2$ \\
	\bottomrule
\end{tabular}
\end{table}

\noindent\textbf{Case 24c:} $(\dspin{c_1}{\varepsilon_1}, 
	\dspin{c_2}{\varepsilon_2}, 
	\varepsilon_3,
	\dspin{c_4}{\varepsilon_4}, 
	\dspin{c_5}{\varepsilon_5}, 
	\varepsilon_6) = 
	(\dspin{0}{-},
	\dspin{0}{-},
	+,
	\dspin{0}{+},
	\dspin{1}{-},
	-)$.

\begin{table}[H]
\begin{tabular}[t]{@{}ccc@{}}
	\toprule 
	$\dspin{0}{-}$
	& $\dspin{0}{-}$
	&  $v(1-v)(v z_{1}^n - z_{2}^n) z_1 z_2$ \\
	\bottomrule
\end{tabular}\qquad%
\begin{tabular}[t]{@{}ccc@{}}
	\toprule 
	$\dspin{0}{+}$
	& $\dspin{1}{-}$
	& $v(1-v) (z_{1}^n - z_{2}^n) z_1 z_2$ \\
	\midrule 
	$\dspin{1}{-}$
	& $\dspin{0}{+}$
	&  $-v(1-v)^2 z_{1}^{n+1} z_2$ \\
	\bottomrule
\end{tabular}
\end{table}

\noindent\textbf{Case 25:} $(\dspin{c_1}{\varepsilon_1}, 
	\dspin{c_2}{\varepsilon_2}, 
	\varepsilon_3,
	\dspin{c_4}{\varepsilon_4}, 
	\dspin{c_5}{\varepsilon_5}, 
	\varepsilon_6) = 
	(\dspin{0}{+},
	\dspin{0}{+},
	-,
	\dspin{0}{+},
	\dspin{0}{+},
	-)$.

\begin{table}[H]
\begin{tabular}[t]{@{}ccc@{}}
	\toprule 
	$\dspin{0}{+}$
	& $\dspin{0}{+}$
	&  $z_{1}^n - v z_{2}^n$ \\
	\bottomrule
\end{tabular}\qquad%
\begin{tabular}[t]{@{}ccc@{}}
	\toprule 
	$\dspin{0}{+}$
	& $\dspin{0}{+}$
	& $z_{1}^n - v z_{2}^n$ \\
	\bottomrule
\end{tabular}
\end{table}

\noindent\textbf{Case 26:} $(\varepsilon_1, \varepsilon_2, \varepsilon_3, \varepsilon_4, \varepsilon_5, \varepsilon_6) = (+, +, -, -, -, -)$. No admissible states exist.

\noindent\textbf{Case 27a:} $(\dspin{c_1}{\varepsilon_1}, 
	\dspin{c_2}{\varepsilon_2}, 
	\varepsilon_3,
	\dspin{c_4}{\varepsilon_4}, 
	\dspin{c_5}{\varepsilon_5}, 
	\varepsilon_6) = 
	(\dspin{0}{+},
	\dspin{a}{-},
	-,
	\dspin{a{+}1}{-},
	\dspin{0}{+},
	-)$.

\begin{table}[H]
\begin{tabular}[t]{@{}ccc@{}}
	\toprule 
	$\dspin{a}{-}$
	& $\dspin{0}{+}$
	&  $(1-v) g(a) z_{1}^a z_{2}^{n-a+1}$ \\
	\bottomrule
\end{tabular}\qquad%
\begin{tabular}[t]{@{}ccc@{}}
	\toprule 
	$\dspin{0}{+}$
	& $\dspin{a{+}1}{-}$
	&  $(1-v) g(a) z_{1}^a z_{2}^{n-a+1}$ \\
	\bottomrule
\end{tabular}
\end{table}

\noindent\textbf{Case 27b:} $(\dspin{c_1}{\varepsilon_1}, 
	\dspin{c_2}{\varepsilon_2}, 
	\varepsilon_3,
	\dspin{c_4}{\varepsilon_4}, 
	\dspin{c_5}{\varepsilon_5}, 
	\varepsilon_6) = 
	(\dspin{0}{+},
	\dspin{0}{-},
	-,
	\dspin{1}{-},
	\dspin{0}{+},
	-)$.

\begin{table}[H]
\begin{tabular}[t]{@{}ccc@{}}
	\toprule 
	$\dspin{0}{+}$
	& $\dspin{0}{-}$
	&  $v(1-v)(z_{1}^n - z_{2}^n) z_2$ \\
	\midrule 
	$\dspin{0}{-}$
	& $\dspin{0}{+}$
	& $-v(1-v) z_{1}^n z_2$ \\
	\bottomrule
\end{tabular}\qquad%
\begin{tabular}[t]{@{}ccc@{}}
	\toprule 
	$\dspin{0}{+}$
	& $\dspin{1}{-}$
	& $-v(1-v) z_{2}^{n+1}$ \\
	\bottomrule
\end{tabular}
\end{table}

\noindent\textbf{Case 28:} $(\dspin{c_1}{\varepsilon_1}, 
	\dspin{c_2}{\varepsilon_2}, 
	\varepsilon_3,
	\dspin{c_4}{\varepsilon_4}, 
	\dspin{c_5}{\varepsilon_5}, 
	\varepsilon_6) = 
	(\dspin{0}{+},
	\dspin{a}{-},
	-,
	\dspin{0}{+},
	\dspin{a{+}1}{-},
	-)$.

\begin{table}[H]
\begin{tabular}[t]{@{}ccc@{}}
	\toprule 
	$\dspin{0}{+}$
	& $\dspin{a}{-}$
	& $v g(a) (z_{1}^n - z_{2}^n)$ \\
	\bottomrule
\end{tabular}\qquad%
\begin{tabular}[t]{@{}ccc@{}}
	\toprule 
	$\dspin{0}{+}$
	& $\dspin{a{+}1}{-}$
	& $v g(a) (z_{1}^n - z_{2}^n)$ \\
	\bottomrule
\end{tabular}
\end{table}

\noindent This is also true when $a = 0$.

\noindent\textbf{Case 29:} $(\dspin{c_1}{\varepsilon_1}, 
	\dspin{c_2}{\varepsilon_2}, 
	\varepsilon_3,
	\dspin{c_4}{\varepsilon_4}, 
	\dspin{c_5}{\varepsilon_5}, 
	\varepsilon_6) = 
	(\dspin{0}{-},
	\dspin{0}{+},
	-,
	\dspin{1}{-},
	\dspin{0}{+},
	-)$.

\begin{table}[H]
\begin{tabular}[t]{@{}ccc@{}}
	\toprule 
	$\dspin{0}{-}$
	& $\dspin{0}{+}$
	&  $v (z_{2}^n - z_{1}^n) z_1$ \\
	\midrule 
	$\dspin{0}{+}$
	& $\dspin{0}{-}$
	& $(1-v)^2 z_1 z_{2}^{n}$ \\
	\bottomrule
\end{tabular}\qquad%
\begin{tabular}[t]{@{}ccc@{}}
	\toprule 
	$\dspin{1}{-}$
	& $\dspin{0}{+}$
	& $v (z_{2}^n - z_{1}^n) z_1$ \\
	\midrule 
	$\dspin{0}{+}$
	& $\dspin{1}{-}$
	& $(1-v)^2 z_1 z_{2}^{n}$ \\
	\bottomrule
\end{tabular}
\end{table}

\noindent\textbf{Case 30a:} $(\dspin{c_1}{\varepsilon_1}, 
	\dspin{c_2}{\varepsilon_2}, 
	\varepsilon_3,
	\dspin{c_4}{\varepsilon_4}, 
	\dspin{c_5}{\varepsilon_5}, 
	\varepsilon_6) = 
	(\dspin{a}{-},
	\dspin{0}{+},
	-,
	\dspin{0}{+},
	\dspin{a{+}1}{-},
	-)$.

\begin{table}[H]
\begin{tabular}[t]{@{}ccc@{}}
	\toprule 
	$\dspin{0}{+}$
	& $\dspin{a}{-}$
	&  $(1-v) g(a) z_{1}^{n-a+1} z_{2}^{a}$ \\
	\bottomrule
\end{tabular}\qquad%
\begin{tabular}[t]{@{}ccc@{}}
	\toprule 
	$\dspin{a{+}1}{-}$
	& $\dspin{0}{+}$
	&  $(1-v) g(a) z_{1}^{n-a+1} z_{2}^{a}$ \\
	\bottomrule
\end{tabular}
\end{table}

\noindent\textbf{Case 30b:} $(\dspin{c_1}{\varepsilon_1}, 
	\dspin{c_2}{\varepsilon_2}, 
	\varepsilon_3,
	\dspin{c_4}{\varepsilon_4}, 
	\dspin{c_5}{\varepsilon_5}, 
	\varepsilon_6) = 
	(\dspin{0}{-},
	\dspin{0}{+},
	-,
	\dspin{0}{+},
	\dspin{1}{-},
	-)$.

\begin{table}[H]
\begin{tabular}[t]{@{}ccc@{}}
	\toprule 
	$\dspin{0}{+}$
	& $\dspin{0}{-}$
	&  $-v (1-v) z_1 z_{2}^n$ \\
	\bottomrule
\end{tabular}\qquad%
\begin{tabular}[t]{@{}ccc@{}}
	\toprule 
	$\dspin{0}{+}$
	& $\dspin{1}{-}$
	& $v (1 - v) (z_{1}^n - z_{2}^n) z_1$ \\
	\midrule 
	$\dspin{1}{-}$
	& $\dspin{0}{+}$
	&  $- v (1-v) z_{1}^{n+1}$ \\
	\bottomrule
\end{tabular}
\end{table}

\noindent\textbf{Case 31:} $(\varepsilon_1, \varepsilon_2, \varepsilon_3, \varepsilon_4, \varepsilon_5, \varepsilon_6) = (-, -, -, +, +, -)$. No admissible states exist.

\noindent\textbf{Case 32a:} $(\dspin{c_1}{\varepsilon_1}, 
	\dspin{c_2}{\varepsilon_2}, 
	\varepsilon_3,
	\dspin{c_4}{\varepsilon_4}, 
	\dspin{c_5}{\varepsilon_5}, 
	\varepsilon_6) = 
	(\dspin{a}{-},
	\dspin{b}{-},
	-,
	\dspin{b{+}1}{-},
	\dspin{a{+}1}{-},
	-)$.

\begin{table}[H]
\begin{tabular}[t]{@{}ccc@{}}
	\toprule 
	$\dspin{b}{-}$
	& $\dspin{a}{-}$
	&  $(1-v) g(a) g(b) z_{1}^{n-c+1} z_{2}^{c+1}$ \\
	\bottomrule
\end{tabular}\qquad%
\begin{tabular}[t]{@{}ccc@{}}
	\toprule 
	$\dspin{a{+}1}{-}$
	& $\dspin{b{+}1}{-}$
	&  $(1-v) g(a) g(b) z_{1}^{n-c+1} z_{2}^{c+1}$ \\
	\bottomrule
\end{tabular}
\end{table}

\noindent Here $c \equiv a - b \pod{n}$ with $c \in [1, n-1]$.

\noindent\textbf{Case 32b:} $(\dspin{c_1}{\varepsilon_1}, 
	\dspin{c_2}{\varepsilon_2}, 
	\varepsilon_3,
	\dspin{c_4}{\varepsilon_4}, 
	\dspin{c_5}{\varepsilon_5}, 
	\varepsilon_6) = 
	(\dspin{a}{-},
	\dspin{b}{-},
	-,
	\dspin{a{+}1}{-},
	\dspin{b{+}1}{-},
	-)$.

\begin{table}[H]
\begin{tabular}[t]{@{}ccc@{}}
	\toprule 
	$\dspin{a}{-}$
	& $\dspin{b}{-}$
	& $g(a) g(b) g(a-b) (z_{1}^n - z_{2}^n) z_1 z_2$ \\
	\bottomrule
\end{tabular}\qquad%
\begin{tabular}[t]{@{}ccc@{}}
	\toprule 
	$\dspin{a{+}1}{-}$
	& $\dspin{b{+}1}{-}$
	& $g(a) g(b) g(a-b) (z_{1}^n - z_{2}^n) z_1 z_2$ \\
	\bottomrule
\end{tabular}
\end{table}

\noindent\textbf{Case 32c:} $(\dspin{c_1}{\varepsilon_1}, 
	\dspin{c_2}{\varepsilon_2}, 
	\varepsilon_3,
	\dspin{c_4}{\varepsilon_4}, 
	\dspin{c_5}{\varepsilon_5}, 
	\varepsilon_6) = 
	(\dspin{a}{-},
	\dspin{a}{-},
	-,
	\dspin{a{+}1}{-},
	\dspin{a{+}1}{-},
	-)$.

\begin{table}[H]
\begin{tabular}[t]{@{}ccc@{}}
	\toprule 
	$\dspin{a}{-}$
	& $\dspin{a}{-}$
	& $g(a) g(a) (z_{2}^n - v z_{1}^n) z_1 z_2$ \\
	\bottomrule
\end{tabular}\qquad%
\begin{tabular}[t]{@{}ccc@{}}
	\toprule 
	$\dspin{a{+}1}{-}$
	& $\dspin{a{+}1}{-}$
	& $g(a) g(a) (z_{2}^n - v z_{1}^n) z_1 z_2$ \\
	\bottomrule
\end{tabular}
\end{table}

\section{Cases of the Yang--Baxter Equation ($\Delta \Gamma$ Ice)}
\label{app:appendixDG}

\noindent\textbf{Case 1:} $(
	\dspin{c_1}{\varepsilon_1}, 
	\dspin{c_2}{\varepsilon_2}, 
	\varepsilon_3,
	\dspin{c_4}{\varepsilon_4}, 
	\dspin{c_5}{\varepsilon_5}, 
	\varepsilon_6) = 
	(\dspin{0}{+},
	\dspin{a}{+},
	+,
	\dspin{0}{+},
	\dspin{a{-}1}{+},
	+)$.

\begin{table}[H]
\begin{tabular}[t]{@{}ccc@{}}
	\toprule 
	$\dspin{0}{+}$
	& $\dspin{a}{+}$
	& $z_{2}^n - v^n z_{1}^n$ \\
	\bottomrule
\end{tabular}\qquad%
\begin{tabular}[t]{@{}ccc@{}}
	\toprule 
	$\dspin{0}{+}$
	& $\dspin{a{-}1}{+}$
	& $z_{2}^n - v^n z_{1}^n$ \\
	\bottomrule
\end{tabular}
\end{table}

\noindent This is also true when $a = 0$.

\noindent\textbf{Case 2:} $(\varepsilon_1, \varepsilon_2, \varepsilon_3, \varepsilon_4, \varepsilon_5, \varepsilon_6) = (+,+,+,-,-,+)$. No admissible states exist.

\noindent\textbf{Case 3a:} $(
	\dspin{c_1}{\varepsilon_1}, 
	\dspin{c_2}{\varepsilon_2}, 
	\varepsilon_3,
	\dspin{c_4}{\varepsilon_4}, 
	\dspin{c_5}{\varepsilon_5}, 
	\varepsilon_6) = 
	(\dspin{0}{+},
	\dspin{0}{-},
	+,
	\dspin{1}{-},
	\dspin{0}{+},
	+)$.

\begin{table}[H]
\begin{tabular}[t]{@{}ccc@{}}
	\toprule 
	$\dspin{0}{-}$
	& $\dspin{1}{+}$
	&  $v^{n-1} (1-v) z_{1}^n z_2$ \\
	\bottomrule
\end{tabular}\qquad%
\begin{tabular}[t]{@{}ccc@{}}
	\toprule 
	$\dspin{1}{-}$
	& $\dspin{0}{+}$
	& $(1-v) (v^{n-1} z_{1}^n - z_{2}^n) z_2$ \\
	\midrule 
	$\dspin{0}{+}$
	& $\dspin{0}{-}$
	& $(1-v) z_{2}^{n+1}$ \\
	\bottomrule
\end{tabular}
\end{table}

\noindent\textbf{Case 3b:} $(
	\dspin{c_1}{\varepsilon_1}, 
	\dspin{c_2}{\varepsilon_2}, 
	\varepsilon_3,
	\dspin{c_4}{\varepsilon_4}, 
	\dspin{c_5}{\varepsilon_5}, 
	\varepsilon_6) = 
	(\dspin{0}{+},
	\dspin{0}{-},
	+,
	\dspin{a{+}1}{-},
	\dspin{b{-}1}{+},
	+)$, where $a + b \equiv 1 \pod{n}$.

\begin{table}[H]
\begin{tabular}[t]{@{}ccc@{}}
	\toprule 
	$\dspin{a}{-}$
	& $\dspin{b}{+}$
	&  $v^{a-1} (1-v) z_{1}^a z_{2}^{n-a+1}$ \\
	\bottomrule
\end{tabular}\qquad%
\begin{tabular}[t]{@{}ccc@{}}
	\toprule 
	$\dspin{1}{-}$
	& $\dspin{0}{+}$
	& $v^{a-1} (1-v)^2 z_{1}^a z_{2}^{n-a+1}$ \\
	\midrule 
	$\dspin{0}{+}$
	& $\dspin{0}{-}$
	& $v^a (1-v) z_{1}^a z_{2}^{n-a+1}$ \\
	\bottomrule
\end{tabular}
\end{table}

\noindent This is also true when $a = b$ or $b = 0$.

\noindent\textbf{Case 4:} $(\dspin{c_1}{\varepsilon_1}, 
	\dspin{c_2}{\varepsilon_2}, 
	\varepsilon_3,
	\dspin{c_4}{\varepsilon_4}, 
	\dspin{c_5}{\varepsilon_5}, 
	\varepsilon_6) = 
	(\dspin{0}{+},
	\dspin{0}{-},
	+,
	\dspin{0}{+},
	\dspin{0}{-},
	+)$.

\begin{table}[H]
\begin{tabular}[t]{@{}ccc@{}}
	\toprule  
	$\dspin{0}{+}$
	& $\dspin{0}{-}$
	&  $(z_{2}^n - v^{n+1} z_{1}^n) z_2$ \\
	\midrule 
	$\dspin{0}{-}$
	& $\dspin{1}{+}$
	& $v^{n-1} (1-v)^2 z_{1}^n z_2$ \\
	\bottomrule
\end{tabular}\qquad%
\begin{tabular}[t]{@{}ccc@{}}
	\toprule 
	$\dspin{0}{+}$
	& $\dspin{0}{-}$
	&  $(z_{2}^n - v^{n+1} z_{1}^n) z_2$ \\
	\midrule 
	$\dspin{1}{-}$
	& $\dspin{0}{+}$
	& $v^{n-1} (1-v)^2 z_{1}^n z_2$ \\
	\bottomrule
\end{tabular}
\end{table}

\noindent\textbf{Case 5a:} $(\dspin{c_1}{\varepsilon_1}, 
	\dspin{c_2}{\varepsilon_2}, 
	\varepsilon_3,
	\dspin{c_4}{\varepsilon_4}, 
	\dspin{c_5}{\varepsilon_5}, 
	\varepsilon_6) = 
	(\dspin{a}{-},
	\dspin{b}{+},
	+,
	\dspin{a{+}1}{-},
	\dspin{b{-}1}{+},
	+)$, where $a + b \equiv 1 \pod{n}$.

\begin{table}[H]
\begin{tabular}[t]{@{}ccc@{}}
	\toprule 
	$\dspin{a}{-}$
	& $\dspin{b}{+}$
	& $(v^{n-1} z_{1}^n - z_{2}^n) z_1$ \\
	\bottomrule
\end{tabular}\qquad%
\begin{tabular}[t]{@{}ccc@{}}
	\toprule 
	$\dspin{a{+}1}{-}$
	& $\dspin{b{-}1}{+}$
	& $(v^{n-1} z_{1}^n - z_{2}^n) z_1$ \\
	\bottomrule
\end{tabular}
\end{table}

\noindent This is also true when $a = b$ or $a = 0$.

\noindent\textbf{Case 5b:} $(\dspin{c_1}{\varepsilon_1}, 
	\dspin{c_2}{\varepsilon_2}, 
	\varepsilon_3,
	\dspin{c_4}{\varepsilon_4}, 
	\dspin{c_5}{\varepsilon_5}, 
	\varepsilon_6) = 
	(\dspin{a}{-},
	\dspin{b}{+},
	+,
	\dspin{a{+}1}{-},
	\dspin{b{-}1}{+},
	+)$, where $a + b \not\equiv 1 \pod{n}$.

\begin{table}[H]
\begin{tabular}[t]{@{}ccc@{}}
	\toprule 
	$\dspin{a}{-}$
	& $\dspin{b}{+}$
	& $g(a+b-1)^{-1} (z_{2}^n - v^n z_{1}^n) z_1$ \\
	\bottomrule
\end{tabular}\qquad%
\begin{tabular}[t]{@{}ccc@{}}
	\toprule 
	$\dspin{a{+}1}{-}$
	& $\dspin{b{-}1}{+}$
	& $g(a+b-1)^{-1} (z_{2}^n - v^n z_{1}^n) z_1$ \\
	\bottomrule
\end{tabular}
\end{table}

\noindent This is also true when $a = b$ or $a = 0$.

\noindent\textbf{Case 5c:} $(\dspin{c_1}{\varepsilon_1}, 
	\dspin{c_2}{\varepsilon_2}, 
	\varepsilon_3,
	\dspin{c_4}{\varepsilon_4}, 
	\dspin{c_5}{\varepsilon_5}, 
	\varepsilon_6) = 
	(\dspin{a}{-},
	\dspin{b}{+},
	+,
	\dspin{c{+}1}{-},
	\dspin{d{-}1}{+},
	+)$, where $a + b \equiv c+d \equiv 1 \pod{n}$ and $a \neq c$.

\begin{table}[H]
\begin{tabular}[t]{@{}ccc@{}}
	\toprule 
	$\dspin{c}{-}$
	& $\dspin{d}{+}$
	& \footnotesize$v^{c-a-1} (1-v) z_{1}^{c-a+1} z_{2}^{a-c}  \begin{cases} v^n z_{1}^n &\text{if $a > c$} \\ z_{2}^n &\text{if $a < c$} \end{cases}$ \\
	\bottomrule
\end{tabular}\quad%
\begin{tabular}[t]{@{}ccc@{}}
	\toprule 
	$\dspin{a{+}1}{-}$
	& $\dspin{b{-}1}{+}$
	& \footnotesize$v^{c-a-1} (1-v) z_{1}^{c-a+1} z_{2}^{a-c}  \begin{cases} v^n z_{1}^n &\text{if $a > c$} \\ z_{2}^n &\text{if $a < c$} \end{cases}$ \\
	\bottomrule
\end{tabular}
\end{table}

\noindent This is also true when $a = b$, $a = 0$, $c = d$, or $c = 0$.

\noindent\textbf{Case 6a:} $(\dspin{c_1}{\varepsilon_1}, 
	\dspin{c_2}{\varepsilon_2}, 
	\varepsilon_3,
	\dspin{c_4}{\varepsilon_4}, 
	\dspin{c_5}{\varepsilon_5}, 
	\varepsilon_6) = 
	(\dspin{a}{-},
	\dspin{b}{+},
	+,
	\dspin{0}{+},
	\dspin{0}{-},
	+)$, where $a + b \equiv 1 \pod{n}$.

\begin{table}[H]
\begin{tabular}[t]{@{}ccc@{}}
	\toprule 
	$\dspin{0}{-}$
	& $\dspin{1}{+}$
	& $v^{n-a-1}(1-v)^2 z_{1}^{n-a+1} z_{2}^{a} $ \\
	\midrule 
	$\dspin{0}{+}$
	& $\dspin{0}{-}$
	& $v^{n-a} (1-v) z_{1}^{n-a+1} z_{2}^{a}$ \\
	\bottomrule
\end{tabular}\qquad%
\begin{tabular}[t]{@{}ccc@{}}
	\toprule 
	$\dspin{a{+}1}{-}$
	& $\dspin{b{-}1}{+}$
	& $v^{n-a-1} (1-v) z_{1}^{n-a+1} z_{2}^a$ \\
	\bottomrule
\end{tabular}
\end{table}

\noindent This is also true when $a = b$, $a = 0$, or $b = 0$.

\noindent\textbf{Case 6b:} $(\dspin{c_1}{\varepsilon_1}, 
	\dspin{c_2}{\varepsilon_2}, 
	\varepsilon_3,
	\dspin{c_4}{\varepsilon_4}, 
	\dspin{c_5}{\varepsilon_5}, 
	\varepsilon_6) = 
	(\dspin{0}{-},
	\dspin{1}{+},
	+,
	\dspin{0}{+},
	\dspin{0}{-},
	+)$.

\begin{table}[H]
\begin{tabular}[t]{@{}ccc@{}}
	\toprule 
	$\dspin{0}{-}$
	& $\dspin{1}{+}$
	&  $(1-v) (v^{n-1} z_{1}^n - z_{2}^n) z_1$ \\
	\midrule 
	$\dspin{0}{+}$
	& $\dspin{0}{-}$
	& $(1-v) z_1 z_{2}^n$ \\
	\bottomrule
\end{tabular}\qquad%
\begin{tabular}[t]{@{}ccc@{}}
	\toprule 
	$\dspin{1}{-}$
	& $\dspin{0}{+}$
	& $v^{n-1} (1-v) z_{1}^{n+1}$ \\
	\bottomrule
\end{tabular}
\end{table}

\noindent\textbf{Case 7:} $(\varepsilon_1, \varepsilon_2, \varepsilon_3, \varepsilon_4, \varepsilon_5, \varepsilon_6) = (-, -, +, +, +, +)$. No admissible states exist.

\noindent\textbf{Case 8:} $(\dspin{c_1}{\varepsilon_1}, 
	\dspin{c_2}{\varepsilon_2}, 
	\varepsilon_3,
	\dspin{c_4}{\varepsilon_4}, 
	\dspin{c_5}{\varepsilon_5}, 
	\varepsilon_6) = 
	(\dspin{a}{-},
	\dspin{0}{-},
	+,
	\dspin{a{+}1}{-},
	\dspin{0}{-},
	+)$.

\begin{table}[H]
\begin{tabular}[t]{@{}ccc@{}}
	\toprule 
	$\dspin{a}{-}$
	& $\dspin{0}{-}$
	& $(z_{2}^n - v^n z_{1}^n) z_1 z_2$ \\
	\bottomrule
\end{tabular}\qquad%
\begin{tabular}[t]{@{}ccc@{}}
	\toprule 
	$\dspin{a{+}1}{-}$
	& $\dspin{0}{-}$
	& $(z_{2}^n - v^n z_{1}^n) z_1 z_2$ \\
	\bottomrule
\end{tabular}
\end{table}

\noindent This is also true when $a = 0$.

\noindent\textbf{Case 9:} $(\dspin{c_1}{\varepsilon_1}, 
	\dspin{c_2}{\varepsilon_2}, 
	\varepsilon_3,
	\dspin{c_4}{\varepsilon_4}, 
	\dspin{c_5}{\varepsilon_5}, 
	\varepsilon_6) = 
	(\dspin{0}{+},
	\dspin{1}{+},
	-,
	\dspin{1}{-},
	\dspin{0}{+},
	+)$.

\begin{table}[H]
\begin{tabular}[t]{@{}ccc@{}}
	\toprule 
	$\dspin{0}{+}$
	& $\dspin{1}{+}$
	&  $z_{2}^n - v^n z_{1}^n$ \\
	\bottomrule
\end{tabular}\qquad%
\begin{tabular}[t]{@{}ccc@{}}
	\toprule 
	$\dspin{0}{+}$
	& $\dspin{0}{-}$
	& $(1-v) z_{2}^n$ \\
	\midrule 
	$\dspin{1}{-}$
	& $\dspin{0}{+}$
	& $v z_{2}^n - v^n z_{1}^n$ \\
	\bottomrule
\end{tabular}
\end{table}

\noindent\textbf{Case 10:} $(\dspin{c_1}{\varepsilon_1}, 
	\dspin{c_2}{\varepsilon_2}, 
	\varepsilon_3,
	\dspin{c_4}{\varepsilon_4}, 
	\dspin{c_5}{\varepsilon_5}, 
	\varepsilon_6) = 
	(\dspin{0}{+},
	\dspin{1}{+},
	-,
	\dspin{0}{+},
	\dspin{0}{-},
	+)$.

\begin{table}[H]
\begin{tabular}[t]{@{}ccc@{}}
	\toprule 
	$\dspin{0}{+}$
	& $\dspin{1}{+}$
	&  $z_{2}^n - v^n z_{1}^n$ \\
	\bottomrule
\end{tabular}\qquad%
\begin{tabular}[t]{@{}ccc@{}}
	\toprule 
	$\dspin{0}{+}$
	& $\dspin{0}{-}$
	& $z_{2}^n - v^{n+1} z_{1}^n$ \\
	\midrule 
	$\dspin{1}{-}$
	& $\dspin{0}{+}$
	& $-v^n (1-v) z_{1}^n$ \\
	\bottomrule
\end{tabular}
\end{table}

\noindent\textbf{Case 11:} $(\varepsilon_1, \varepsilon_2, \varepsilon_3, \varepsilon_4, \varepsilon_5, \varepsilon_6) = (+, -, -, +, +, +)$. No admissible states exist.

\noindent\textbf{Case 12:} $(\dspin{c_1}{\varepsilon_1}, 
	\dspin{c_2}{\varepsilon_2}, 
	\varepsilon_3,
	\dspin{c_4}{\varepsilon_4}, 
	\dspin{c_5}{\varepsilon_5}, 
	\varepsilon_6) =  
	(\dspin{0}{+},
	\dspin{0}{-},
	-,
	\dspin{1}{-},
	\dspin{0}{-},
	+)$.

\begin{table}[H]
\begin{tabular}[t]{@{}ccc@{}}
	\toprule 
	$\dspin{0}{+}$
	& $\dspin{0}{-}$
	&  $(z_{2}^n - v^{n+1} z_{1}^n) z_2$ \\
	\midrule 
	$\dspin{0}{-}$
	& $\dspin{1}{+}$
	& $- v^n (1-v)  z_{1}^n z_2$ \\
	\bottomrule
\end{tabular}\qquad%
\begin{tabular}[t]{@{}ccc@{}}
	\toprule 
	$\dspin{1}{-}$
	& $\dspin{0}{-}$
	& $(z_{2}^n - v^n z_{1}^n) z_2$ \\
	\bottomrule
\end{tabular}
\end{table}

\noindent\textbf{Case 13:} $(\varepsilon_1, \varepsilon_2, \varepsilon_3, \varepsilon_4, \varepsilon_5, \varepsilon_6) = (-, +, -, +, +, +)$. No admissible states exist.

\noindent\textbf{Case 14a:} $(\dspin{c_1}{\varepsilon_1}, 
	\dspin{c_2}{\varepsilon_2}, 
	\varepsilon_3,
	\dspin{c_4}{\varepsilon_4}, 
	\dspin{c_5}{\varepsilon_5}, 
	\varepsilon_6) = 
	(\dspin{a}{-},
	\dspin{1}{+},
	-,
	\dspin{a{+}1}{-},
	\dspin{0}{-},
	+)$.

\begin{table}[H]
\begin{tabular}[t]{@{}ccc@{}}
	\toprule 
	$\dspin{a}{-}$
	& $\dspin{1}{+}$
	& $(z_{2}^n - v^n z_{1}^n) z_1$ \\
	\bottomrule
\end{tabular}\qquad%
\begin{tabular}[t]{@{}ccc@{}}
	\toprule 
	$\dspin{a{+}1}{-}$
	& $\dspin{0}{-}$
	& $(z_{2}^n - v^n z_{1}^n) z_1$ \\
	\bottomrule
\end{tabular}
\end{table}

\noindent\textbf{Case 14b:} $(\dspin{c_1}{\varepsilon_1}, 
	\dspin{c_2}{\varepsilon_2}, 
	\varepsilon_3,
	\dspin{c_4}{\varepsilon_4}, 
	\dspin{c_5}{\varepsilon_5}, 
	\varepsilon_6) = 
	(\dspin{0}{-},
	\dspin{1}{+},
	-,
	\dspin{1}{-},
	\dspin{0}{-},
	+)$.

\begin{table}[H]
\begin{tabular}[t]{@{}ccc@{}}
	\toprule 
	$\dspin{0}{+}$
	& $\dspin{0}{-}$
	&  $(1-v) z_1 z_{2}^n$ \\
	\midrule 
	$\dspin{0}{-}$
	& $\dspin{1}{+}$
	& $v(z_{2}^n - v^{n-1} z_{1}^n) z_1$ \\
	\bottomrule
\end{tabular}\qquad%
\begin{tabular}[t]{@{}ccc@{}}
	\toprule 
	$\dspin{1}{-}$
	& $\dspin{0}{-}$
	& $(z_{2}^n - v^n z_{1}^n) z_1$ \\
	\bottomrule
\end{tabular}
\end{table}

\noindent\textbf{Case 15:} $(\varepsilon_1, \varepsilon_2, \varepsilon_3, \varepsilon_4, \varepsilon_5, \varepsilon_6) = (-, -, -, -, +, +)$. No admissible states exist.

\noindent\textbf{Case 16:} $(\varepsilon_1, \varepsilon_2, \varepsilon_3, \varepsilon_4, \varepsilon_5, \varepsilon_6) = (-, -, -, +, -, +)$. No admissible states exist.

\noindent\textbf{Case 17:} $(\varepsilon_1, \varepsilon_2, \varepsilon_3, \varepsilon_4, \varepsilon_5, \varepsilon_6) = (+, +, +, -, +, -)$. No admissible states exist.

\noindent\textbf{Case 18:} $(\varepsilon_1, \varepsilon_2, \varepsilon_3, \varepsilon_4, \varepsilon_5, \varepsilon_6) = (+, +, +, +, -, -)$. No admissible states exist.

\noindent\textbf{Case 19:} $(\dspin{c_1}{\varepsilon_1}, 
	\dspin{c_2}{\varepsilon_2}, 
	\varepsilon_3,
	\dspin{c_4}{\varepsilon_4}, 
	\dspin{c_5}{\varepsilon_5}, 
	\varepsilon_6) = 
	(\dspin{0}{+},
	\dspin{0}{-},
	+,
	\dspin{0}{+},
	\dspin{0}{+},
	-)$.

\begin{table}[H]
\begin{tabular}[t]{@{}ccc@{}}
	\toprule 
	$\dspin{0}{+}$
	& $\dspin{0}{-}$
	& $(1 - v) (z_{2}^n - v^{n+1} z_{1}^n) z_2$ \\
	\midrule 
	$\dspin{0}{-}$
	& $\dspin{1}{+}$
	& $-v^n (1-v)^2 z_{1}^n z_2$ \\
	\bottomrule
\end{tabular}\qquad%
\begin{tabular}[t]{@{}ccc@{}}
	\toprule 
	$\dspin{0}{+}$
	& $\dspin{0}{+}$
	& $(1-v) (z_{2}^n - v^n z_{1}^n) z_2$ \\
	\bottomrule
\end{tabular}
\end{table}

\noindent\textbf{Case 20:} $(\varepsilon_1, \varepsilon_2, \varepsilon_3, \varepsilon_4, \varepsilon_5, \varepsilon_6) = (+, -, +, -, -, -)$. No admissible states exist.

\noindent\textbf{Case 21a:} $(\dspin{c_1}{\varepsilon_1}, 
	\dspin{c_2}{\varepsilon_2}, 
	\varepsilon_3,
	\dspin{c_4}{\varepsilon_4}, 
	\dspin{c_5}{\varepsilon_5}, 
	\varepsilon_6) = 
	(\dspin{0}{-},
	\dspin{1}{+},
	+,
	\dspin{0}{+},
	\dspin{0}{+},
	-)$.

\begin{table}[H]
\begin{tabular}[t]{@{}ccc@{}}
	\toprule 
	$\dspin{0}{-}$
	& $\dspin{1}{+}$
	&  $v (1-v) (z_{2}^n - v^{n-1} z_{1}^n) z_1$ \\
	\midrule 
	$\dspin{0}{+}$
	& $\dspin{0}{-}$
	& $(1-v)^2 z_1 z_{2}^n$ \\
	\bottomrule
\end{tabular}\qquad%
\begin{tabular}[t]{@{}ccc@{}}
	\toprule 
	$\dspin{0}{+}$
	& $\dspin{0}{+}$
	& $(1-v) (z_{2}^n - v^n z_{1}^n) z_1$ \\
	\bottomrule
\end{tabular}
\end{table}

\noindent\textbf{Case 21b:} $(\dspin{c_1}{\varepsilon_1}, 
	\dspin{c_2}{\varepsilon_2}, 
	\varepsilon_3,
	\dspin{c_4}{\varepsilon_4}, 
	\dspin{c_5}{\varepsilon_5}, 
	\varepsilon_6) = 
	(\dspin{0}{-},
	\dspin{a{+}1}{+},
	+,
	\dspin{0}{+},
	\dspin{a}{+},
	-)$.

\begin{table}[H]
\begin{tabular}[t]{@{}ccc@{}}
	\toprule 
	$\dspin{0}{-}$
	& $\dspin{a{+}1}{+}$
	& $(1-v) (z_{2}^n - v^n z_{1}^n) z_1$ \\
	\bottomrule
\end{tabular}\qquad%
\begin{tabular}[t]{@{}ccc@{}}
	\toprule 
	$\dspin{0}{+}$
	& $\dspin{a}{+}$
	& $(1-v) (z_{2}^n - v^n z_{1}^n) z_1$ \\
	\bottomrule
\end{tabular}
\end{table}

\noindent\textbf{Case 22:} $(\varepsilon_1, \varepsilon_2, \varepsilon_3, \varepsilon_4, \varepsilon_5, \varepsilon_6) = (-, +, +, -, -, -)$. No admissible states exist. 

\noindent\textbf{Case 23a:} $(\dspin{c_1}{\varepsilon_1}, 
	\dspin{c_2}{\varepsilon_2}, 
	\varepsilon_3,
	\dspin{c_4}{\varepsilon_4}, 
	\dspin{c_5}{\varepsilon_5}, 
	\varepsilon_6) = 
	(\dspin{a}{-},
	\dspin{0}{-},
	+,
	\dspin{a{+}1}{-},
	\dspin{0}{+},
	-)$.

\begin{table}[H]
\begin{tabular}[t]{@{}ccc@{}}
	\toprule 
	$\dspin{a}{-}$
	& $\dspin{0}{-}$
	&  $(1-v) (z_{2}^n - v^n z_{1}^n) z_1 z_2$ \\
	\bottomrule
\end{tabular}\qquad%
\begin{tabular}[t]{@{}ccc@{}}
	\toprule 
	$\dspin{a{+}1}{-}$
	& $\dspin{0}{+}$
	&  $(1-v) (z_{2}^n - v^n z_{1}^n) z_1 z_2$ \\
	\bottomrule
\end{tabular}
\end{table}

\noindent\textbf{Case 23b:} $(\dspin{c_1}{\varepsilon_1}, 
	\dspin{c_2}{\varepsilon_2}, 
	\varepsilon_3,
	\dspin{c_4}{\varepsilon_4}, 
	\dspin{c_5}{\varepsilon_5}, 
	\varepsilon_6) = 
	(\dspin{0}{-},
	\dspin{0}{-},
	+,
	\dspin{1}{-},
	\dspin{0}{+},
	-)$.

\begin{table}[H]
\begin{tabular}[t]{@{}ccc@{}}
	\toprule 
	$\dspin{0}{-}$
	& $\dspin{0}{-}$
	& $(1 - v)(z_{2}^n - v^n z_{1}^n) z_1 z_2$ \\
	\bottomrule
\end{tabular}\qquad%
\begin{tabular}[t]{@{}ccc@{}}
	\toprule 
	$\dspin{0}{+}$
	& $\dspin{0}{-}$
	& $(1 - v)^2 z_1 z_{2}^{n+1}$ \\
	\midrule 
	$\dspin{1}{-}$
	& $\dspin{0}{+}$
	& $v(1 - v) (z_{2}^n - v^{n-1} z_{1}^n) z_1 z_2$ \\
	\bottomrule
\end{tabular}
\end{table}

\noindent\textbf{Case 24:} $(\dspin{c_1}{\varepsilon_1}, 
	\dspin{c_2}{\varepsilon_2}, 
	\varepsilon_3,
	\dspin{c_4}{\varepsilon_4}, 
	\dspin{c_5}{\varepsilon_5}, 
	\varepsilon_6) = 
	(\dspin{0}{-},
	\dspin{0}{-},
	+,
	\dspin{0}{+},
	\dspin{0}{-},
	-)$.

\begin{table}[H]
\begin{tabular}[t]{@{}ccc@{}}
	\toprule 
	$\dspin{0}{-}$
	& $\dspin{0}{-}$
	& $(1 - v)(z_{2}^n - v^n z_{1}^n) z_1 z_2$ \\
	\bottomrule
\end{tabular}\qquad%
\begin{tabular}[t]{@{}ccc@{}}
	\toprule 
	$\dspin{0}{+}$
	& $\dspin{0}{-}$
	& $(1-v)(z_{2}^n - v^{n+1} z_{1}^n) z_1 z_2$ \\
	\midrule 
	$\dspin{1}{-}$
	& $\dspin{0}{+}$
	& $-v^n (1-v)^2 z_{1}^{n+1} z_2$ \\
	\bottomrule
\end{tabular}
\end{table}

\noindent\textbf{Case 25:} $(\dspin{c_1}{\varepsilon_1}, 
	\dspin{c_2}{\varepsilon_2}, 
	\varepsilon_3,
	\dspin{c_4}{\varepsilon_4}, 
	\dspin{c_5}{\varepsilon_5}, 
	\varepsilon_6) = 
	(\dspin{0}{+},
	\dspin{a}{+},
	-,
	\dspin{0}{+},
	\dspin{a{-}1}{+},
	-)$.

\begin{table}[H]
\begin{tabular}[t]{@{}ccc@{}}
	\toprule 
	$\dspin{0}{+}$
	& $\dspin{a}{+}$
	&  $g(a-1) (z_{2}^n - v^n z_{1}^n)$ \\
	\bottomrule
\end{tabular}\qquad%
\begin{tabular}[t]{@{}ccc@{}}
	\toprule 
	$\dspin{0}{+}$
	& $\dspin{a{-}1}{+}$
	&  $g(a-1) (z_{2}^n - v^n z_{1}^n)$ \\
	\bottomrule
\end{tabular}
\end{table}

\noindent This is also true when $a = 0$.

\noindent\textbf{Case 26:} $(\varepsilon_1, \varepsilon_2, \varepsilon_3, \varepsilon_4, \varepsilon_5, \varepsilon_6) = (+, +, -, -, -, -)$. No admissible states exist.

\noindent\textbf{Case 27a:} $(\dspin{c_1}{\varepsilon_1}, 
	\dspin{c_2}{\varepsilon_2}, 
	\varepsilon_3,
	\dspin{c_4}{\varepsilon_4}, 
	\dspin{c_5}{\varepsilon_5}, 
	\varepsilon_6) = 
	(\dspin{0}{+},
	\dspin{0}{-},
	-,
	\dspin{a{+}1}{-},
	\dspin{b{-}1}{+},
	-)$, where $a + b \equiv 1 \pod{n}$.

\begin{table}[H]
\begin{tabular}[t]{@{}ccc@{}}
	\toprule 
	$\dspin{a}{-}$
	& $\dspin{n{-}a{+}1}{+}$
	& $v^a (1-v) z_{1}^a z_{2}^{n-a+1}$ \\
	\bottomrule
\end{tabular}\qquad%
\begin{tabular}[t]{@{}ccc@{}}
	\toprule 
	$\dspin{0}{+}$
	& $\dspin{0}{-}$
	& $v^a (1-v) z_{1}^a z_{2}^{n-a+1}$ \\
	\bottomrule
\end{tabular}
\end{table}

\noindent This is also true when $a = b$.

\noindent\textbf{Case 27b:} $(\dspin{c_1}{\varepsilon_1}, 
	\dspin{c_2}{\varepsilon_2}, 
	\varepsilon_3,
	\dspin{c_4}{\varepsilon_4}, 
	\dspin{c_5}{\varepsilon_5}, 
	\varepsilon_6) = 
	(\dspin{0}{+},
	\dspin{0}{-},
	-,
	\dspin{1}{-},
	\dspin{0}{+},
	-)$.

\begin{table}[H]
\begin{tabular}[t]{@{}ccc@{}}
	\toprule 
	$\dspin{0}{+}$
	& $\dspin{0}{-}$
	&  $(1-v)(z_{2}^n - v^{n+1} z_{1}^n) z_2$ \\
	\midrule 
	$\dspin{0}{-}$
	& $\dspin{1}{+}$
	& $v^{n+1} (1-v) z_{1}^n z_{2}$ \\
	\bottomrule
\end{tabular}\qquad%
\begin{tabular}[t]{@{}ccc@{}}
	\toprule 
	$\dspin{0}{+}$
	& $\dspin{0}{-}$
	& $(1-v) z_{2}^{n+1}$ \\
	\bottomrule
\end{tabular}
\end{table}

\noindent\textbf{Case 28:} $(\dspin{c_1}{\varepsilon_1}, 
	\dspin{c_2}{\varepsilon_2}, 
	\varepsilon_3,
	\dspin{c_4}{\varepsilon_4}, 
	\dspin{c_5}{\varepsilon_5}, 
	\varepsilon_6) = 
	(\dspin{0}{+},
	\dspin{0}{-},
	-,
	\dspin{0}{+},
	\dspin{0}{-},
	-)$.

\begin{table}[H]
\begin{tabular}[t]{@{}ccc@{}}
	\toprule 
	$\dspin{0}{+}$
	& $\dspin{0}{-}$
	& $(z_{2}^n - v^{n+1} z_{1}^n) z_2$ \\
	\bottomrule
\end{tabular}\qquad%
\begin{tabular}[t]{@{}ccc@{}}
	\toprule 
	$\dspin{0}{+}$
	& $\dspin{0}{-}$
	& $(z_{2}^n - v^{n+1} z_{1}^n) z_2$ \\
	\bottomrule
\end{tabular}
\end{table}

\noindent\textbf{Case 29a:} $(\dspin{c_1}{\varepsilon_1}, 
	\dspin{c_2}{\varepsilon_2}, 
	\varepsilon_3,
	\dspin{c_4}{\varepsilon_4}, 
	\dspin{c_5}{\varepsilon_5}, 
	\varepsilon_6) = 
	(\dspin{0}{-},
	\dspin{1}{+},
	-,
	\dspin{1}{-},
	\dspin{0}{+},
	-)$.

\begin{table}[H]
\begin{tabular}[t]{@{}ccc@{}}
	\toprule 
	$\dspin{0}{+}$
	& $\dspin{0}{-}$
	& $(1-v)^2 z_1 z_{2}^n$ \\
	\midrule 
	$\dspin{0}{-}$
	& $\dspin{1}{+}$
	& $v^2 (v^{n-1} z_{1}^n - z_{2}^n) z_1$ \\
	\bottomrule
\end{tabular}\qquad%
\begin{tabular}[t]{@{}ccc@{}}
	\toprule 
	$\dspin{0}{+}$
	& $\dspin{0}{-}$
	& $(1-v)^2 z_1 z_{2}^n$ \\
	\midrule 
	$\dspin{1}{-}$
	& $\dspin{0}{+}$
	& $v^2 (v^{n-1} z_{1}^n - z_{2}^n) z_1$ \\
	\bottomrule
\end{tabular}
\end{table}

\noindent\textbf{Case 29b:} $(\dspin{c_1}{\varepsilon_1}, 
	\dspin{c_2}{\varepsilon_2}, 
	\varepsilon_3,
	\dspin{c_4}{\varepsilon_4}, 
	\dspin{c_5}{\varepsilon_5}, 
	\varepsilon_6) = 
	(\dspin{a}{-},
	\dspin{b}{+},
	-,
	\dspin{a{+}1}{-},
	\dspin{b{-}1}{+},
	-)$, where $a + b \equiv 1 \pod{n}$.

\begin{table}[H]
\begin{tabular}[t]{@{}ccc@{}}
	\toprule 
	$\dspin{a}{-}$
	& $\dspin{b}{+}$
	& $v (v^{n-1} z_{1}^n - z_{2}^n) z_1$ \\
	\bottomrule
\end{tabular}\qquad%
\begin{tabular}[t]{@{}ccc@{}}
	\toprule 
	$\dspin{a{+}1}{-}$
	& $\dspin{b{-}1}{+}$
	& $v (v^{n-1} z_{1}^n - z_{2}^n) z_1$ \\
	\bottomrule
\end{tabular}
\end{table}

\noindent\textbf{Case 29c:} $(\dspin{c_1}{\varepsilon_1}, 
	\dspin{c_2}{\varepsilon_2}, 
	\varepsilon_3,
	\dspin{c_4}{\varepsilon_4}, 
	\dspin{c_5}{\varepsilon_5}, 
	\varepsilon_6) = 
	(\dspin{0}{-},
	\dspin{1}{+},
	-,
	\dspin{a{+}1}{-},
	\dspin{b{-}1}{+},
	-)$, where $a + b \equiv 1 \pod{n}$.

\begin{table}[H]
\begin{tabular}[t]{@{}ccc@{}}
	\toprule 
	$\dspin{a}{-}$
	& $\dspin{b}{+}$
	& $v^a (1-v) z_{1}^{a+1} z_{2}^{n-a}$ \\
	\bottomrule
\end{tabular}\qquad%
\begin{tabular}[t]{@{}ccc@{}}
	\toprule 
	$\dspin{0}{+}$
	& $\dspin{0}{-}$
	& $v^a (1-v)^2 z_{1}^{a+1} z_{2}^{n-a}$ \\
	\midrule 
	$\dspin{1}{-}$
	& $\dspin{0}{+}$
	& $v^{a+1} (1-v) z_{1}^{a+1} z_{2}^{n-a}$ \\
	\bottomrule
\end{tabular}
\end{table}

\noindent This is also true when $a = b$.

\noindent\textbf{Case 29d:} $(\dspin{c_1}{\varepsilon_1}, 
	\dspin{c_2}{\varepsilon_2}, 
	\varepsilon_3,
	\dspin{c_4}{\varepsilon_4}, 
	\dspin{c_5}{\varepsilon_5}, 
	\varepsilon_6) = 
	(\dspin{a}{-},
	\dspin{b}{+},
	-,
	\dspin{1}{-},
	\dspin{0}{+},
	-)$, where $a + b \equiv 1 \pod{n}$.

\begin{table}[H]
\begin{tabular}[t]{@{}ccc@{}}
	\toprule 
	$\dspin{0}{+}$
	& $\dspin{0}{-}$
	& $v^{n-a} (1-v)^2 z_{1}^{n-a+1} z_{2}^a$ \\
	\midrule 
	$\dspin{0}{-}$
	& $\dspin{1}{+}$
	& $v^{n-a+1}(1-v)z_{1}^{n-a+1}z_{2}^{a}$ \\
	\bottomrule
\end{tabular}\qquad%
\begin{tabular}[t]{@{}ccc@{}}
	\toprule 
	$\dspin{a{+}1}{-}$
	& $\dspin{b{-}1}{+}$
	& $v^{n-a} (1-v) z_{1}^{n-a+1} z_{2}^{a}$ \\
	\bottomrule
\end{tabular}
\end{table}

\noindent\textbf{Case 29e:} $(\dspin{c_1}{\varepsilon_1}, 
	\dspin{c_2}{\varepsilon_2}, 
	\varepsilon_3,
	\dspin{c_4}{\varepsilon_4}, 
	\dspin{c_5}{\varepsilon_5}, 
	\varepsilon_6) = 
	(\dspin{a}{-},
	\dspin{b}{+},
	-,
	\dspin{a{+}1}{-},
	\dspin{b{-}1}{+},
	-)$, where $a + b \not \equiv 1 \pod{n}$.

\begin{table}[H]
\begin{tabular}[t]{@{}ccc@{}}
	\toprule 
	$\dspin{a}{-}$
	& $\dspin{b}{+}$
	& \footnotesize$g(a) g(b-1) g(a+b-1)^{-1} (z_{2}^n - v^n z_{1}^n)$ \\
	\bottomrule
\end{tabular}\qquad%
\begin{tabular}[t]{@{}ccc@{}}
	\toprule 
	$\dspin{a{+}1}{-}$
	& $\dspin{b{-}1}{+}$
	& \footnotesize$g(a) g(b-1) g(a+b-1)^{-1} (z_{2}^n - v^n z_{1}^n)$ \\
	\bottomrule
\end{tabular}
\end{table}

\noindent This is also true when $a = b$, $a = 0$, or $b = 0$.

\noindent\textbf{Case 30a:} $(\dspin{c_1}{\varepsilon_1}, 
	\dspin{c_2}{\varepsilon_2}, 
	\varepsilon_3,
	\dspin{c_4}{\varepsilon_4}, 
	\dspin{c_5}{\varepsilon_5}, 
	\varepsilon_6) = 
	(\dspin{0}{-},
	\dspin{1}{+},
	-,
	\dspin{0}{+},
	\dspin{0}{-},
	-)$.

\begin{table}[H]
\begin{tabular}[t]{@{}ccc@{}}
	\toprule 
	$\dspin{0}{+}$
	& $\dspin{0}{-}$
	&  $(1-v) z_1 z_{2}^n$ \\
	\bottomrule
\end{tabular}\qquad%
\begin{tabular}[t]{@{}ccc@{}}
	\toprule 
	$\dspin{1}{-}$
	& $\dspin{0}{+}$
	& $v^{n+1} (1-v) z_{1}^{n+1}$ \\
	\midrule 
	$\dspin{0}{+}$
	& $\dspin{0}{+}$
	&  $(1-v) (z_{2}^n - v^{n+1} z_{1}^n) z_1$ \\
	\bottomrule
\end{tabular}
\end{table}

\noindent\textbf{Case 30b:} $(\dspin{c_1}{\varepsilon_1}, 
	\dspin{c_2}{\varepsilon_2}, 
	\varepsilon_3,
	\dspin{c_4}{\varepsilon_4}, 
	\dspin{c_5}{\varepsilon_5}, 
	\varepsilon_6) = 
	(\dspin{a}{-},
	\dspin{b}{+},
	-,
	\dspin{0}{+},
	\dspin{0}{-},
	-)$, where $a + b \equiv 1 \pod{n}$.

\begin{table}[H]
\begin{tabular}[t]{@{}ccc@{}}
	\toprule 
	$\dspin{0}{+}$
	& $\dspin{0}{-}$
	&  $v^{n-a} (1-v) z_{1}^{n-a+1} z_{2}^a$ \\
	\bottomrule
\end{tabular}\qquad%
\begin{tabular}[t]{@{}ccc@{}}
	\toprule 
	$\dspin{a{+}1}{-}$
	& $\dspin{b-1}{+}$
	&  $v^{n-a} (1-v) z_{1}^{n-a+1} z_{2}^a$ \\
	\bottomrule
\end{tabular}
\end{table}

\noindent\textbf{Case 31:} $(\varepsilon_1, \varepsilon_2, \varepsilon_3, \varepsilon_4, \varepsilon_5, \varepsilon_6) = (-, -, -, +, +, -)$. No admissible states exist.

\noindent\textbf{Case 32:} $(\dspin{c_1}{\varepsilon_1}, 
	\dspin{c_2}{\varepsilon_2}, 
	\varepsilon_3,
	\dspin{c_4}{\varepsilon_4}, 
	\dspin{c_5}{\varepsilon_5}, 
	\varepsilon_6) = 
	(\dspin{a}{-},
	\dspin{0}{-},
	-,
	\dspin{a{+}1}{-},
	\dspin{0}{-},
	-)$.

\begin{table}[H]
\begin{tabular}[t]{@{}ccc@{}}
	\toprule 
	$\dspin{a}{-}$
	& $\dspin{0}{-}$
	& $g(a) (z_{2}^n - v^n z_{1}^n) z_1 z_2$ \\
	\bottomrule
\end{tabular}\qquad%
\begin{tabular}[t]{@{}ccc@{}}
	\toprule 
	$\dspin{a{+}1}{-}$
	& $\dspin{0}{-}$
	& $g(a) (z_{2}^n - v^n z_{1}^n) z_1 z_2$ \\
	\bottomrule
\end{tabular}
\end{table}

\noindent This is also true when $a = 0$.

\section{Cases of the Yang--Baxter Equation ($\Gamma \Delta$ Ice)}
\label{app:appendixGD}

\noindent \textbf{Case 1:} $(
	\dspin{c_1}{\varepsilon_1}, 
	\dspin{c_2}{\varepsilon_2}, 
	\varepsilon_3,
	\dspin{c_4}{\varepsilon_4}, 
	\dspin{c_5}{\varepsilon_5}, 
	\varepsilon_6) = 
	(\dspin{a}{+},
	\dspin{0}{+},
	+,
	\dspin{a{-}1}{+},
	\dspin{0}{+},
	+)$.

\begin{table}[H]
\begin{tabular}[t]{@{}ccc@{}}
	\toprule 
	$\dspin{a}{+}$
	& $\dspin{0}{+}$
	&  $z_{1}^n - v z_{2}^n$ \\
	\bottomrule
\end{tabular}\qquad%
\begin{tabular}[t]{@{}ccc@{}}
	\toprule 
	$\dspin{a{-}1}{+}$
	& $\dspin{0}{+}$
	& $z_{1}^n - v z_{2}^n$ \\
	\bottomrule
\end{tabular}
\end{table}

\noindent This is also true when $a = 0$.

\noindent\textbf{Case 2:} $(\varepsilon_1, \varepsilon_2, \varepsilon_3, \varepsilon_4, \varepsilon_5, \varepsilon_6) = (+,+,+,-,-,+)$. No admissible states exist.

\noindent\textbf{Case 3a:} $(
	\dspin{c_1}{\varepsilon_1}, 
	\dspin{c_2}{\varepsilon_2}, 
	\varepsilon_3,
	\dspin{c_4}{\varepsilon_4}, 
	\dspin{c_5}{\varepsilon_5}, 
	\varepsilon_6) = 
	(\dspin{0}{+},
	\dspin{1}{-},
	+,
	\dspin{0}{-},
	\dspin{0}{+},
	+)$.

\begin{table}[H]
\begin{tabular}[t]{@{}ccc@{}}
	\toprule 
	$\dspin{0}{-}$
	& $\dspin{0}{+}$
	&  $(1-v) z_{1} z_2^n$ \\
	\bottomrule
\end{tabular}\qquad%
\begin{tabular}[t]{@{}ccc@{}}
	\toprule 
	$\dspin{n{-}1}{+}$
	& $\dspin{2}{-}$
	&  $(1-v) z_{1} z_2^n$ \\
	\bottomrule
\end{tabular}
\end{table}

\noindent\textbf{Case 3b:} $(
	\dspin{c_1}{\varepsilon_1}, 
	\dspin{c_2}{\varepsilon_2}, 
	\varepsilon_3,
	\dspin{c_4}{\varepsilon_4}, 
	\dspin{c_5}{\varepsilon_5}, 
	\varepsilon_6) = 
	(\dspin{1}{+},
	\dspin{0}{-},
	+,
	\dspin{0}{-},
	\dspin{0}{+},
	+)$.

\begin{table}[H]
\begin{tabular}[t]{@{}ccc@{}}
	\toprule 
	$\dspin{0}{-}$
	& $\dspin{0}{+}$
	&  $(1-v) z_{1}^n z_2$ \\
	\bottomrule
\end{tabular}\qquad%
\begin{tabular}[t]{@{}ccc@{}}
	\toprule 
	$\dspin{0}{-}$
	& $\dspin{0}{+}$
	&  $(1-v) (z_{1}^n - z_{2}^n) z_2$ \\
	\midrule
	$\dspin{0}{+}$
	& $\dspin{1}{-}$
	&  $(1-v) z_{2}^{n+1}$ \\
	\bottomrule
\end{tabular}
\end{table}

\noindent\textbf{Case 3c:} $(
	\dspin{c_1}{\varepsilon_1}, 
	\dspin{c_2}{\varepsilon_2}, 
	\varepsilon_3,
	\dspin{c_4}{\varepsilon_4}, 
	\dspin{c_5}{\varepsilon_5}, 
	\varepsilon_6) = 
	(\dspin{a}{+},
	\dspin{b}{-},
	+,
	\dspin{0}{-},
	\dspin{0}{+},
	+)$, where $a + b \equiv 1 \pod{n}$.

\begin{table}[H]
\begin{tabular}[t]{@{}ccc@{}}
	\toprule 
	$\dspin{0}{-}$
	& $\dspin{0}{+}$
	&  $(1-v) z_{1}^{n-a+1} z_{2}^{a}$ \\
	\bottomrule
\end{tabular}\qquad%
\begin{tabular}[t]{@{}ccc@{}}
	\toprule 
	$\dspin{a{-}1}{+}$
	& $\dspin{b{+}1}{-}$
	&  $(1-v) z_{1}^{n-a+1} z_{2}^{a}$ \\
	\bottomrule
\end{tabular}
\end{table}

\noindent\textbf{Case 4a:} $(
	\dspin{c_1}{\varepsilon_1}, 
	\dspin{c_2}{\varepsilon_2}, 
	\varepsilon_3,
	\dspin{c_4}{\varepsilon_4}, 
	\dspin{c_5}{\varepsilon_5}, 
	\varepsilon_6) = 
	(\dspin{a}{+},
	\dspin{b}{-},
	+,
	\dspin{a{-}1}{+},
	\dspin{b{+}1}{-},
	+)$, where $a + b \not \equiv 1 \pod{n}$.

\begin{table}[H]
\begin{tabular}[t]{@{}ccc@{}}
	\toprule 
	$\dspin{a}{+}$
	& $\dspin{b}{-}$
	& $g(a + b - 1) (z_{1}^n - v z_{2}^n) z_2$ \\
	\bottomrule
\end{tabular}\qquad%
\begin{tabular}[t]{@{}ccc@{}}
	\toprule 
	$\dspin{a{-}1}{+}$
	& $\dspin{b{+}1}{-}$
	& $g(a + b - 1) (z_{1}^n - v z_{2}^n) z_2$ \\
	\bottomrule
\end{tabular}
\end{table}

\noindent This is also true when $a=0$ or $b=0$.

\noindent\textbf{Case 4b:} $(
	\dspin{c_1}{\varepsilon_1}, 
	\dspin{c_2}{\varepsilon_2}, 
	\varepsilon_3,
	\dspin{c_4}{\varepsilon_4}, 
	\dspin{c_5}{\varepsilon_5}, 
	\varepsilon_6) = 
	(\dspin{a}{+},
	\dspin{a}{-},
	+,
	\dspin{a{-}1}{+},
	\dspin{a{+}1}{-},
	+)$.

\begin{table}[H]
\begin{tabular}[t]{@{}ccc@{}}
	\toprule 
	$\dspin{a}{+}$
	& $\dspin{a}{-}$
	& \footnotesize$\begin{cases} (v^2 z_{2}^n - z_{1}^n) z_2 &\text{if $2a \equiv 1 \pod{n}$} \\ g(2a-1) (z_{1}^n - v z_{2}^n) z_2 &\text{if $2a \not \equiv 1 \pod{n}$} \end{cases}$ \\
	\bottomrule
\end{tabular}\qquad%
\begin{tabular}[t]{@{}ccc@{}}
	\toprule 
	$\dspin{a{-}1}{+}$
	& $\dspin{a{+}1}{-}$
	& \footnotesize$\begin{cases} (v^2 z_{2}^n - z_{1}^n) z_2 &\text{if $2a \equiv 1 \pod{n}$} \\ g(2a-1) (z_{1}^n - v z_{2}^n) z_2 &\text{if $2a \not \equiv 1 \pod{n}$} \end{cases}$ \\
	\bottomrule
\end{tabular}
\end{table}

\noindent This is also true when $a = 0$.

\noindent\textbf{Case 4c:} $(
	\dspin{c_1}{\varepsilon_1}, 
	\dspin{c_2}{\varepsilon_2}, 
	\varepsilon_3,
	\dspin{c_4}{\varepsilon_4}, 
	\dspin{c_5}{\varepsilon_5}, 
	\varepsilon_6) = 
	(\dspin{a}{+},
	\dspin{b}{-},
	+,
	\dspin{0}{+},
	\dspin{1}{-},
	+)$, where $a + b \equiv 1 \pod{n}$ and $a \neq 1$.

\begin{table}[H]
\begin{tabular}[t]{@{}ccc@{}}
	\toprule 
	$\dspin{0}{-}$
	& $\dspin{0}{+}$
	& $(1-v)^2 z_{1}^{b} z_2^a$ \\
	\midrule 
	$\dspin{1}{+}$
	& $\dspin{0}{-}$
	& $-(1-v) z_{1}^{b} z_{2}^{a}$ \\
	\bottomrule
\end{tabular}\qquad%
\begin{tabular}[t]{@{}ccc@{}}
	\toprule 
	$\dspin{a{-}1}{+}$
	& $\dspin{b{+}1}{-}$
	& $-v (1-v) z_{1}^{b} z_{2}^{a}$ \\
	\bottomrule
\end{tabular}
\end{table}

\noindent This is also true when $a = 0$, in which case we replace $a$ by $n$.

\noindent\textbf{Case 4d:} $(
	\dspin{c_1}{\varepsilon_1}, 
	\dspin{c_2}{\varepsilon_2}, 
	\varepsilon_3,
	\dspin{c_4}{\varepsilon_4}, 
	\dspin{c_5}{\varepsilon_5}, 
	\varepsilon_6) = 
	(\dspin{1}{+},
	\dspin{0}{-},
	+,
	\dspin{0}{+},
	\dspin{1}{-},
	+)$.

\begin{table}[H]
\begin{tabular}[t]{@{}ccc@{}}
	\toprule 
	$\dspin{0}{-}$
	& $\dspin{0}{+}$
	& $(1-v)^2 z_{1}^n z_2$ \\
	\midrule 
	$\dspin{1}{+}$
	& $\dspin{0}{-}$
	& $(v^2 z_{2}^n - z_{1}^n) z_2$ \\
	\bottomrule
\end{tabular}\qquad%
\begin{tabular}[t]{@{}ccc@{}}
	\toprule 
	$\dspin{0}{-}$
	& $\dspin{0}{+}$
	& $(1-v)^2 z_{1}^n z_2$ \\
	\midrule 
	$\dspin{0}{+}$
	& $\dspin{1}{-}$
	& $(v^2 z_{2}^n - z_{1}^n) z_2$ \\
	\bottomrule
\end{tabular}
\end{table}

\noindent\textbf{Case 4e:} $(
	\dspin{c_1}{\varepsilon_1}, 
	\dspin{c_2}{\varepsilon_2}, 
	\varepsilon_3,
	\dspin{c_4}{\varepsilon_4}, 
	\dspin{c_5}{\varepsilon_5}, 
	\varepsilon_6) = 
	(\dspin{a}{+},
	\dspin{b}{-},
	+,
	\dspin{a{-}1}{+},
	\dspin{b{+}1}{-},
	+)$, where $a + b \equiv 1 \pod{n}$.

\begin{table}[H]
\begin{tabular}[t]{@{}ccc@{}}
	\toprule 
	$\dspin{a}{+}$
	& $\dspin{b}{-}$
	& $(v^2 z_{2}^n - z_{1}^n) z_2$ \\
	\bottomrule
\end{tabular}\qquad%
\begin{tabular}[t]{@{}ccc@{}}
	\toprule 
	$\dspin{a{-}1}{+}$
	& $\dspin{b{+}1}{-}$
	& $(v^2 z_{2}^n - z_{1}^n) z_2$ \\
	\bottomrule
\end{tabular}
\end{table}

\noindent This is also true when $a=0$ or $b=0$.

\noindent\textbf{Case 4f:} $(
	\dspin{c_1}{\varepsilon_1}, 
	\dspin{c_2}{\varepsilon_2}, 
	\varepsilon_3,
	\dspin{c_4}{\varepsilon_4}, 
	\dspin{c_5}{\varepsilon_5}, 
	\varepsilon_6) = 
	(\dspin{a}{+},
	\dspin{b}{-},
	+,
	\dspin{c{-}1}{+},
	\dspin{d{+}1}{-},
	+)$, where $a + b \equiv c + d \equiv 1 \pod{n}$, $a \neq  c$, and $c \notin \{0,1\}$.

\begin{table}[H]
\begin{tabular}[t]{@{}ccc@{}}
	\toprule 
	$\dspin{c}{+}$
	& $\dspin{d}{-}$
	& $-(1-v) z_{1}^{n-a+c} z_{2}^{a-c+1}$ \\
	\bottomrule
\end{tabular}\qquad%
\begin{tabular}[t]{@{}ccc@{}}
	\toprule 
	$\dspin{a{-}1}{+}$
	& $\dspin{b{+}1}{-}$
	& $-(1-v) z_{1}^{n-a+c} z_{2}^{a-c+1}$ \\
	\bottomrule
\end{tabular}
\end{table}

\noindent This is also true when $a = 0$, in which case we replace $a$ by $n$.

\noindent\textbf{Case 4g:} $(
	\dspin{c_1}{\varepsilon_1}, 
	\dspin{c_2}{\varepsilon_2}, 
	\varepsilon_3,
	\dspin{c_4}{\varepsilon_4}, 
	\dspin{c_5}{\varepsilon_5}, 
	\varepsilon_6) = 
	(\dspin{1}{+},
	\dspin{0}{-},
	+,
	\dspin{a{-}1}{+},
	\dspin{b{+}1}{-},
	+)$, where $a + b \equiv 1 \pod{n}$ and $a \neq 1$.

\begin{table}[H]
\begin{tabular}[t]{@{}ccc@{}}
	\toprule 
	$\dspin{a}{+}$
	& $\dspin{b}{-}$
	& $-v (1-v) z_{1}^{a-1} z_{2}^{n-a+2} $ \\
	\bottomrule
\end{tabular}\qquad%
\begin{tabular}[t]{@{}ccc@{}}
	\toprule 
	$\dspin{0}{-}$
	& $\dspin{0}{+}$
	& $(1-v)^2 z_{1}^{a-1} z_{2}^{n-a+2}$ \\
	\midrule 
	$\dspin{0}{+}$
	& $\dspin{1}{-}$
	& $-(1-v) z_{1}^{a-1} z_{2}^{n-a+2}$ \\
	\bottomrule
\end{tabular}
\end{table}

\noindent This is also true when $a = 0$, in which case we replace $a$ by $n$.

\noindent\textbf{Case 4h:} $(
	\dspin{c_1}{\varepsilon_1}, 
	\dspin{c_2}{\varepsilon_2}, 
	\varepsilon_3,
	\dspin{c_4}{\varepsilon_4}, 
	\dspin{c_5}{\varepsilon_5}, 
	\varepsilon_6) = 
	(\dspin{a}{+},
	\dspin{b}{-},
	+,
	\dspin{c{-}1}{+},
	\dspin{d{+}1}{-},
	+)$, where $a+b \equiv c + d \equiv 1 \pod{n}$, $a \neq c$, $a \notin \{0,1\}$, and $c \neq 1$.

\begin{table}[H]
\begin{tabular}[t]{@{}ccc@{}}
	\toprule 
	$\dspin{c}{+}$
	& $\dspin{d}{-}$
	& $- (1-v) z_{1}^{c-a} z_{2}^{a-c+1} \begin{cases} z_{1}^n &\text{if $a > c$} \\ v z_{2}^n &\text{if $a < c$} \end{cases}$ \\
	\bottomrule
\end{tabular}\qquad%
\begin{tabular}[t]{@{}ccc@{}}
	\toprule 
	$\dspin{a{-}1}{+}$
	& $\dspin{b{+}1}{-}$
	& $- (1-v) z_{1}^{c-a} z_{2}^{a-c+1} \begin{cases} z_{1}^n &\text{if $a > c$} \\ v z_{2}^n &\text{if $a < c$} \end{cases}$ \\
	\bottomrule
\end{tabular}
\end{table}

\noindent This is also true when $c = 0$.

\noindent\textbf{Case 5:} $(\dspin{c_1}{\varepsilon_1}, 
	\dspin{c_2}{\varepsilon_2}, 
	\varepsilon_3,
	\dspin{c_4}{\varepsilon_4}, 
	\dspin{c_5}{\varepsilon_5}, 
	\varepsilon_6) = 
	(\dspin{0}{-},
	\dspin{0}{+},
	+,
	\dspin{0}{-},
	\dspin{0}{+},
	+)$.

\begin{table}[H]
\begin{tabular}[t]{@{}ccc@{}}
	\toprule 
	$\dspin{0}{-}$
	& $\dspin{0}{+}$
	&  $(z_{1}^n - z_{2}^n) z_1$ \\
	\bottomrule
\end{tabular}\qquad%
\begin{tabular}[t]{@{}ccc@{}}
	\toprule 
	$\dspin{0}{-}$
	& $\dspin{0}{+}$
	&  $(z_{1}^n - z_{2}^n) z_1$ \\
	\bottomrule
\end{tabular}
\end{table}

\noindent\textbf{Case 6a:} $(\dspin{c_1}{\varepsilon_1}, 
	\dspin{c_2}{\varepsilon_2}, 
	\varepsilon_3,
	\dspin{c_4}{\varepsilon_4}, 
	\dspin{c_5}{\varepsilon_5}, 
	\varepsilon_6) = 
	(\dspin{0}{-},
	\dspin{0}{+},
	+,
	\dspin{0}{+},
	\dspin{1}{-},
	+)$.

\begin{table}[H]
\begin{tabular}[t]{@{}ccc@{}}
	\toprule 
	$\dspin{0}{-}$
	& $\dspin{0}{+}$
	& $(1-v) (z_{1}^n - z_{2}^n) z_1$ \\
	\midrule 
	$\dspin{1}{+}$
	& $\dspin{0}{-}$
	& $(1-v) z_1 z_{2}^n$ \\
	\bottomrule
\end{tabular}\qquad%
\begin{tabular}[t]{@{}ccc@{}}
	\toprule 
	$\dspin{0}{-}$
	& $\dspin{0}{+}$
	& $(1-v) z_{1}^{n+1}$ \\
	\bottomrule
\end{tabular}
\end{table}

\noindent\textbf{Case 6b:} $(\dspin{c_1}{\varepsilon_1}, 
	\dspin{c_2}{\varepsilon_2}, 
	\varepsilon_3,
	\dspin{c_4}{\varepsilon_4}, 
	\dspin{c_5}{\varepsilon_5}, 
	\varepsilon_6) = 
	(\dspin{0}{-},
	\dspin{0}{+},
	+,
	\dspin{a}{+},
	\dspin{b}{-},
	+)$, where $a + b \equiv 1 \pod{n}$.

\begin{table}[H]
\begin{tabular}[t]{@{}ccc@{}}
	\toprule 
	$\dspin{a{+}1}{+}$
	& $\dspin{b{-}1}{-}$
	& $(1-v) z_{1}^{a+1} z_{2}^{n-a}$ \\
	\bottomrule
\end{tabular}\qquad%
\begin{tabular}[t]{@{}ccc@{}}
	\toprule 
	$\dspin{0}{-}$
	& $\dspin{0}{+}$
	& $(1-v) z_{1}^{a+1} z_{2}^{n-a}$ \\
	\bottomrule
\end{tabular}
\end{table}

\noindent\textbf{Case 7:} $(\varepsilon_1, \varepsilon_2, \varepsilon_3, \varepsilon_4, \varepsilon_5, \varepsilon_6) = (-, -, +, +, +, +)$. No admissible states exist.

\noindent\textbf{Case 8:} $(\dspin{c_1}{\varepsilon_1}, 
	\dspin{c_2}{\varepsilon_2}, 
	\varepsilon_3,
	\dspin{c_4}{\varepsilon_4}, 
	\dspin{c_5}{\varepsilon_5}, 
	\varepsilon_6) = 
	(\dspin{0}{-},
	\dspin{a}{-},
	+,
	\dspin{0}{-},
	\dspin{a{+}1}{-},
	+)$.

\begin{table}[H]
\begin{tabular}[t]{@{}ccc@{}}
	\toprule 
	$\dspin{0}{-}$
	& $\dspin{a}{-}$
	& $(z_{1}^n - v z_{2}^n ) z_1 z_2$ \\
	\bottomrule
\end{tabular}\qquad%
\begin{tabular}[t]{@{}ccc@{}}
	\toprule 
	$\dspin{0}{-}$
	& $\dspin{a{+}1}{-}$
	& $(z_{1}^n - v z_{2}^n ) z_1 z_2$ \\
	\bottomrule
\end{tabular}
\end{table}

\noindent\textbf{Case 9:} $(\dspin{c_1}{\varepsilon_1}, 
	\dspin{c_2}{\varepsilon_2}, 
	\varepsilon_3,
	\dspin{c_4}{\varepsilon_4}, 
	\dspin{c_5}{\varepsilon_5}, 
	\varepsilon_6) = 
	(\dspin{1}{+},
	\dspin{0}{+},
	-,
	\dspin{0}{-},
	\dspin{0}{+},
	+)$.

\begin{table}[H]
\begin{tabular}[t]{@{}ccc@{}}
	\toprule 
	$\dspin{1}{+}$
	& $\dspin{0}{+}$
	&  $z_{1}^n - v z_{2}^n$ \\
	\bottomrule
\end{tabular}\qquad%
\begin{tabular}[t]{@{}ccc@{}}
	\toprule 
	$\dspin{0}{-}$
	& $\dspin{0}{+}$
	& $z_{1}^n - z_{2}^n$ \\
	\midrule 
	$\dspin{0}{+}$
	& $\dspin{1}{-}$
	& $(1-v) z_{2}^n$ \\
	\bottomrule
\end{tabular}
\end{table}

\noindent\textbf{Case 10a:} $(\dspin{c_1}{\varepsilon_1}, 
	\dspin{c_2}{\varepsilon_2}, 
	\varepsilon_3,
	\dspin{c_4}{\varepsilon_4}, 
	\dspin{c_5}{\varepsilon_5}, 
	\varepsilon_6) = 
	(\dspin{1}{+},
	\dspin{0}{+},
	-,
	\dspin{0}{+},
	\dspin{1}{-},
	+)$.

\begin{table}[H]
\begin{tabular}[t]{@{}ccc@{}}
	\toprule 
	$\dspin{1}{+}$
	& $\dspin{0}{+}$
	&  $v^2 z_{2}^n - v z_{1}^n$ \\
	\bottomrule
\end{tabular}\qquad%
\begin{tabular}[t]{@{}ccc@{}}
	\toprule 
	$\dspin{0}{-}$
	& $\dspin{0}{+}$
	& $(1-v) z_{1}^n$ \\
	\midrule 
	$\dspin{0}{+}$
	& $\dspin{1}{-}$
	& $v^2 z_{2}^n - z_{1}^n$ \\
	\bottomrule
\end{tabular}
\end{table}

\noindent\textbf{Case 10b:} $(\dspin{c_1}{\varepsilon_1}, 
	\dspin{c_2}{\varepsilon_2}, 
	\varepsilon_3,
	\dspin{c_4}{\varepsilon_4}, 
	\dspin{c_5}{\varepsilon_5}, 
	\varepsilon_6) = 
	(\dspin{a}{+},
	\dspin{0}{+},
	-,
	\dspin{a{-}1}{+},
	\dspin{1}{-},
	+)$, where $a \neq 1$.

\begin{table}[H]
\begin{tabular}[t]{@{}ccc@{}}
	\toprule 
	$\dspin{a}{+}$
	& $\dspin{0}{+}$
	& $g(a-1) (z_{1}^n - v z_{2}^n)$ \\
	\bottomrule
\end{tabular}\qquad%
\begin{tabular}[t]{@{}ccc@{}}
	\toprule 
	$\dspin{a{-}1}{+}$
	& $\dspin{1}{-}$
	& $g(a-1) (z_{1}^n - v z_{2}^n)$ \\
	\bottomrule
\end{tabular}
\end{table}

\noindent\textbf{Case 11:} $(\varepsilon_1, \varepsilon_2, \varepsilon_3, \varepsilon_4, \varepsilon_5, \varepsilon_6) = (+, -, -, +, +, +)$. No admissible states exist.

\noindent\textbf{Case 12a:} $(\dspin{c_1}{\varepsilon_1}, 
	\dspin{c_2}{\varepsilon_2}, 
	\varepsilon_3,
	\dspin{c_4}{\varepsilon_4}, 
	\dspin{c_5}{\varepsilon_5}, 
	\varepsilon_6) =  
	(\dspin{1}{+},
	\dspin{0}{-},
	-,
	\dspin{0}{-},
	\dspin{1}{-},
	+)$.

\begin{table}[H]
\begin{tabular}[t]{@{}ccc@{}}
	\toprule 
	$\dspin{0}{-}$
	& $\dspin{0}{+}$
	& $(1-v) z_{1}^n z_2$ \\
	\midrule 
	$\dspin{1}{+}$
	& $\dspin{0}{-}$
	& $(v^2 z_{2}^n - z_{1}^n) z_2$ \\
	\bottomrule
\end{tabular}\qquad%
\begin{tabular}[t]{@{}ccc@{}}
	\toprule 
	$\dspin{0}{-}$
	& $\dspin{1}{-}$
	& $v (v z_{2}^n - z_{1}^n) z_2$ \\
	\bottomrule
\end{tabular}
\end{table}

\noindent\textbf{Case 12b:} $(\dspin{c_1}{\varepsilon_1}, 
	\dspin{c_2}{\varepsilon_2}, 
	\varepsilon_3,
	\dspin{c_4}{\varepsilon_4}, 
	\dspin{c_5}{\varepsilon_5}, 
	\varepsilon_6) =  
	(\dspin{1}{+},
	\dspin{a{-}1}{-},
	-,
	\dspin{0}{-},
	\dspin{a}{-},
	+)$, where $a \neq 1$.

\begin{table}[H]
\begin{tabular}[t]{@{}ccc@{}}
	\toprule 
	$\dspin{1}{+}$
	& $\dspin{a{-}1}{-}$
	& $g(a-1) (z_{1}^n - v z_{2}^n) z_2$ \\
	\bottomrule
\end{tabular}\qquad%
\begin{tabular}[t]{@{}ccc@{}}
	\toprule 
	$\dspin{0}{-}$
	& $\dspin{a}{-}$
	& $g(a-1) (z_{1}^n - v z_{2}^n) z_2$ \\
	\bottomrule
\end{tabular}
\end{table}

\noindent\textbf{Case 13:} $(\varepsilon_1, \varepsilon_2, \varepsilon_3, \varepsilon_4, \varepsilon_5, \varepsilon_6) = (-, +, -, +, +, +)$. No admissible states exist.

\noindent\textbf{Case 14:} $(\dspin{c_1}{\varepsilon_1}, 
	\dspin{c_2}{\varepsilon_2}, 
	\varepsilon_3,
	\dspin{c_4}{\varepsilon_4}, 
	\dspin{c_5}{\varepsilon_5}, 
	\varepsilon_6) = 
	(\dspin{0}{-},
	\dspin{0}{+},
	-,
	\dspin{0}{-},
	\dspin{1}{-},
	+)$.

\begin{table}[H]
\begin{tabular}[t]{@{}ccc@{}}
	\toprule 
	$\dspin{0}{-}$
	& $\dspin{0}{+}$
	&  $(z_{1}^n - z_{2}^n) z_1$ \\
	\midrule 
	$\dspin{1}{+}$
	& $\dspin{0}{-}$
	& $(1 - v) z_1 z_{2}^n$ \\
	\bottomrule
\end{tabular}\qquad%
\begin{tabular}[t]{@{}ccc@{}}
	\toprule 
	$\dspin{0}{-}$
	& $\dspin{1}{-}$
	& $(z_{1}^n - v z_{2}^n) z_1$ \\
	\bottomrule
\end{tabular}
\end{table}

\noindent\textbf{Case 15:} $(\varepsilon_1, \varepsilon_2, \varepsilon_3, \varepsilon_4, \varepsilon_5, \varepsilon_6) = (-, -, -, -, +, +)$. No admissible states exist.

\noindent\textbf{Case 16:} $(\varepsilon_1, \varepsilon_2, \varepsilon_3, \varepsilon_4, \varepsilon_5, \varepsilon_6) = (-, -, -, +, -, +)$. No admissible states exist.

\noindent\textbf{Case 17:} $(\varepsilon_1, \varepsilon_2, \varepsilon_3, \varepsilon_4, \varepsilon_5, \varepsilon_6) = (+, +, +, -, +, -)$. No admissible states exist.

\noindent\textbf{Case 18:} $(\varepsilon_1, \varepsilon_2, \varepsilon_3, \varepsilon_4, \varepsilon_5, \varepsilon_6) = (+, +, +, +, -, -)$. No admissible states exist.

\noindent\textbf{Case 19a:} $(\dspin{c_1}{\varepsilon_1}, 
	\dspin{c_2}{\varepsilon_2}, 
	\varepsilon_3,
	\dspin{c_4}{\varepsilon_4}, 
	\dspin{c_5}{\varepsilon_5}, 
	\varepsilon_6) = 
	(\dspin{1}{+},
	\dspin{0}{-},
	+,
	\dspin{0}{+},
	\dspin{0}{+},
	-)$.

\begin{table}[H]
\begin{tabular}[t]{@{}ccc@{}}
	\toprule 
	$\dspin{1}{+}$
	& $\dspin{0}{-}$
	& $(1-v) (v^2 z_{2}^n - z_{1}^n) z_2$ \\
	\midrule 
	$\dspin{0}{-}$
	& $\dspin{0}{+}$
	& $(1-v)^2 z_{1}^n z_2$ \\
	\bottomrule
\end{tabular}\qquad%
\begin{tabular}[t]{@{}ccc@{}}
	\toprule 
	$\dspin{0}{+}$
	& $\dspin{0}{+}$
	& $v (1-v) (v z_{2}^n - z_{1}^n) z_2$ \\
	\bottomrule
\end{tabular}
\end{table}

\noindent\textbf{Case 19b:} $(\dspin{c_1}{\varepsilon_1}, 
	\dspin{c_2}{\varepsilon_2}, 
	\varepsilon_3,
	\dspin{c_4}{\varepsilon_4}, 
	\dspin{c_5}{\varepsilon_5}, 
	\varepsilon_6) = 
	(\dspin{a}{+},
	\dspin{0}{-},
	+,
	\dspin{a{-}1}{+},
	\dspin{0}{+},
	-)$, where $a \neq 1$.

\begin{table}[H]
\begin{tabular}[t]{@{}ccc@{}}
	\toprule 
	$\dspin{a}{+}$
	& $\dspin{0}{-}$
	& $(1-v)g(a-1)(z_{1}^n - v z_{2}^n) z_2$ \\
	\bottomrule
\end{tabular}\qquad%
\begin{tabular}[t]{@{}ccc@{}}
	\toprule 
	$\dspin{a{-}1}{+}$
	& $\dspin{0}{+}$
	& $(1-v)g(a-1)(z_{1}^n - v z_{2}^n) z_2$ \\
	\bottomrule
\end{tabular}
\end{table}

\noindent\textbf{Case 20:} $(\varepsilon_1, \varepsilon_2, \varepsilon_3, \varepsilon_4, \varepsilon_5, \varepsilon_6) = (+, -, +, -, -, -)$. No admissible states exist.

\noindent\textbf{Case 21:} $(\dspin{c_1}{\varepsilon_1}, 
	\dspin{c_2}{\varepsilon_2}, 
	\varepsilon_3,
	\dspin{c_4}{\varepsilon_4}, 
	\dspin{c_5}{\varepsilon_5}, 
	\varepsilon_6) = 
	(\dspin{0}{-},
	\dspin{0}{+},
	+,
	\dspin{0}{+},
	\dspin{0}{+},
	-)$.

\begin{table}[H]
\begin{tabular}[t]{@{}ccc@{}}
	\toprule 
	$\dspin{0}{-}$
	& $\dspin{0}{+}$
	&  $(1-v) (z_{1}^n - z_{2}^n) z_1$ \\
	\midrule 
	$\dspin{1}{+}$
	& $\dspin{0}{-}$
	& $(1-v)^2 z_1 z_{2}^n$ \\
	\bottomrule
\end{tabular}\qquad%
\begin{tabular}[t]{@{}ccc@{}}
	\toprule 
	$\dspin{0}{+}$
	& $\dspin{0}{+}$
	& $(1-v)(z_{1}^n - v z_{2}^n) z_1$ \\
	\bottomrule
\end{tabular}
\end{table}

\noindent\textbf{Case 22:} $(\varepsilon_1, \varepsilon_2, \varepsilon_3, \varepsilon_4, \varepsilon_5, \varepsilon_6) = (-, +, +, -, -, -)$. No admissible states exist.

\noindent\textbf{Case 23:} $(\dspin{c_1}{\varepsilon_1}, 
	\dspin{c_2}{\varepsilon_2}, 
	\varepsilon_3,
	\dspin{c_4}{\varepsilon_4}, 
	\dspin{c_5}{\varepsilon_5}, 
	\varepsilon_6) = 
	(\dspin{0}{-},
	\dspin{0}{-},
	+,
	\dspin{0}{-},
	\dspin{0}{+},
	-)$.

\begin{table}[H]
\begin{tabular}[t]{@{}ccc@{}}
	\toprule 
	$\dspin{0}{-}$
	& $\dspin{0}{-}$
	&  $(1-v) (z_{1}^n - v z_{2}^n) z_1 z_2$ \\
	\bottomrule
\end{tabular}%
\end{table}

\noindent This is also true when $a = 0$.

\noindent\textbf{Case 26:} $(\varepsilon_1, \varepsilon_2, \varepsilon_3, \varepsilon_4, \varepsilon_5, \varepsilon_6) = (+, +, -, -, -, -)$. No admissible states exist.

\noindent\textbf{Case 27a:} $(\dspin{c_1}{\varepsilon_1}, 
	\dspin{c_2}{\varepsilon_2}, 
	\varepsilon_3,
	\dspin{c_4}{\varepsilon_4}, 
	\dspin{c_5}{\varepsilon_5}, 
	\varepsilon_6) = 
	(\dspin{1}{+},
	\dspin{0}{-},
	-,
	\dspin{0}{-},
	\dspin{0}{+},
	-)$.

\begin{table}[H]
%
\end{table}

\noindent Here $e \equiv a-c \pod{n}$ with $e \in [1,n-1]$, and 
\[
	x = 
	\begin{cases}
	1
	&\text{if either $ad=0$ or else $abcd \neq 0$ and $a > c$,}\\
	v
	&\text{if either \parbox[t]{\widthof{$ad=0$}}{\hfill$bc=0$} or else $abcd \neq 0$ and $a < c$.}
	\end{cases}
\]

\noindent\textbf{Case 29:} $(\dspin{c_1}{\varepsilon_1}, 
	\dspin{c_2}{\varepsilon_2}, 
	\varepsilon_3,
	\dspin{c_4}{\varepsilon_4}, 
	\dspin{c_5}{\varepsilon_5}, 
	\varepsilon_6) = 
	(\dspin{0}{-},
	\dspin{0}{+},
	-,
	\dspin{0}{-},
	\dspin{0}{+},
	-)$.

\begin{table}[H]

\end{table}

\noindent\textbf{Case 31:} $(\varepsilon_1, \varepsilon_2, \varepsilon_3, \varepsilon_4, \varepsilon_5, \varepsilon_6) = (-, -, -, +, +, -)$. No admissible states exist.

\noindent\textbf{Case 32:} $(\dspin{c_1}{\varepsilon_1}, 
	\dspin{c_2}{\varepsilon_2}, 
	\varepsilon_3,
	\dspin{c_4}{\varepsilon_4}, 
	\dspin{c_5}{\varepsilon_5}, 
	\varepsilon_6) = 
	(\dspin{0}{-},
	\dspin{a}{-},
	-,
	\dspin{0}{-},
	\dspin{a{+}1}{-},
	-)$.

\begin{table}[H]
%
\end{table}

\noindent This is also true when $a = 0$.


%
%
%
%
%
%
%
%
%
%
\end{document}